%% file: _main.tex
\documentclass[11pt]{article}

\usepackage{array, amsmath, amsfonts, amsthm, amssymb, subcaption}
\usepackage{authblk, color, soul, bm, graphicx, empheq, enumitem, bbold}
\usepackage[dvipsnames]{xcolor}
\usepackage[margin = 1.5cm, font = small, labelfont = bf, labelsep = period]{caption}
\usepackage{xspace}
\usepackage[utf8]{inputenc}
\usepackage[T1]{fontenc}
\usepackage[round]{natbib}
\usepackage[ruled]{algorithm2e}
\usepackage{multirow}
\usepackage{mathrsfs}
\usepackage{comment}
\usepackage{eurosym}
\PassOptionsToPackage{hyphens}{url}\usepackage{hyperref}
\usepackage{accents}
\usepackage{booktabs}
\usepackage{mathtools}
\usepackage{multirow}
\usepackage{nicefrac}
\usepackage{cases}
\usepackage{csquotes}
\usepackage{float}
\usepackage{xcolor}
\usepackage{bibunits}

\usepackage{fullpage}[2cm]

\usepackage{cleveref}

\usepackage[colorinlistoftodos]{todonotes}
\usepackage{comment}

\usepackage{optidef}
\usepackage{etoolbox}
\robustify{\label}

\usepackage{pgfplots, pgfplotstable, tikz-3dplot}
\pgfplotsset{compat=1.18}
\usetikzlibrary {3d, arrows.meta, external}
\tikzexternaldisable
\PassOptionsToPackage{off}{epstopdf}

\definecolor{dodgerblue}{rgb}{0.118,0.565,1}
\definecolor{MITRed}{HTML}{750014}

\theoremstyle{plain}
\newtheorem{Th}{Theorem}
\newtheorem{Prop}{Proposition}

\theoremstyle{definition}
\newtheorem{Ex}{Example}
\newtheorem{Rmk}{Remark}
\newtheorem{Ass}{Assumption}
\crefname{Ass}{Assumption}{Assumptions}
\crefname{Prop}{Proposition}{Propositions}
\crefname{lem}{Lemma}{Lemmas}

\newcommand{\eg}{\textit{e.g.}}
\newcommand{\ie}{\textit{i.e.}}

\newcommand{\set}[1]{\mathcal{#1}}

\newcommand{\ubar}[1]{\underaccent{\bar}{#1}}

\newcommand{\K}{\mathcal{K}}

\newcommand{\T}{\mathcal{T}}

\newcommand{\R}{\mathbb R}

\newcommand{\xdn}{x^\downarrow}
\newcommand{\xup}{x^\uparrow}

\DeclarePairedDelimiter{\floor}{\lfloor}{\rfloor}
\DeclarePairedDelimiter{\ceil}{\lceil}{\rceil}

\pgfmathsetmacro{\x}{2.0}
\pgfmathsetmacro{\n}{0.65}
\pgfmathsetmacro{\lbd}{2.5}
\pgfmathsetmacro{\r}{0.1}
\pgfmathsetmacro{\m}{1.1}
\pgfmathsetmacro{\dn}{1/\n - \n}
\pgfmathsetmacro{\gridopa}{0.25}

\date{}

\allowdisplaybreaks

\begin{document}
\begin{bibunit}[abbrvnat]
\begin{center}
\LARGE
    Storage Participation in Electricity Markets: \\Time Discretization through Robust Optimization
\end{center}

\begin{center}
Dirk Lauinger, Luc Coté, Andy Sun

\footnotesize
    Energy Systems Optimization Group, Massachusetts Institute of Technology \\
    \href{mailto:lauinger@mit.edu}{lauinger@mit.edu},  \href{mailto:luccote@mit.edu}{luccote@mit.edu},  \href{mailto:sunx@mit.edu}{sunx@mit.edu}  
\end{center}

Electricity storage is used for intertemporal price arbitrage and for ancillary services that balance unforeseen supply and demand fluctuations via frequency regulation. We present an  optimization model that computes bids for both arbitrage and frequency regulation and ensures that storage operators can honor their market commitments at all times for all fluctuation signals in an uncertainty set inspired by market rules. This requirement, initially expressed by an infinite number of nonconvex functional constraints, is shown to be equivalent to a finite number of deterministic constraints. The resulting formulation is a mixed-integer bilinear program that admits mixed-integer linear relaxations and restrictions. Empirical tests on European electricity markets show a negligible optimality gap between the relaxation and the restriction. The model can account for intraday trading and, with a solution time of under 5~seconds, may serve as a building block for more complex trading strategies. Such strategies become necessary as battery capacity exceeds the demand for ancillary services. In a backtest from 1 July 2020 through 30 June 2024 joint market participation more than doubles profits and almost halves energy output compared to no FCR participation.
\vspace{0.2cm}
		

\input{0_intro}
\input{1_description}
\input{2_reformulation}
\input{3_tractability}
\input{4_intraday}
\input{5_applications}

\linespread{1}
\small
\addcontentsline{toc}{section}{References (Main Text)}
\putbib[_mybib]
\end{bibunit}

\newpage

\linespread{1.5}
\normalsize
\appendix
\begin{bibunit}[abbrvnat]
\renewcommand{\theequation}{\Alph{section}\arabic{equation}}
\setcounter{equation}{0}
\renewcommand{\thefigure}{A\arabic{figure}} 
\setcounter{figure}{0}
\renewcommand{\thetable}{A\arabic{table}} 
\setcounter{table}{0}
\input{6_appendix}
\input{7_proofs}
\linespread{1}
\small
\addcontentsline{toc}{section}{References (Appendix)}
\putbib[_mybib]
\end{bibunit}
\end{document}

%% file: 0_intro.tex
\section{Introduction}
\subsection{Background: Storage deployment in electricity markets}

Until recently, large-scale electricity storage remained elusive. As cultural anthropologist Gretchen Bakke observes in \emph{The Grid: The Fraying Wires Between Americans and Our Energy Future}, grid-scale storage, essential for integrating intermittent renewable generation and enabling efficient market operations, was remarkably limited:\emph{``some artificial lakes, one compressed air plant, three molten salt towers, eight solar trough plants, and a lot of dreams about batteries''} \citeyearpar[p.~225]{GAB16}. Since then, battery technology and manufacturing improved, prices decreased \citep{MZ21}, and batteries have been saturating high-value low-volume electricity markets. For example, the deployment of battery storage for frequency containment reserves (FCR), which grid operators contract as ancillary services to balance unforeseen fluctuations in supply and demand, started at scale in~2020 in Germany. Three years later, storage capacity exceeded FCR demand, depressing prices \citep{figgener2022development}. Similar effects have been observed in the US and the UK~\citep{mastropietro2024taxonomy}. Ancillary services are appealing because remuneration is based on the capacity to produce energy on short notice, rather than on actual production~\citep{kempton2005vehicle}. Compared to arbitrage, \ie, storing energy for resale at a higher price, ancillary services require less energy output, which benefits battery health~\citep{vetter2005ageing}. However, as storage capacity  exceeds the demand for ancillary services, prices decrease \citep{deman2025day}, making it increasingly advantageous to participate in both arbitrage and ancillary services~\citep{figgener2022development, trueman2026bess}.

\subsection{Research Gap and Prior Work}

We study \emph{how storage may bid for both FCR and arbitrage}. To this end, we formulate an optimization problem that jointly determines market bids and ensures that physical constraints on storage systems are satisfied at all times for all FCR signals in an uncertainty set inspired by market rules~\citep{EU17}. This formulation is nonconvex due to charging and discharging losses and is subject to functional uncertainty in continuous-time constraints that arise from the market rules. The problem thus incorporates two storage operations challenges outlined by~\cite{parker2019electric} in their survey of the electric power industry: multi-market participation and stochasticity.

To our knowledge, there is no mathematical theory that directly addresses this problem. Most prior work adopts a discretize-then-optimize approach, wherein the continuous dynamics are first discretized and then optimized. This approach introduces systemic inaccuracies as it tends to overestimate the minimum and terminal state-of-charge (SOC), and underestimate the maximum SOC \citep[Example~1]{my_v2g_lp}. These inaccuracies may result in bids that are infeasible in practice, exposing operators to financial penalties or market exclusion.

\cite{my_v2g_as, my_v2g_lp} incorporate continuous-time constraints in FCR bidding but either restrict arbitrage to a single trading interval or to covering FCR-related energy losses. We generalize their models by allowing for unrestricted arbitrage across the full planning horizon. While this may appear to be a modest extension, it introduces substantial complexity: without arbitrage, optimal bids can be computed via linear programming~\citep[Theorem~3]{my_v2g_lp}, whereas including arbitrage leads to a mixed-integer bilinear reformulation, as we will see in Theorem~\ref{th:P}. A central modeling challenge lies in accounting for charging and discharging losses, which cause the SOC to be nonlinear in the market bids and FCR signals. Consequently, worst-case FCR signals need not align temporally with market decisions, which invalidates the common assumption of piecewise-constant signals over fixed time intervals, as we will see in Example~\ref{ex:time_discretization}.

Most prior studies assume linear storage models with the SOC being affine in the FCR signal. In this setting, the discretize-then-optimize approach is exact under a technical condition on the uncertainty set~\citep[Proposition~7]{my_v2g_lp}. This property also holds under common battery degradation models~\citep[Proposition~1]{bae2025stacking}.
The textbook approach for achieving linearity is to relax complementarity constraints that prohibit simultaneous charging and discharging. When prices are positive, this relaxation is tight, as simultaneous charging and discharging would be suboptimal due to increased energy losses~\citep[p.~84]{JAT15}. However, in the presence of negative electricity prices, which are increasingly common \citep{seel2021plentiful, biber2022negative}, or FCR bids constrained by limited headroom, \ie, the difference between storage capacity and the SOC, it may be optimal to incur energy losses to enable increased arbitrage or FCR gains. Alternatively, some works assume lossless storage to achieve linearity, though this may overstate arbitrage profitability.

In practice, \cite{anderson2017co}, \cite{kaya2024delivering}, and \cite{schindler2024planner} use the textbook approach for multistage stochastic multimarket storage optimization. \cite{schindler2024planner} consider hydro plants with separate pumps and turbines, where the relaxation is exact if the two can operate simultaneously. \cite{lohndorf2023value} and \cite{seifert2024coordinated} assume lossless storage for similar models. \cite{zhang2025joint} impose the complementarity constraints through binary variables and assume piecewise-constant FCR signals. Their approach captures nonlinear SOC dynamics but still risks underestimating the maximum SOC. In contrast, \cite{cheng2016co} account for nonlinear storage in a dynamic programming model that, given cleared market bids, decides how closely to follow FCR signals. By modeling these signals at the mandated sampling rate of 10s, they maintain an accurate SOC at the expense of a high-dimensional state space. 

In summary, prior work either (\emph{i}) assumes linear storage and optimizes on the timescale of market decisions (minutes to hours) at low computational cost, or (\emph{ii}) uses nonlinear models and optimizes on the timescale of the FCR signal (seconds) at high computational cost. Our work bridges the gap by focusing on nonlinear models on the timescale of market decisions. The goal is to capture nonlinear SOC dynamics at low computational cost.

\subsection{Research Questions and Contributions}
We will answer two research questions: ($i$) How to capture nonlinear SOC dynamics in joint arbitrage and FCR participation on the timescale of market decisions; and ($ii$) What is the effect of joint market participation on profits and energy output? In answering these questions, we make the following contributions:
\begin{enumerate}
    \item \textbf{Exact finite-dimensional reformulation}: We derive an exact finite-dimensional reformulation of a nonconvex multimarket optimization problem with continuous time constraints and functional uncertainty. The resulting mixed-integer bilinear optimization problem computes bids that are guaranteed to be feasible under a nonlinear storage model with charging and discharging losses, improving upon the feasibility guarantees of linear models. The reformulation discretizes time on the scale of market decisions, yielding smaller models than those operating at the FCR signal level. Notably, our derivation reveals that modeling continuous-time constraints requires just one additional linear constraint per trading interval. This is surprising given that continuous-linear programs are typically much harder to solve than their finite-dimensional counterparts.
    \item \textbf{Fast near-optimal reformulations}: We derive a mixed-integer relaxation and restriction of the exact bilinear model. These approximations provide near-optimal solutions in a four-year backtest. In particular, the restriction computes market bids within 5s on average when balancing FCR-induced SOC fluctuations with intraday trading.
    \item \textbf{Four-year backtest}: We backtest our model on data from 1 July 2020 through 30 June 2024. This extends prior studies, which are typically limited to a single year~\citep{cheng2016co, schindler2024planner, seifert2024coordinated} or predate~2021~\citep{anderson2017co, my_v2g_lp}. Multi-year analysis is critical given the accelerating deployment of battery storage and the evolving geopolitical and regulatory context. We find that joint participation in FCR and arbitrage more than doubles profits and almost halves energy output compared to arbitrage alone. Compared to FCR alone, joint participation increases profits by~14\% and more than doubles energy output.
\end{enumerate}

\subsection{Structure}
The paper unfolds as follows. Section~\ref{sec:description} introduces the optimization problem. Section~\ref{sec:reformulation} develops an exact mixed-integer bilinear finite-dimensional reformulation. Section~\ref{sec:tractability} presents cases in which the reformulation is tractable and derives a mixed-integer linear relaxation and restriction. Section~\ref{sec:intraday} shows how the model can accommodate multimarket arbitrage. Section~\ref{sec:app} presents numerical results for several European countries, followed by the conclusion. We describe the  intuition behind mathematical statements in the main paper and relegate all formal proofs to  Appendix~\ref{sec:proofs}.

%% file: 1_description.tex
\paragraph{Notation.} Define $\set{T} = [0,T] = \bigcup_{k=1}^K T_k$ with non-overlapping intervals $\T_k = [(k-1)\Delta t, k\Delta t)$, for $k < K$, and $\T_K = [T - \Delta t, T]$, all of length $\Delta t>0$, $T=K\Delta t$, and $K \in \mathbb{N}$. Let $\set{U}$ be a subset of the real line $\R$. Let $\set{F}(\set{T},\set{U})$
denote the space of all functions $f:\set{T} \to \set{U}$ that are piecewise constant on the intervals $\T_k$. Let $\set{R}(\set{T},\set{U})$ denote the space of all Riemann integrable functions $f:\set{T} \to \set{U}$. Let $\bm A \odot \bm B$ denote the Hadamart product, \ie, the element-wise product of two matrices $\bm A$ and $\bm B$. For any $z \in \mathbb{R}$, define $[ z ]^+ = \max \{z,0\}$ and $[z]^- = \max\{-z,0\}$. 

\section{Problem Description}\label{sec:description}
\subsection{Power Output and SOC}
Consider a storage device whose SOC and power output must be between $\ubar y \in [0,+\infty)$ and~$\bar y \in [\ubar y, +\infty)$ and between $\ubar x \in (-\infty, 0]$ and $\bar x \in [0, +\infty)$, respectively. If the power output is positive, the SOC decreases at a rate of $\frac{1}{\eta^\mathrm{d}}$ times the output, where $\eta^\mathrm{d} \in (0,1)$ is the discharging efficiency of the device. If the output is negative, the SOC increases at a rate of $\eta^\mathrm{c}$ times the magnitude of the output, where $\eta^\mathrm{c} \in (0,1)$ is the charging efficiency of the device.

Over a planning horizon $\set{T} = [0,T]$ of length~$T$ comprised of $K$ trading intervals $\T_k = [(k-1)\Delta t, k\Delta t)$ with length $\Delta t$ and $k = 1,\ldots,K$, the device is used for arbitrage and ancillary services, specifically up- and downregulation. Before the beginning of the planning horizon, the storage operator decides on how much power~$x^0 \in \set{F}(\set{T}, \mathbb{R})$, $\xup \in \set{F}(\set{T}, \mathbb{R}_+)$, and $\xdn \in \set{F}(\set{T}, \mathbb{R}_+)$ to sell for arbitrage, upregulation, and downregulation, respectively. The actual power output depends on the regulation signal~$\xi \in \set{R}(\set{T}, [-1,1])$, which is observed as it unfolds over time. Formally, the power output at any time $t \in \set{T}$ is given by $x : \mathbb{R} \times \mathbb{R}^2_+ \times \mathbb{R} \to \mathbb{R}$,
\begin{equation}
    \label{eq:x}
        x(x^0(t), x^\uparrow(t), x^\downarrow(t), \xi(t)) 
        = x^0(t) + \left[\xi(t)\right]^+ x^\uparrow(t) - \left[\xi(t)\right]^- x^\downarrow(t),
\end{equation}
where the regulation signal $\xi(t)$ determines the proportion of the planned up- and downregulation that is actually produced. Due to complementarity, $[\xi(t)]^+\cdot[\xi(t)]^-=0$, up- and downregulation are never produced simultaneously.

The SOC is given by $y : \set{F}(\set{T}, \mathbb{R}) \times \set{F}(\set{T}, \mathbb{R}_+)^2 \times \set{R}(\set{T}, [-1,1]) \times \mathbb{R}_+ \times \T \to \mathbb{R}$,
\begin{equation}
\label{eq:y}
\begin{aligned}
    y( x^0, x^\uparrow, x^\downarrow, \xi, y_0, t )
    =
    y_0 + \int_0^t &
    \eta^\mathrm{c} \left[ x(x^0(\tau), x^\uparrow(\tau), x^\downarrow(\tau), \xi(\tau))  \right]^- \\
    &- \frac{1}{\eta^\mathrm{d}} \left[ x(x^0(\tau), x^\uparrow(\tau), x^\downarrow(\tau), \xi(\tau))  \right]^+ 
    \, \mathrm{d}\tau
\end{aligned}
\end{equation}
where $y_0$ is the SOC at time $t = 0$. Note that $y$ is a \emph{functional} of $x_0, \xup, \xdn, \xi$, which are themselves functions in the sets $\set{F}(\set{T},\R)$ and $\set{R}(\set{T}, [-1,1])$. As $\set{F}(\set{T},\R)$ is endowed with a real vector space structure, we can discuss convexity of $y(x_0, \xup, \xdn, \xi, y_0, t)$ in $x_0, \xup, \xdn$ for any fixed $t\in \T$.

The following propositions characterize the power output and SOC functions.

\begin{Prop}\label{prop:x}
    The power output function $x$ is affine increasing in~$x^0(t)$, affine nondecreasing in~$x^\uparrow(t)$, affine nonincreasing in~$x^\downarrow(t)$, and nondecreasing in~$\xi(t)$ for any $t \in \T$.
\end{Prop}

\begin{Prop}\label{prop:y}
    The SOC function $y(x_0, \xup, \xdn, \xi, y_0,t)$ is concave decreasing in~$x^0$, concave nonincreasing in~$\xup$, concave nondecreasing in~$\xdn$, and nonincreasing in~$\xi$ for any $t\in\T$.
\end{Prop}

Proposition~\ref{prop:x} follows directly from $x^\uparrow$ and $x^\downarrow$ being nonnegative. For Proposition \ref{prop:y}, the concavity properties of $y$ hold because only a fraction $0 \leq \eta^\mathrm{c} \leq 1$ of a negative power output enters the device, while a multiple $\frac{1}{\eta^\mathrm{d}} \geq 1$ of a positive output leaves the device (see Figure \ref{fig:x_vs_ydot}). The monotonicity properties in Proposition~\ref{prop:y} hold because the SOC is decreasing in power output. 

We define \emph{specific loss} as $\Delta \eta = 1/\eta^\mathrm{d} - \eta^\mathrm{c}$, which is the difference between the change in the SOC after discharging and charging a unit amount of energy. In the absence of charging and discharging losses (\ie, $\eta^\mathrm{d}=\eta^\mathrm{c}=1$), $\Delta \eta$ vanishes. As losses increase, so does $\Delta \eta$, see Figure~\ref{fig:dn_err}. The power flow into storage $\eta^\mathrm{c} x^- - x^+/\eta^\mathrm{d}$ is thus a concave piecewise linear decreasing function in the power output~$x$ as depicted in Figure~\ref{fig:x_vs_ydot}. By this property, the SOC function can be written as
\begin{equation}
\label{eq:ymin}
    \begin{aligned}
      & y(x_0, \xup, \xdn, \xi, y_0,t) \\
= & y_0 + \int_0^t \min\left\{
    -\eta^\mathrm{c} x\left(x^0(\tau), x^\uparrow(\tau), x^\downarrow(\tau), \xi(\tau)\right),
    - \frac{x\left(x^0(\tau), x^\uparrow(\tau), x^\downarrow(\tau), \xi(\tau)\right)}{\eta^\mathrm{d}} \right\}
    \, \mathrm{d}\tau.
    \end{aligned}
\end{equation}

\begin{figure}[t]
\centering
\begin{subfigure}{0.55\textwidth}
\centering
    \begin{tikzpicture}[font=\scriptsize]
        \begin{axis}[
            height=4cm,
            width=9cm,
            ymin = 0, ymax = 0.7,
            ytick distance = 0.1,
            xmin = 0.5, xmax = 1,
            xtick distance = 0.05,
            yticklabel style={
            /pgf/number format/fixed,
            /pgf/number format/precision=1
            },
            xlabel={Roundtrip efficiency $\eta^\mathrm{c} \eta^\mathrm{d}$ (-)},
            ylabel={Specific loss $\Delta \eta$ (-)},
            grid=both, 
            legend style={
            cells={anchor=west},
            }, 
        ]
        \addplot+[no marks, opacity=1] table [x=nsq, y=dn, col sep=space] {data/dn.txt};
        \end{axis}
    \end{tikzpicture}
    \caption{Specific loss for $\eta^\mathrm{c} = \eta^\mathrm{d}$.}
    \label{fig:dn_err}
\end{subfigure}
\hfill
\begin{subfigure}{0.35\textwidth}
\centering
\resizebox{!}{3.75cm}{
    \begin{tikzpicture}[font=\small]
        \draw[-Stealth, opacity = 0.25] (0, -3) -- (0, 1) node[above, text opacity = 1] {Power to storage};
        \draw[-Stealth, opacity = 0.25] (-1.5, 0) -- (1.5, 0) node[right, text opacity = 1] {Power output $x$};
        \draw[-, blue] (-1.5, 0.75) -- (0, 0) --(1.5, -3) node[right, text opacity = 1, black] {with losses};
        \draw[-, opacity = 1] (-0.75, 0.75) -- (1.5, -1.5) node[right, text opacity = 1] {without losses};
        \fill[blue, opacity = 0.1] (-1.5, 0.75) -- (-0.375, 0.75) -- (0, 0) -- cycle;
        \fill[blue, opacity = 0.1] (1.5, -3) -- (1.5, -0.75) -- (0, 0) -- cycle;
        \draw[dashed, opacity = 0.25] (-1.25, 0.625) -- (-1.25, -0.1) node[below, text opacity = 1] {$\ubar x$};
        \draw[dashed, opacity = 0.25] (-.25, -0.075) node[below, text opacity = 1] {$0$};
        \draw[dashed, opacity = 0.25] (1.25, -2.5) -- (1.25, 0.1) node[above, text opacity = 1] {$\bar x$};
        \draw[dashed, opacity = 0.25] (1.25, -2.5) -- (-0.1, -2.5) node[left, text opacity = 1] {$- \frac{\bar x}{\eta^\mathrm{d}}$};
        \draw[dashed, opacity = 0.25] (-1.25, 0.625) -- (0.1, 0.625) node[right, text opacity = 1] {$-\eta^\mathrm{c} \ubar x$};
    \end{tikzpicture}
    }
    \caption{Power to storage vs power output.}
    \label{fig:x_vs_ydot}
\end{subfigure}
\caption{Specific loss and power to storage vs power output.}
\label{fig:dn_y_vs_x}
\end{figure}
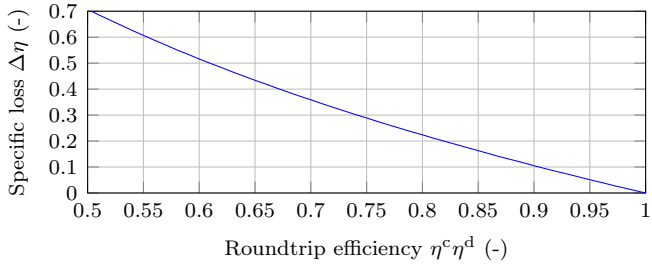
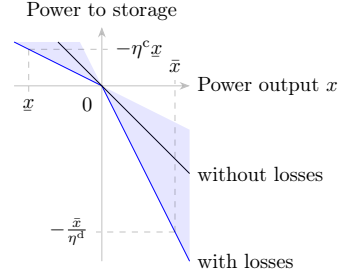

\subsection{Functional Uncertainty Set for Regulation Signals}
The storage operator must be able to provide arbitrage and regulation power for all regulation signals with 1-norm no greater than a deviation time budget~$\gamma \in [0,T]$, \ie, for all signals in the functional uncertainty set
\begin{equation}\label{eq:Xi}
    \Xi = \left\{ \xi \in \set{R}(\set{T}, [-1,1]): \int_{\set{T}} \vert \xi(t) \vert \, \mathrm{d}t \leq \gamma \right\},
\end{equation}
which is inspired by applicable EU~market rules~\citep[Art.~156]{EU17}. For FCR, a specific type of up- and downregulation, the rules require that reserve providers be able to provide all of the reserve power they promised for a minimum amount of time~$\gamma$ over a time horizon~$T$, which limits the time during which a signal $\xi$ may adopt an extreme value of~$-1$ or~$1$. Storage operators should thus indeed guarantee reserve production for all signals with 1-norm no greater than~$\gamma$. For a more detailed explanation of the market rules, see Section~2 in \cite{my_v2g_lp}.

Thanks to the 1-norm, the uncertainty set exhibits a symmetry property.

\begin{Prop}\label{prop:sym}
The uncertainty set $\Xi$ is symmetric such that $\pm\xi \in \Xi \iff \pm\vert \xi \vert \in \Xi$. 
\end{Prop}

The symmetry proposition suggests that it may be helpful to reason about nonnegative regulation signals only. For later use, we thus define $\Xi^+ = \Xi \cap \set{R}(\set{T}, [0,1])$. 

\subsection{Robust Storage Optimization Model and Challenges of Discretization}
The expected cost of providing arbitrage and regulation over the planning horizon is captured by a generic function $c(x^0, x^\uparrow, x^\downarrow)$ and the expected impact on future costs is captured by a generic cost-to-go function $\phi(x^0, x^\uparrow, x^\downarrow, y_0)$.  Putting everything together, the storage operator solves the following optimization problem to find physically feasible market decisions that minimize costs 
\begin{mini!}|s|[2]<b>            
    {}
    {c\left( x^0, x^\uparrow, x^\downarrow \right)  + \phi\left( x^0, x^\uparrow, x^\downarrow, y_0 \right)\label{R:obj}}
    {\label{pb:R}}
    {}
    \addConstraint{
    x^0 \in \set{F}(\set{T}, \mathbb{R}),~x^\uparrow, x^\downarrow \in \set{F}(\set{T}, \mathbb{R}_+)
    }{}{
    \label{R:vars}
    }
    \addConstraint{
    x(x^0(t), x^\uparrow(t), x^\downarrow(t), \xi(t)) \geq \ubar x,
    }{}{
    \quad \forall t \in \T,~\forall \xi \in \Xi
    \label{R:x_lower}
    }
    \addConstraint{
    x(x^0(t), x^\uparrow(t), x^\downarrow(t), \xi(t)) 
    }{
    \leq \bar x,
    }{
    \quad \forall t \in \T,~\forall \xi \in \Xi
    \label{R:x_upper}
    }
    \addConstraint{
    y(x^0, x^\uparrow, x^\downarrow, \xi, y_0, t) 
    }{
    \geq \ubar y,
    }{
    \quad \forall t \in \T,~\forall \xi \in \Xi
    \label{R:y_lower}
    }
    \addConstraint{
    y(x^0, x^\uparrow, x^\downarrow, \xi, y_0, t) 
    }{
    \leq \bar y,
    }{
    \quad \forall t \in \T,~\forall \xi \in \Xi.\label{R:y_upper}
    }
\end{mini!}
Constraints~\eqref{R:x_lower}--\eqref{R:y_upper} require that the power output $x$ and SOC $y$ stay within their respective limits for any regulation signal $\xi$ in the uncertainty set $\Xi$, during the time horizon $\T$.

\begin{Rmk}[Coupling between market bids]
    Market rules may couple bids for up- and downregulation. For example, the European FCR market requires them to be equal. We consider such coupling in the numerical case study in Section~\ref{sec:app}, but omit it from the problem formulation because it is independent of the regulation signal and does not impact the robust reformulations. \hfill $\Box$ 
\end{Rmk}

Compared to standard robust optimization~\citep{bental2002ro}, problem \eqref{pb:R} exhibits the following challenges.
\begin{enumerate}
    \item \textbf{Infinite dimensionality:} Constraints~~\eqref{R:x_lower}--\eqref{R:y_upper} differ from conventional robust constraints in two ways regarding infinite dimensionality. First, as an example, constraint~\eqref{R:y_upper} is equivalent to $\max_{\xi\in\Xi} y(x^0, \xup, \xdn, \xi, y_0, t)\le \bar y$ for all $t\in\T$, which is a set of infinitely many robust constraints indexed by the continuous time~$t$. This should be distinguished from the well-known equivalence of \emph{one} robust constraint to a set of infinitely many \emph{deterministic} constraints. Second, considering again constraint~\eqref{R:y_upper}, it can also be reformulated as $\max_{t\in \T} y(x^0, \xup, \xdn, \xi, y_0, t)\le \bar y$ for all $\xi\in\Xi$, which is \emph{one} robust constraint. However, its uncertainty $\xi$ is a function of continuous time, not a finite dimensional vector. Thus, both reformulations involve infinite dimensionality, either in the number of robust constraints or in the space of uncertainty, that differ from the usual setting of robust optimization. 
    \item \textbf{Nonconvexity:} By the structural properties of the SOC function~$y$ uncovered by Proposition~\ref{prop:y},
    the robust constraint~\eqref{R:y_lower} for the lower bound on the SOC requires minimizing~$y$ over~$\xi$, which is a concave minimization problem if~$x^0 \leq 0$~\citep[Proposition~1]{my_v2g_lp};
    the upper bound~\eqref{R:y_upper} on the SOC is nonconvex as~$y$ is concave in the market decisions.
\end{enumerate}
An intuitive approach to circumvent the dimensionality challenge is to discretize time by imposing that the regulation signal be constant over the discretization intervals and that the robust constraints need only hold at the end points of the discretization intervals. Such an approach removes infinite dimensionality in both the number of robust constraints and the $\xi$~space. This approach can be exact if the deviation budget~$\gamma$ fulfills the following technical assumption.
\begin{Ass}\label{ass:div}
    The deviation budget $\gamma$ is a positive multiple of the length of a trading interval~$\Delta t$.
\end{Ass}
\begin{Rmk} In general, the length of a trading interval $\Delta t$ is unrelated to the deviation budget $\gamma$ for admissible regulation signals. However, Assumption~\ref{ass:div} hardly restricts the generality of our model. If the assumption is invalid but both $\gamma$ and $\Delta t$ are rational, then it can be enforced by reducing $\Delta t$ to the greatest common divisor of $\gamma$ and the original $\Delta t$, and by adding linear constraints that couple the market decisions over the original trading intervals. \hfill $\Box$
\end{Rmk}

Under Assumption~\ref{ass:div}, the time discretization is exact if storage operators cannot sell power for arbitrage, \ie, if~$x^0 \leq 0$~\citep[Theorem~1]{my_v2g_lp}; or if the SOC is affine in the regulation signal, which is the case if $\eta^\mathrm{c} = \eta^\mathrm{d} = 1$ or if charging and discharging rates are modeled as separate affine functions of the regulation signal~\citep[Proposition~7]{my_v2g_lp}. In general, however, time discretization relaxes the robust constraints~\eqref{R:y_lower}--\eqref{R:y_upper} because it underestimates SOC fluctuations,
as the following example shows, 
which may lead to infeasible decisions.

\begin{Ex}[Risks of time discretization]\label{ex:time_discretization}
    Consider a toy problem with parameters $y_0 = 0$, $\eta^\mathrm{c} = \eta^\mathrm{d} = 0.85$, $\gamma = \Delta t = 1$h, $T = 2\Delta t$, and market bids $x^0(t) = 1.0$kW for $0 \leq t < \Delta t$, $0.5$kW for $\Delta t \leq t \leq T$, and $x^\downarrow(t) = 2.5$kW for $0 \leq t < \Delta t$, $3.5$kW for $\Delta t \leq t \leq T$. We compute the maximum SOC and the corresponding regulation signal at times $t = \Delta t, 1.1 \Delta t, \ldots, 2 \Delta t$ by solving the problem%
    \begin{subequations}
    \begin{align}
        & \max_{\xi \in \Xi} \, y(x^0, \xup, \xdn, \xi, y_0, t) \\
        = & y_0 +
        \max_{\xi \in \Xi^+} \int_0^t \min\left\{ -\eta^\mathrm{c} \left( x^0(\tau) - \xi(\tau) \xdn(\tau) \right), \,
        - \frac{1}{\eta^\mathrm{d}} \left( x^0(\tau) - \xi(\tau) \xdn(\tau) \right)
        \right\} \, \mathrm{d}\tau \\
        = & y_0 + \max_{\bm \xi \in [0,1]^k} \, \sum_{l = 1}^{k} \sigma_l(t) \min\left\{ \eta^\mathrm{c}\left(\xi_l \xdn_l - x^0_l\right), \, \frac{1}{\eta^\mathrm{d}}\left(\xi_l \xdn_l - x^0_l\right) \right\}
        ~ \text{s.t.} ~
        \sum_{l = 1}^{k} \sigma_l(t) \xi_l \leq \gamma,
    \end{align}
    \end{subequations}
    where $k$ is such that $t \in \T_k$, and $\sigma_l(t)$ is set to $\Delta t$ for $l < k$, and $\sigma_k(t) = t - (k - 1)\Delta t$. The first equality follows from equation~\eqref{eq:ymin}, monotonicity in~$\xi$, and the symmetry of~$\Xi$. The second equality exploits the concavity of the integrand, the structure of the uncertainty set, and the market decisions being piecewise constant. We recover a finite-dimensional, convex, piecewise-linear optimization problem, which can be reformulated as a linear program by introducing $k$ hypographical variables.
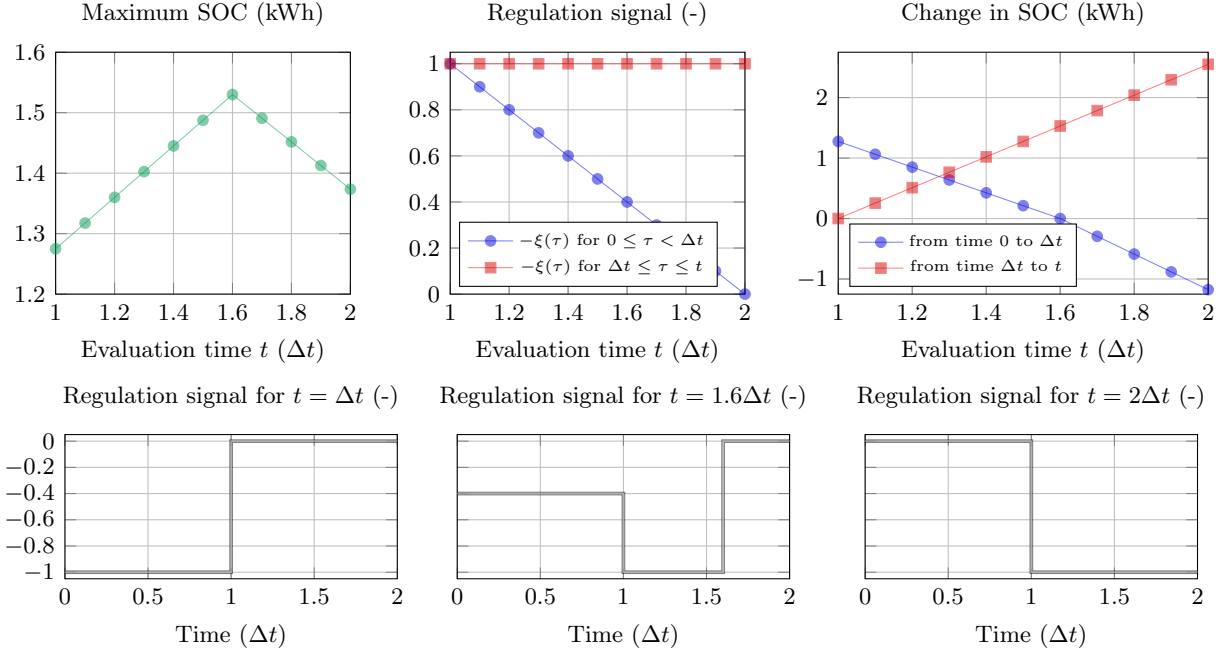
\begin{figure}[t]
\centering
\begin{subfigure}{0.31\textwidth}
\centering
\begin{tikzpicture}[]
    \begin{axis}[
    font=\footnotesize,
    width=5.5cm,
    height=4.8cm,
    enlarge x limits=false,
    ymin = 1.2, ymax = 1.6,
    ytick distance = .1,
    xlabel={Evaluation time $t$ ($\Delta t$)},
    title={Maximum SOC (kWh)},
    grid=both, 
    legend style={
    at={(0.03,0.97)},
    anchor=north west,
    },
    ]
    \addplot+[opacity=0.5, ForestGreen,  
    mark options={fill=ForestGreen,draw=ForestGreen,
    }, ] table [x=t, y=soc, col sep=space] {data/ex_ub.txt};
    \end{axis}
\end{tikzpicture}
\end{subfigure}
\begin{subfigure}{0.31\textwidth}
\centering
\begin{tikzpicture}[]
    \begin{axis}[
    font=\footnotesize,
    width=5.5cm,
    height=4.8cm,
      enlarge x limits=false,
      ymin = 0, ymax = 1.05,
      ytick distance = .2,
      xlabel={Evaluation time $t$ ($\Delta t$)},
      title={Regulation signal (-)},
      grid=both, 
      legend style={
    at={(0.03,0.03),
    }, 
    anchor=south west,
    font=\tiny,
  },
    ]      
    \addplot+[opacity=0.5] table [x=t, y=ξ1, col sep=tab] {data/ex_ub.txt};
    \addplot+[opacity=0.5] table [x=t, y=ξ2, col sep=tab] {data/ex_ub.txt};
    \legend{$-\xi(\tau)$ for $0 \leq \tau < \Delta t$,
    $-\xi(\tau)$ for $\Delta t \leq \tau \leq t$
    }
    \end{axis}
\end{tikzpicture}
\end{subfigure}
\begin{subfigure}{0.36\textwidth}
\centering
\begin{tikzpicture}
    \begin{axis}[
    font=\footnotesize,
    width=6.5cm,
    height=4.8cm,
      enlarge x limits=false,
      ymin = -1.25, ymax = 2.75,
      ytick distance = 1,
      xlabel={Evaluation time $t$ ($\Delta t$)},
      title={Change in SOC (kWh)},
      grid=both,
      legend style={
    at={(0.03,0.03)},
    anchor=south west,
    font=\tiny,
  },
    ]      
    \addplot+[opacity=0.5] table [x=t, y=Δsoc1, col sep=tab] {data/ex_ub.txt};
    \addplot+[opacity=0.5] table [x=t, y=Δsoc2, col sep=tab] {data/ex_ub.txt};
    \legend{from time $0$ to $\Delta t$, from time $\Delta t$ to $t$}
    \end{axis}
\end{tikzpicture}
\end{subfigure}
\begin{subfigure}{0.34\textwidth}
\centering
\begin{tikzpicture}[]
    \begin{axis}[
    font=\footnotesize,
      width=6cm,
      height=3.5cm,
      enlarge x limits=false,
      ymin = -1.05, ymax = 0.05,
      ytick distance = 0.2,
      xlabel={Time ($\Delta t$)},
      title={Regulation signal (-) for $t = \Delta t$},
      grid=both,
      legend style={
    at={(0.03,0.03)},
    anchor=south west,
  },
    ]      
    \addplot[opacity=0.5, double, black] table [x=t1, y=xi1, col sep=tab] {data/ex_ub_xi.txt};
    \end{axis}
\end{tikzpicture}
\end{subfigure}
\begin{subfigure}{0.32\textwidth}
\centering
\begin{tikzpicture}[]
    \begin{axis}[
    font=\footnotesize,
      width=6cm,
      height=3.5cm,
      enlarge x limits=false,
      ymin = -1.05, ymax = 0.05,
      ytick distance = 0.2,
      yticklabel=\empty,
      xlabel={Time ($\Delta t$)},
      title={Regulation signal (-) for $t = 1.6\Delta t$},
      grid=both,
      legend style={
    at={(0.03,0.03)},
    anchor=south west,
  },
    ]      
    \addplot[opacity=0.5, double, black] table [x=t2, y=xi2, col sep=tab] {data/ex_ub_xi.txt};
    \end{axis}
\end{tikzpicture}
\end{subfigure}
\begin{subfigure}{0.32\textwidth}
\centering
\begin{tikzpicture}[]
    \begin{axis}[
    font=\footnotesize,
      width=6cm,
      height=3.5cm,
      enlarge x limits=false,
      ymin = -1.05, ymax = 0.05,
      ytick distance = 0.2,
      yticklabel=\empty,
      xlabel={Time ($\Delta t$)},
      title={Regulation signal (-) for $t = 2\Delta t$},
      grid=both,
      legend style={
    at={(0.03,0.03)},
    anchor=south west,
  },
    ]      
    \addplot[opacity=0.5, double, black] table [x=t3, y=xi3, col sep=tab] {data/ex_ub_xi.txt};
    \end{axis}
\end{tikzpicture}
\end{subfigure}
\caption{Risks of time discretization.}
\label{fig:soc_overshoot}
\end{figure}
    We observe that the maximum SOC peaks at $t = 1.6\Delta t$, \ie, in the interior of a trading interval, see Figure~\ref{fig:soc_overshoot}. For any fixed~$t$, a worst-case regulation signal is given by $\xi(\tau) = -\frac{T - \Delta t}{\Delta t}$ for $0 \leq \tau < \Delta t$, $-1$ for $\Delta t \leq \tau < t$, and $0$ for $t \leq \tau \leq T$, which is \emph{non}-constant over the interval~$[\Delta t, T]$. 
\end{Ex}

Despite these challenges, it turns out that we can derive an \emph{exact finite-dimensional} reformulation of problem~\eqref{pb:R} and that the feasibility of candidate solutions can be checked by solving a \emph{linear program} of small dimensionality.

%% file: 2_reformulation.tex
\section{Finite-Dimensional Reformulation}\label{sec:reformulation}

To handle the continuous-time constraints, we work with piecewise constant functions as the market decisions $x^0$, $x^\uparrow$, $x^\downarrow$ are piecewise constant on the trading intervals $\mathcal{T}_k$, $k \in \K = \{1,\ldots,K\}$. Similar to \cite{my_v2g_lp}, we introduce a lifting operator $L_t:\mathbb{R}^K \to \set{R}(\T,\mathbb{R})$ and its adjoint $L^\dagger_t: \set{R}(\T,\mathbb{R}) \to \mathbb{R}^K$ scaled by a vector of time constants $\sigma_l(t) = \Delta t$ if $l < k$, $= t - (k - 1)\Delta t$ if $l = k$, $= 0$ otherwise, parameterized by a specific time $t \in \T_k$ for some $k \in \K$. Applying $L_t L^\dagger_t$ to a function $w \in \set{R}(\T, \mathbb{R})$ sets the function to zero on~$[t, T]$ and averages it over~$\T_1, \ldots, \T_{k-1}$, and the partial interval $[(k-1)\Delta t, t)$. Unlike in \cite{my_v2g_lp}, the averaging is limited to a fraction of~$\T_k$, not the entire interval. Conversely, applying $L^\dagger_t L_t$ to a vector $\bm v \in \mathbb{R}^K$ preserves the first $k$ elements and sets the rest to zero, see Figure~\ref{fig:LL}.  A formal definition of these operators is provided in Appendix~\ref{apx:L}.

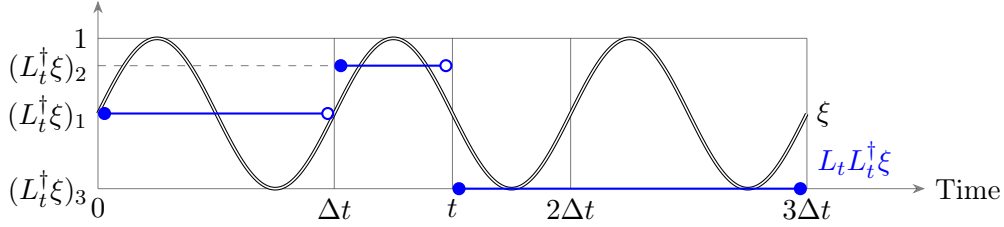
\begin{figure}[!ht]
    \centering
    \begin{tikzpicture}
     \draw[-Stealth, opacity = 0.5] (0,0) node[below, text opacity = 1] {$0$} -- (3.5*pi,0) node[right, text opacity = 1] {Time};
     \draw[-Stealth, opacity = 0.5] (0,0) -- (0,2.5);
     \draw[-, opacity = 0.5] (0,2) node[left, text opacity = 1] {$1$} -- (3*pi,2);
     \draw[-, opacity = 0.5] (pi,0) node[below, text opacity = 1] {$\Delta t$} -- (pi,2);
     \draw[-, opacity = 0.5] (2*pi,0) node[below, text opacity = 1] {$2\Delta t$} -- (2*pi,2);
    \draw[-, opacity = 0.5] (3*pi,0) node[below, text opacity = 1] {$3\Delta t$} -- (3*pi,2);
    \draw[-, opacity = 0.5] (1.5*pi,0) node[below, text opacity = 1] {$t$} -- (1.5*pi,2);
    \draw[domain=0:3*pi, double, smooth,variable=\x,black, samples = 200, opacity = 1] plot ({\x},{1+sin(2*\x r)})  node[right, text opacity = 1] {$\xi$};
    \draw[{Circle[]}-{Circle[open]},blue, thick] (0, 1) node[left, black] {$(L^\dagger_t \xi)_1$} -- (pi,1);
    \draw[{Circle[]}-{Circle[open]},blue, thick] (pi, 1 + 2/pi) -- (1.5*pi,1 + 2/pi);
    \draw[{Circle[]}-{Circle[]},blue, thick] (1.5*pi, 0) -- (3*pi,0) node[above right] {$L_t L^\dagger_t \xi$};
    \draw[dashed, opacity = 0.5] (0, 1 + 2/pi) node[left, text opacity = 1] {$(L^\dagger_t \xi)_2$} -- (pi, 1 + 2/pi);
    \node at (0,0) [left] {$(L^\dagger_t \xi)_{3}$};
    \end{tikzpicture}
    \caption{Applying the lifting and adjoint operators to a regulation signal.}
    \label{fig:LL}
\end{figure}

These operators will be useful for transforming arbitrary regulation signals into signals that are piecewise constant over all trading intervals except the~$k$-th, on which they are constant from the start of the interval up to time~$t$ and vanish thereafter. The construction ensures that we only consider regulation signals that spend their deviation budget before time~$t$. 
To ease notation, we set $\bm x^0 = L^\dagger_T x^0$, $\bm x^\uparrow = L^\dagger_T x^\uparrow$, and $\bm x^\downarrow = L^\dagger_T x^\downarrow$. Building on these operators, we introduce a discretized uncertainty set.

\begin{Prop}\label{prop:Xi_k}
    For any $t \in \T_k$ and any $k \in \K$, we have $L_t L^\dagger_t \Xi^+ \subseteq \Xi^+ $ and
    \begin{equation}\label{eq:Ldagger}
        L^\dagger_t \Xi^+ = \left\{ \bm \xi \in [0, 1]^K: \sum_{l = 1}^{k} \sigma_l(t) \xi_l \leq \gamma, \, \xi_l = 0 ~ \forall l \in \{k+1,\ldots,K\} \right\}.
    \end{equation}
\end{Prop}

We are now ready to provide finite-dimensional reformulations of the robust constraints~\eqref{R:x_lower}--\eqref{R:y_upper}, beginning with the  bounds on power output.

\begin{Prop}[Bounds on power output]\label{prop:charger}
    The following assertions hold.
    \begin{subequations}\label{reformulation:x}
    \begin{align}
        & x(x^0(t), x^\uparrow(t), x^\downarrow(t), \xi(t)) \leq \bar x,
        ~~\forall \xi \in \Xi,~\forall t \in \T \\
        \iff \quad &
        x^0_k + x^\uparrow_k \leq \bar x
        ~~\forall k \in \K,~\mathrm{ and } \label{reformulation:xlower}\\
        & x(x^0(t), x^\uparrow(t), x^\downarrow(t), \xi(t)) \geq \ubar x,
        ~~\forall \xi \in \Xi,~\forall t \in \T \\
        \iff \quad &
        x^0_k - x^\downarrow_k \geq \ubar x,
        ~~\forall k \in \K.
        \label{reformulation:xupper}
    \end{align}
    \end{subequations}
\end{Prop}

Proposition~\ref{prop:charger} holds because the power output function only depends on the value of $\xi$ at a specific time $t \in \T$ and not on the entire trajectory of $\xi$ up to time $t$. As the power output function is monotone in $\xi$, the robust constraints hold if and only if they hold for $\xi(t) = -1$ and $\xi(t) = 1$, which results in $2K$ linear constraints.

For the lower bound on the SOC, we follow Proposition~8 in \cite{my_v2g_lp}.
\begin{Prop}[Lower bound on SOC]\label{prop:soc_lb}
    The constraint
         $y(x^0, x^\uparrow, x^\downarrow, \xi, y_0, t) \geq \ubar y$
         holds for all $(t,\xi) \in \T \times \Xi$
    if and only if there exist $\bm \alpha, \bm \beta \in \mathbb{R}^K$, $\ubar{\bm{\lambda}} \in \mathbb{R}^K_+$, $\ubar{\bm \Lambda}_k \in \mathbb{R}^k_+$ for all $k \in \K$ such that
    \begin{subequations}\label{reformulation:soc_lb}
    \begin{align}
        & y_0 - \gamma \ubar \lambda_k - \Delta t \sum_{l \leq k} \alpha_l + \ubar \Lambda_{kl} \geq \ubar y,
        && \forall k \in \K, \label{reformulation:soc_lb:soc}\\
        & \alpha_{k} \geq \eta^\mathrm{c} x^0_{k},
        \, \alpha_{k} \geq \frac{x^0_{k}}{\eta^\mathrm{d}},
        \, \beta_{k} \geq \eta^\mathrm{c}\left(x^0_{k} + x^\uparrow_{k}\right),
        \, \beta_{k} \geq \frac{x^0_{k} + x^\uparrow_{k}}{\eta^\mathrm{d}},
        && \forall k \in \K, \label{reformulation:soc_lb:ldr}\\
        & \ubar \Lambda_{kl} + \ubar \lambda_{k} + \alpha_{l} - \beta_l \geq 0,
        && \forall k,l \in \K : l \leq k.
    \end{align}
    \end{subequations}
\end{Prop}

The claim holds thanks to a total unimodularity property of the nonnegative uncertainty set~$\Xi^+$ that arises from the bound~$\gamma$ on the 1-norm, which is a multiple of $\Delta t$ thanks to Assumption~\ref{ass:div}. The robust lower bound on the SOC thus results in $\frac{K(K+1)}{2}+5K$ linear constraints.

Our methodological contribution is the reformulation of the upper bound on the SOC.

\begin{Prop}[Upper bound on SOC]\label{prop:soc_ub}
    The constraint
     $y(x^0, x^\uparrow, x^\downarrow, \xi, y_0, t) \leq \bar y$ holds for all $(t,\xi) \in \T \times \Xi$
    if and only if $\bar{\bm{\lambda}} \in [0, (\bar x - \ubar x)/\eta^\mathrm{d}]^K$, $\bm \upsilon \in \{0,1\}^{2(K-1)}$, $\bar{\bm \Lambda}_k \in \mathbb{R}^k$ exist for all $k \in \K$ such that 
    \begin{subequations}\label{reformulation:soc_ub}
    \begin{align}
        & y_0 + \gamma \bar \lambda_k + \Delta t \sum_{l \leq k} \bar \Lambda_{kl} \leq \bar y,  && ~~ \forall k \in \K \label{reformulation:soc_ub:soc} \\
        & \bar \Lambda_{kk} \geq - \eta^\mathrm{c} x^0_k,  && ~~ \forall k \in \K \\
        & \bar \Lambda_{kk} \geq \eta^\mathrm{c} \left( x^\downarrow_k - x^0_k \right) - \bar \lambda_k, && ~~ \forall k \in \K \\
        & \bar \Lambda_{kk} \geq 0,
        && ~~ \forall k \in \K \\
        & (1 - \upsilon_{1k})\ubar x \leq x^0_{k} - x^\downarrow_{k} \leq \upsilon_{1k} \bar x, && ~~\forall k \in \K \setminus \{K\} \label{reformulation:soc_ub:upsilon1}\\
        & \upsilon_{2k} \ubar x \leq x^0_{k} \leq (1 - \upsilon_{2k})\bar x,
        && ~~\forall k \in \K \setminus \{K\} \label{reformulation:soc_ub:upsilon2} \\
        & \bar \Lambda_{kl} \geq \frac{x^\downarrow_{l} - x^0_{l}}{\eta^\mathrm{d}} - \bar \lambda_k + (1 - \upsilon_{1l}) \Delta \eta \, \ubar x, && ~~ \forall k,l \in \K: l < k \label{reformulation:soc_ub:Lbd1}\\
        & \bar \Lambda_{kl} \geq -\frac{x^0_{l}}{\eta^\mathrm{d}} + \upsilon_{2l} \Delta \eta \, \ubar x,
        && ~~ \forall k,l \in \K: l < k \label{reformulation:soc_ub:Lbd2} \\
        & \bar \Lambda_{kl} \geq \eta^\mathrm{c} (x^\downarrow_{l} - x^0_{l}) - \bar \lambda_k - \upsilon_{1l} \Delta \eta \, \bar x, && ~~ \forall k,l \in \K: l < k \label{reformulation:soc_ub:Lbd3} \\
        & \bar \Lambda_{kl} \geq -\eta^\mathrm{c} x^0_{l} - (1 - \upsilon_{2l}) \Delta \eta \, \bar x, 
        && ~~ \forall k,l \in \K: l < k \label{reformulation:soc_ub:Lbd4} \\
        & \bar \Lambda_{kl} x^\downarrow_{l} + \bar \lambda_k x^0_{l} \geq \upsilon_{2l} \frac{\ubar x (\bar x - \ubar x)}{\eta^\mathrm{d}} - \upsilon_{1l} \frac{\bar x^2}{4\eta^\mathrm{d}},
        && ~~ \forall k,l \in \K: l < k.\label{reformulation:soc_ub:bilinear}
    \end{align}
    \end{subequations}
\end{Prop}

The robust upper bound on the SOC is equivalent to $\frac{K(K-1)}{2}$ bilinear constraints, $(2K+4)(K-1)$ mixed-binary linear constraints, and $4K$ linear constraints, and requires $2(K-1)$ auxiliary binary variables. For fixed market decisions, the binary variables can be determined analytically, simplifying the bound to $\frac{5K(K-1)}{2} + 4K$ linear constraints. Combined with the constraints from the other bounds, the  feasibility of candidate solutions can thus be checked by solving a linear program.

\subsection{Intuition Behind the Upper Bound on the SOC}

The proof of~Proposition~\ref{prop:soc_ub} reveals that for any fixed time $t \in \T_k$ and any fixed $k \in \K$, there exists a piecewise constant regulation signal in $L^\dagger_t \Xi^+$ that maximizes the SOC at time $t$ because the SOC is concave and nonincreasing in the regulation signal. Dualizing the budget constraint in the definition of $L^\dagger_t \Xi^+$ with associated variable~$\bar \lambda_k$ yields
\begin{align}
    \max_{\xi \in \Xi}~y(x^0, x^\uparrow, x^\downarrow, \xi, y_0, t)
    = \min_{0 \leq \bar \lambda_k} \,  y_0 + \gamma \bar \lambda_k + \sum_{l = 1}^k \sigma_l(t) \varphi(x^0_l, x^\downarrow_l, \bar \lambda_k),
\end{align}
where
\begin{equation}\label{eq:varphi}
    \varphi(x^0_l, x^\downarrow_l, \bar \lambda_k)
    =
    \min_{0 \leq u_l \leq 1} \max\Big\{ \Big( \frac{1}{\eta^\mathrm{d}} - \Delta \eta \, u_l \Big) \Big( x^\downarrow_l - x^0_l \Big) - \bar \lambda_k, \, \Big( \Delta \eta \, u_l - \frac{1}{\eta^\mathrm{d}} \Big) x^0_l \Big\}.
\end{equation}
The variable $\bar \lambda_k$ can be interpreted as the marginal change in the SOC with respect to an increase in the deviation budget~$\gamma$. For any period $l =1,\ldots, k$, the optimal value function~$\varphi$, shown in Figure~\ref{fig:varphi}, measures the maximum rate of change of the SOC caused by the arbitrage decision $x^0_l$ and any surplus from spending deviation time in period~$l$. 

\subsubsection{Evaluating Feasibility of Candidate Decisions }\label{sec:evaluating_feasibility}
For candidate market decisions $\bm x^0$ and $\bm x^\downarrow$ and any fixed time $t \in \T_k$, the maximum SOC can be computed by solving the one-dimensional convex piecewise-linear problem $\min_{0 \leq \bar \lambda_k} \,  y_0 + \gamma \bar \lambda_k + \sum_{l = 1}^k \sigma_l(t) \varphi(x^0_l, x^\downarrow_l, \bar \lambda_k)$, which can be transformed into a multi-dimensional linear program by introducing $k$ epigraphical decision variables $\bar \Lambda_{kl}$ and requiring $\bar \Lambda_{kl} \geq \varphi(x^0_l, x^\downarrow_l, \bar \lambda_k)$ for $l=1,\ldots,k$. 

\begin{figure}
    \centering
    \begin{tikzpicture}[scale=0.75, transform shape]
        \draw[-Stealth] (0, -2.25) -- (0, 2.25);
        \node[above] at (1.3125, 2.25) {$\max\left\{ \frac{x^\downarrow_l - x^0_l}{\eta^\mathrm{d}} - \bar \lambda_k, -\frac{x^0_l}{\eta^\mathrm{d}} \right\}$};
        \node[below] at (1.3125, -2.25) {Case $x^\downarrow_l \leq x^0_l$};
        \draw[-Stealth] (0, 0) -- (4.25, 0) node[right] {$\bar \lambda_k$};
        \draw[-, color = blue, thick] (0, -1) -- (1, -2) -- (4.25, -2);
        \draw[-, dotted] (0, -2) node[left] {$-\frac{x^0_l}{\eta^\mathrm{d}}$} -- (1, -2) -- (1, 0) node[above] {$\frac{x^\downarrow_l}{\eta^\mathrm{d}}$};
        \node[left] at (0, -1) {$\frac{x^\downarrow_l - x^0_l}{\eta^\mathrm{d}}$};
        \draw[-Stealth] (7.5, -2.25) -- (7.5, 2.25);
        \node[above] at (8.8125, 2.25) {$\max\left\{ \eta^\mathrm{c}(x^\downarrow_l - x^0_l) - \bar \lambda_k, -\bar \lambda_k \frac{x^0_l}{x^\downarrow_l}, -\frac{x^0_l}{\eta^\mathrm{d}} \right\}$};
        \node[below] at (8.8125, -2.25) {Case $0 \le x^0_l \le x^\downarrow_l \, \land \, x^\downarrow_l > 0$};
        \draw[-Stealth] (7.5, 0) -- (11.75, 0) node[right] {$\bar \lambda_k$};
        \draw[-, color = blue, thick] (7.5, 1) -- (9, -0.5) -- (11, -1.5) -- (11.75, -1.5);
        \node[left] at (7.5, 1) {$\eta^\mathrm{c} (x^\downarrow_l - x^0_l)$};
        \draw[-, dotted] (7.5,-0.5) node[left] {$-\eta^\mathrm{c} x^0_l$} -- (9, -0.5) -- (9, 0) node[above] {$\eta^\mathrm{c} x^\downarrow_l$};
        \draw[-, dotted] (7.5,-1.5) node[left] {$-\frac{x^0_l}{\eta^\mathrm{d}}$} -- (11, -1.5) -- (11.0, 0) node[above] {$\frac{x^\downarrow_l}{\eta^\mathrm{d}}$};
        \draw[-Stealth] (15, -2.25) -- (15, 2.25);
        \node[above] at (16.3125, 2.25) {$\max\left\{ \eta^\mathrm{c}(x^\downarrow_l - x^0_l) - \bar \lambda_k, -\eta^\mathrm{c} x^0_l \right\}$};
        \node[below] at (16.3125, -2.25) {Case $x^0_l \leq 0$};
        \draw[-Stealth] (15, 0) -- (19.25, 0) node[right] {$\bar \lambda_k$};
        \draw[-, color = blue, thick] (15, 1.5) -- (16, 0.5) -- (19.25, 0.5);
        \node[left] at (15,1.5) {$\eta^\mathrm{c} (x^\downarrow_l - x^0_l)$};
        \draw[-, dotted] (15,0.5) node[left] {$-\eta^\mathrm{c} x^0_l$} -- (16, 0.5) -- (16, 0) node[below] {$\eta^\mathrm{c} x^\downarrow_l$};
    \end{tikzpicture}
    \caption{The optimal value function $\varphi(x^0_l, x^\downarrow_l, \bar \lambda_k)$ for fixed $x^0_l$ and $x^\downarrow_l$.}
    \label{fig:varphi}
    \end{figure}
To find the maximum SOC over $t \in \T_k$, we note that $\varphi(x^0_l, x^\downarrow_l, \bar \lambda_k)$ vanishes for all $\bar \lambda_k$ greater than $\max_{l=1,\ldots,k} x^\downarrow_l/\eta^\mathrm{d}$. The maximum SOC thus equals%
\begin{subequations}
\begin{align}
    & y_0 + \sup_{t \in \T_k} \, \min_{0 \leq \bar \lambda_k } \, \gamma \bar \lambda_k + \sum_{l = 1}^{k} \sigma_l(t) \varphi(x^0_l, x^\downarrow_l, \bar \lambda_k) \\
    = & y_0 + \min_{0 \leq \bar \lambda_k \leq \bar{\bar \lambda}} \, \gamma \bar \lambda_k + \Delta t \sum_{l = 1}^{k-1} \varphi(x^0_l, x^\downarrow_l, \bar \lambda_k) + \sup_{t \in \T_k} \, (t - (k-1)\Delta t) \varphi(x^0_k, x^\downarrow_k, \bar \lambda_k) \\
     = & y_0 + \min_{0 \leq \bar \lambda_k \leq \bar{\bar \lambda}} \, \gamma \bar \lambda_k + \Delta t \sum_{l = 1}^{k-1} \varphi(x^0_l, x^\downarrow_l, \bar \lambda_k) + \Delta t \, \big[\varphi(x^0_k, x^\downarrow_k, \bar \lambda_k)\big]^+,
\end{align}
\end{subequations}
where $\bar{\bar \lambda} \geq \max_{l=1,\ldots,k} x^\downarrow_l/\eta^\mathrm{d}$. The first equality follows from \cite{vonNeumman1928spiel}'s minimax theorem. The second equality holds as it is optimal to set $t = k\Delta t$ if $\varphi(x^0_k, x^\downarrow_k, \bar \lambda_k) \geq 0$ and $=(k-1)\Delta t$ otherwise. Surprisingly, finding the maximum SOC over $t \in \T_k$ can again be done by solving a one-dimensional convex piecewise-linear optimization problem. The corresponding linear program requires just one additional constraint, $\bar \Lambda_{kk} \geq 0$, compared to finding the maximum SOC at a fixed~$t$. 

\subsubsection{Optimizing over Candidate Decisions}

\begin{figure}
\centering
\begin{subfigure}{0.495\textwidth}
\includegraphics[width=\textwidth]{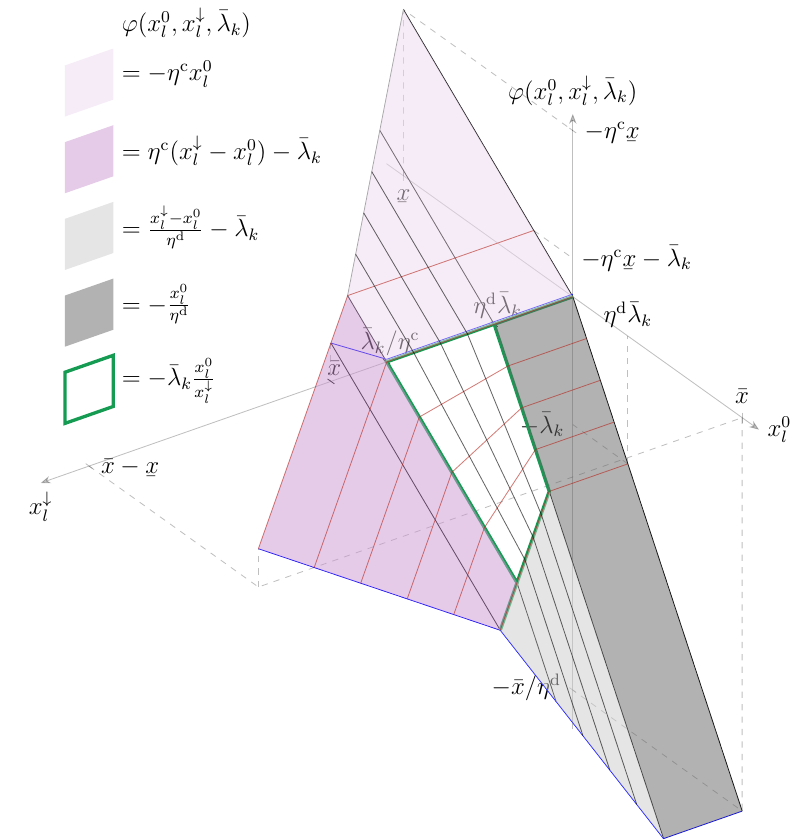}
    \caption{True function.}
    \label{fig:varphi_true}
\end{subfigure}
\hfill
\begin{subfigure}{0.495\textwidth}    \includegraphics[width=\textwidth]{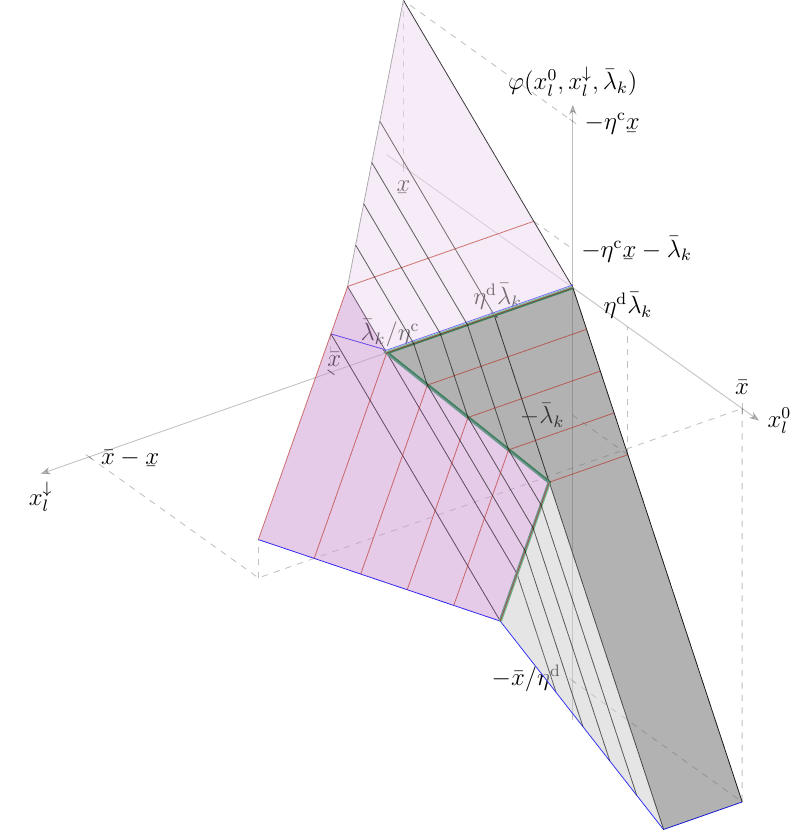}
    \caption{Lower bound}
    \label{fig:varphi_lb}
\end{subfigure}
\hfill
\begin{subfigure}{0.495\textwidth}
\includegraphics[width=\textwidth]{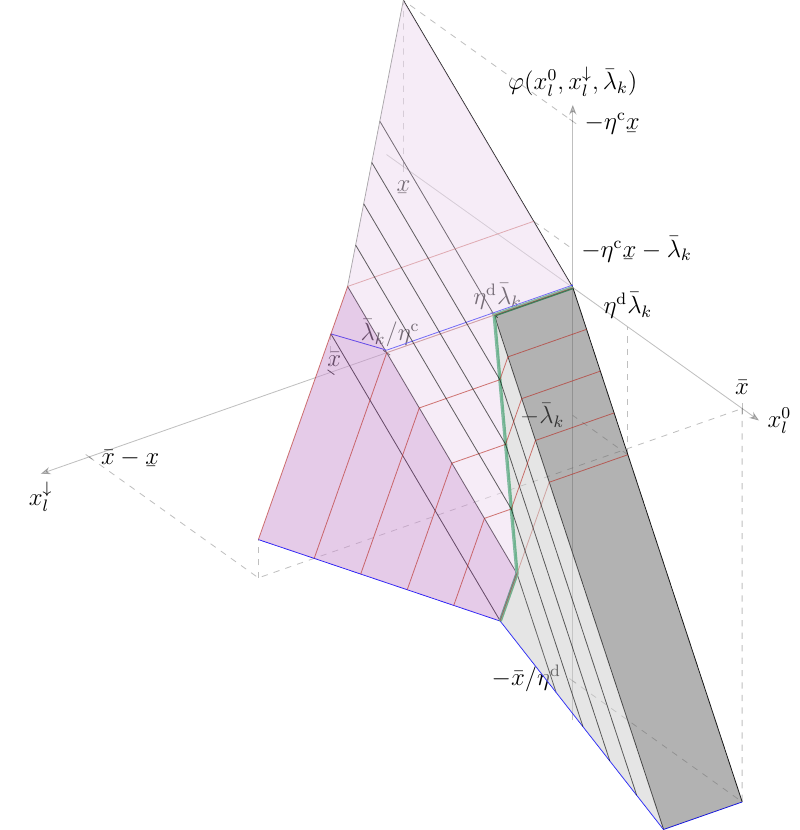}
    \caption{Upper bound.}
    \label{fig:varphi_ub}
\end{subfigure}
\caption{The optimal value function $\varphi(x^0_l, x^\downarrow_l, \bar \lambda_k)$ for fixed $\bar \lambda_k$. Contour lines are in brown, iso-$x^0_l$ lines are in blue, iso-$x^\downarrow_l$ lines are in black, the boundary between charging and discharging is in green.}
\label{fig:varphi_lbd}
\end{figure}

Figure~\ref{fig:varphi_true} depicts the $\varphi$-function for fixed $\bar \lambda_k$ and variable $x^0_l$ and $x^\downarrow_l$. As expected, the function is identical to the power-to-storage vs power output curve in Figure~\ref{fig:x_vs_ydot} if $x^\downarrow_l = 0$. As $x^\downarrow_l$ increases, the power that can be bought for arbitrage, \ie, the minimum feasible value of $x^0_l$, decreases in magnitude because downregulation takes up headroom in the lower bound on the power output. The $\varphi$-function is a saddle function in the market bids if $x^\downarrow_l \in [\eta^\mathrm{d} \bar \lambda_k, \frac{\bar \lambda_k}{\eta^\mathrm{c}}]$ and $x^0_l \in [0, x^\downarrow_l)$, and nonconvex piecewise affine otherwise. Its graph consists of two hyperplanes corresponding to charging at the arbitrage bid ($\varphi(x^0_l, x^\downarrow_l, \bar \lambda_k) = -\eta^\mathrm{c}x^0_l$) and discharging at the arbitrage bid ($\varphi(x^0_l, x^\downarrow_l, \bar \lambda_k) = -\frac{x^0_l}{\eta^\mathrm{d}}$), two hyperplanes that capture the effect of downregulation on the state-of-charge ($\varphi(x^0_l, x^\downarrow_l, \bar \lambda_k) = \eta^\mathrm{c}(x^\downarrow_l - x^0_l) - \bar \lambda_k$ and $\varphi(x^0_l, x^\downarrow_l, \bar \lambda_k) = \frac{x^\downarrow_l - x^0_l}{\eta^\mathrm{d}} - \bar \lambda_k$), and a bilinear region ($\varphi(x^0_l, x^\downarrow_l, \bar \lambda_k) = - \bar \lambda_k \frac{x^0_l}{x^\downarrow_l}$) that connects the four hyperplanes.

Incorporating the upper bound on the SOC into the storage operator's decision problem, where $\bm x^0$ and $\bm x^\downarrow$ are decision variables, is challenging because of these nonconvexities. We handle these challenges by introducing $2(K-1)$ binary variables, $(2K + 4)(K-1)$ mixed-binary linear big-$M$ constraints, and $\frac{K(K-1)}{2}$ mixed-binary bilinear big-$M$ constraints \citep[p.~67]{conforti2014integer}. To improve numerical stability, we use the lowest admissible~$M$, given by Table~\ref{tab:M} in the proof of Proposition~\ref{prop:soc_ub}, for each constraint. The table reveals that all $M$-constants are proportional to the specific loss~$\Delta \eta$, except for the $M$-constant required by the bilinear constraints, which is proportional to $\frac{1}{\eta^\mathrm{d}}$. As losses increase, so do the $M$-constants, which renders the numerical solution of the optimization problem more challenging, as one would expect. In fact, in the absence of losses, the problem reduces to a linear program, as we will establish in Proposition~\ref{prop:tractable}.

\subsection{Resulting Optimization Problem}
Having discussed Propositions \ref{prop:charger}--\ref{prop:soc_ub}, we now show that problem~\eqref{pb:R} is equivalent to the finite-dimensional mixed-binary bilinear problem
\begin{mini}|s|[2]               
    {}
    {c(\bm x^0, \bm x^\uparrow, \bm x^\downarrow) + \phi(\bm x^0, \bm x^\uparrow, \bm x^\downarrow, y_0)}
    {\tag{P}\label{pb:P}}
    {}
    \addConstraint{\bm x^0 \in \mathbb{R}^K,~\bm x^\uparrow, \bm x^\downarrow \in \mathbb{R}^K_+
    }{}{}
    \addConstraint{\bm \alpha, \bm \beta \in \mathbb{R}^K, \ubar{\bm{\lambda}}, \bar{\bm \lambda} \in \mathbb{R}^K_+, \ubar{\bm \Lambda}_k \in \mathbb{R}^k_+, \bar{\bm \Lambda}_k \in \mathbb{R}^k, \bm \upsilon \in \{0,1\}^{2(K-1)},
    }{}{\quad \forall k \in \K
    }
    \addConstraint{
    \eqref{reformulation:xlower}, 
    \eqref{reformulation:xupper},
    \eqref{reformulation:soc_lb},
    \eqref{reformulation:soc_ub}.
    }{}{}
\end{mini}
\begin{Th}\label{th:P}
    The problems~\eqref{pb:R} and~\eqref{pb:P} are equivalent. 
\end{Th}

Theorem~\ref{th:P} follows immediately from Propositions \ref{prop:charger}--\ref{prop:soc_ub}. Despite the four challenges stated at the end of Section~\ref{sec:description}, we derived an exact finite-dimensional reformulation of the original problem. The feasibility of candidate market decisions can be checked by solving the linear program to which problem \eqref{pb:P} reduces when $\bm x^0$, $\bm x^\downarrow$, and $\bm x^\uparrow$ are fixed.

It is technically possible to solve problem~\eqref{pb:P} with commercially available software. Gurobi, for example, has been relying on spatial branch-and-bound, \ie, solving a series of mixed-integer linear restrictions and relaxations, to handle bilinear constraints since version~9.0 \citep{gurobi9bilinear}. However, at the time of writing, the required computational effort did not scale well with the problem size.

%% file: 3_tractability.tex
\section{Towards Tractability}\label{sec:tractability}

We first present several special cases in which problem~\eqref{pb:P} is equivalent to more tractable problems. Next, we derive a mixed-binary linear relaxation and restriction for the general case.

\subsection{Tractable Cases}\label{sec:tractable_cases}
We establish that problem~\eqref{pb:P} reduces to a linear or mixed-integer linear program in some cases.
\begin{Prop}\label{prop:tractable}
    Problem~\eqref{pb:P} reduces to
    \begin{enumerate}
        \item a mixed-integer linear program if no downregulation is sold~$(\xdn = 0)$,
        \item a linear program if no electricity is sold for arbitrage~$(x^0 \leq 0)$, or there is only one trading interval~$(K=1)$, or there are no charging and discharging losses $(\eta^\mathrm{c} = \eta^\mathrm{d} = 1$).
    \end{enumerate}
\end{Prop}

If no power is sold for downregulation, the storage operator faces a classical arbitrage problem with $K-1$ binary variables to model the complementarity between charging and discharging. If prices are nonnegative, a continuous relaxation is optimal because any energy dissipated during simultaneous charging and discharging would lead to economic losses \citep[p.~84]{JAT15}. 

If no power is sold for arbitrage, we recover a linear program as in \cite{my_v2g_lp}. If there is only one trading interval $(K = 1)$, the resulting linear program admits an analytical solution~\citep{my_v2g_as}. In this setting, there are no price fluctuations and hence no arbitrage opportunities during the planning horizon. However, inter-horizons arbitrage opportunities may be captured by the cost-to-go function~$\phi$. In the absence of losses~$(\eta^\mathrm{c} = \eta^\mathrm{d} = 1)$, the SOC is affine in the market bids and the regulation signal, and we recover another linear program.

\subsection{Mixed-Binary Linear Relaxation and Restriction}\label{sec:MILPs}

Deleting the bilinear constraints~\eqref{reformulation:soc_ub:bilinear} yields a mixed-integer linear relaxation
\begin{mini}|s|[2]               
    {}
    {c(\bm x^0, \bm x^\uparrow, \bm x^\downarrow) + \phi(\bm x^0, \bm x^\uparrow, \bm x^\downarrow, y_0)}
    {\tag{\underline{P}}\label{pb:P:relax}}
    {}
    \addConstraint{\bm x^0 \in \mathbb{R}^K,~\bm x^\uparrow, \bm x^\downarrow \in \mathbb{R}^K_+
    }{}{}
    \addConstraint{\bm \alpha, \bm \beta \in \mathbb{R}^K, \ubar{\bm{\lambda}}, \bar{\bm \lambda} \in \mathbb{R}^K_+, \ubar{\bm \Lambda}_k \in \mathbb{R}^k_+, \bar{\bm \Lambda}_k \in \mathbb{R}^k, \bm \upsilon \in \{0,1\}^{2(K-1)},
    }{}{\quad \forall k \in \K
    }
    \addConstraint{
    \eqref{reformulation:xlower}, 
    \eqref{reformulation:xupper},
    \eqref{reformulation:soc_lb},
    \text{\eqref{reformulation:soc_ub:soc}--\eqref{reformulation:soc_ub:Lbd4}},
    }{}{}
\end{mini}
which can be used to obtain a lower bound on the optimal value of the original problem. Such a bound can be used to judge the quality of any feasible solution. A feasible solution can be found by solving the mixed-integer linear restriction obtained by introducing $(K-1)$ additional binary variables $\bm \upsilon_3 \in \{0,1\}^{(K-1)}$ and replacing the bilinear constraints~\eqref{reformulation:soc_ub:bilinear} by
\begin{subequations}
\label{reformulation:restriction}
\begin{align}
    & x^\downarrow_k - (1 - \eta^\mathrm{c} \eta^\mathrm{d})x^0_k - \eta^\mathrm{d} \bar \lambda_k \geq -((2 - \eta^\mathrm{c} \eta^\mathrm{d}) \bar x - \ubar x)(1 - \upsilon_{3k}), && \forall k \in \K \setminus \{K\}, \\
    & x^\downarrow_k - (1 - \eta^\mathrm{c} \eta^\mathrm{d})x^0_k - \eta^\mathrm{d} \bar \lambda_k \leq \upsilon_{3k}(\eta^\mathrm{c}\eta^\mathrm{d} \bar x - \ubar x), && \forall k \in \K \setminus \{K\}, \\
    & \bar \Lambda_{kl} \geq - \eta^\mathrm{c} x^0_l - (\upsilon_{1l} + 1 - \upsilon_{3l}) \Delta \eta \, \bar x, \,
    && \forall k,l \in \K: l < k, \\
    & \bar \Lambda_{kl} \geq \frac{x^\downarrow_{l} - x^0_{l}}{\eta^\mathrm{d}} - \bar \lambda_k + (\upsilon_{2l} + \upsilon_{3l}) \Delta \eta \, \ubar x, && \forall k,l \in \K: l < k,
\end{align}
\end{subequations}
resulting in the optimization problem
\begin{mini}|s|[2]               
    {}
    {c(\bm x^0, \bm x^\uparrow, \bm x^\downarrow) + \phi(\bm x^0, \bm x^\uparrow, \bm x^\downarrow, y_0)}
    {\tag{$\overline{ \text{P}}$}\label{pb:P:restrict}}
    {}
    \addConstraint{\bm x^0 \in \mathbb{R}^K,~\bm x^\uparrow, \bm x^\downarrow \in \mathbb{R}^K_+
    }{}{}
    \addConstraint{\bm \alpha, \bm \beta \in \mathbb{R}^K, \ubar{\bm{\lambda}}, \bar{\bm \lambda} \in \mathbb{R}^K_+, \ubar{\bm \Lambda}_k \in \mathbb{R}^k_+, \bar{\bm \Lambda}_k \in \mathbb{R}^k, \bm \upsilon \in \{0,1\}^{3(K-1)},
    }{}{\quad \forall k \in \K
    }
    \addConstraint{
    \eqref{reformulation:xlower}, 
    \eqref{reformulation:xupper},
    \eqref{reformulation:soc_lb},
    \text{\eqref{reformulation:soc_ub:soc}--\eqref{reformulation:soc_ub:Lbd4}},
    \eqref{reformulation:restriction}.
    }{}{}
\end{mini}
\begin{Prop}\label{prop:pres}
    Problem~\eqref{pb:P:restrict} is a restriction of problem~\eqref{pb:P}.
\end{Prop}

The bilinear constraints~\eqref{reformulation:soc_ub:bilinear} stem from the upper bound on the SOC. The mixed-binary linear relaxation and restriction  under- and overestimate the true maximum SOC, respectively, see Figures~\ref{fig:varphi_lb} and~\ref{fig:varphi_ub}. The region of the feasible set in which the bilinear constraints are binding in the original problem is generated by the constraint
$
    \bar \Lambda_{kl} \geq \max\{ 
    \eta^\mathrm{c} (x^\downarrow_{l} - x^0_{l}) - \bar \lambda_k, \,
    -\frac{x^0_{l}}{\eta^\mathrm{d}}
    \}
$
in the relaxation and by the constraint
$
     \bar \Lambda_{kl} \geq \min\{ 
     - \eta^\mathrm{c} x^0_l, \,
     \frac{x^\downarrow_{l} - x^0_{l}}{\eta^\mathrm{d}} - \bar \lambda_k
     \}
$
in the restriction. 

The problems~\eqref{pb:P:relax} and~\eqref{pb:P:restrict} can be seen as lower convex and upper concave McCormick envelopes of the original bilinear region. Solving these problems constitutes the first step in a spatial branch-and-bound approach~\citep{costa2012relaxations}. Although commercially available optimization solvers already implement spatial branch-and-bound, it is still useful to state the reformulations explicitly. We will see in Section~\ref{sec:app} that the gap between the relaxation and restriction is negligible in our case study. For this application, there is thus no need to solve the exact bilinear problem. Instead, we will always solve the restriction. We now characterize the gap between the maximum SOC estimated in the restriction and the relaxation.

\begin{Prop}\label{prop:bound}
The difference between the maximum SOC estimated by the restriction and the relaxation is no greater than $ (T - \Delta t)\Delta \eta \cdot \min\{ -\ubar x, \, \bar x, \,(\bar x - \ubar x)/(1 + \eta^\mathrm{c} \eta^\mathrm{d}) \}$.
\end{Prop}

As expected, the gap decreases with roundtrip efficiency and vanishes in the absence of losses.

\begin{Rmk}
    As the relaxation and restriction are mixed-binary linear programs, they may admit multiple different optimal solutions, which means that additional objectives may be achieved without reducing profits. For example, battery degradation may be reduced by minimizing the $\infty$-norm of the power output or of SOC deviations~\citep{thompson2018economic}. \hfill $\Box$
\end{Rmk}

%% file: 4_intraday.tex
\section{Multimarket Arbitrage} \label{sec:intraday}
So far, we have assumed that storage operators participate in a single market for arbitrage and cannot adjust their arbitrage position in response to the regulation signal. In practice, however, they may participate in multiple markets covering different timescales for arbitrage, \eg, day-ahead and real-time (in the US) or intraday (in Europe) markets. Multimarket arbitrage allows for adjustments to positions on the market closest to delivery. These adjustments may compensate SOC deviations induced by up- and downregulation. Such compensation causes an additional power flow, which tightens the bounds on charging and discharging power but relaxes the bounds on the SOC as operators may consider a smaller deviation budget~$\gamma$ when deciding on regulation power.

In the following, we focus on operators participating in day-ahead and intraday markets and examine the impact of intraday adjustments. We assume that the intraday adjustment $x^\text{a}_k$ for trading interval~$k$ is decided on at the beginning of the interval. We restrict admissible regulation signals to the exact uncertainty set specified by the \cite{EU17} for FCR,
\begin{equation}
    \Xi' = \left\{ \xi \in \set{R}(\set{T}, [-1,1]): \int_{[t - \Gamma']}^t \vert \xi(\tau) \vert \, \mathrm{d}\tau \leq \gamma' ~~ \forall t \in \T \right\},
\end{equation}
where $\gamma'$ is a parameter between $15$min and $30$min and $\Gamma'$ equals $\gamma'$ plus two hours. The activation ratio $\gamma' / \Gamma'$ can be interpreted as the ratio of time for which storage operators must be able to deliver all the regulation power they promise \citep{my_v2g_lp}. Similar to Assumption~\ref{ass:div}, we consider that $\gamma'$ is a multiple of $\Delta t$. When abstaining from intraday trading, we set
\begin{equation}\label{eq:gamma}
    \gamma = \gamma' \floor*{ \frac{T}{\Gamma'} } + \min\left\{
    \gamma', \, T - \Gamma' \floor*{\frac{T}{\Gamma'}}
    \right\} 
\end{equation}
to ensure that $\Xi' \subseteq \Xi$. Problem~\eqref{pb:R} is thus a restriction of the real decision problem faced by  storage operators. When participating in intraday trading, the following proposition shows that operators can reduce $\gamma$ to $\gamma'$ and recover feasible decisions for the case $\eta^\text{c} = \eta^\text{d} = 1$ with symmetric regulation bids $x^\text{r}_k \coloneq x^\downarrow_k = x^\uparrow_k$ for all $k \in \K$.

\begin{Prop}\label{prop:intraday}
    With $\gamma = \gamma'$, problem~\eqref{pb:P} yields feasible regulation bids under the exact EU uncertainty set if the bounds on charging and discharging power are replaced by 
    \begin{subequations}
    \begin{align}
        x^0_{k+1} - x^\mathrm{r}_{k+1} -  \frac{\gamma'}{\Delta t} \lambda_k - \sum_{i = \ubar i(k+1)}^k \left[ \frac{\Delta t}{\Gamma' - \Delta t} x^\mathrm{r}_i - \lambda_k \right]^+
        & \geq \ubar x, ~~ \forall k \in \K \setminus \{K\}, \\
        x^0_{k+1} + x^\mathrm{r}_{k+1} +  \frac{\gamma'}{\Delta t} \lambda_k + \sum_{i = \ubar i(k+1)}^k \left[ \frac{\Delta t}{\Gamma' - \Delta t} x^\mathrm{r}_i - \lambda_k \right]^+ & \leq \bar x, ~~  \forall k \in \K \setminus \{K\},
    \end{align}
    \end{subequations}
    where $\ubar i(k) = \max\{1, \, k - \frac{\Gamma'}{\Delta t} +1\}$ and $\lambda \in \mathbb{R}^K_+$ is an auxiliary decision variable, and the intraday adjustments are set to
    \begin{equation}
        x^\mathrm{a}_{k+1} = - \frac{1}{\Gamma' - \Delta t} \sum_{i = \ubar i(k+1)}^k x^\mathrm{r}_{i} \int_{(i-1)\Delta t}^{i \Delta t} \xi(\tau) \, \mathrm{d}\tau, ~~ \forall k \in \K \setminus \{K\}.
    \end{equation}
\end{Prop}

\begin{Rmk}
    For $\gamma' = \Delta t$, the tightened bounds simplify to
    \begin{subequations}
    \begin{align}
        x^0_{k+1} - x^\mathrm{r}_{k+1} - \max_{i \in \{ \ubar i(k+1), \ldots, k \}} \frac{\Delta t}{\Gamma' - \Delta t} x^\mathrm{r}_i
        \geq \ubar x, ~~ \forall k \in \K \setminus \{K\}, \\
        x^0_{k+1} + x^\mathrm{r}_{k+1} + \max_{i \in \{ \ubar i(k+1), \ldots, k \}} \frac{\Delta t}{\Gamma' - \Delta t} x^\mathrm{r}_i \leq \bar x, ~~  \forall k \in \K \setminus \{K\}. 
    \end{align}
    \end{subequations}
\end{Rmk}

Example~\ref{ex:intraday} illustrates the relationship between the tightened constraints on charging and discharging power and the activation ratio.

\begin{Ex}[Role of the activation ratio]\label{ex:intraday}
    If $x^\text{r}_k = x^\text{r}_1$ for all $k \in \K$, then for all $k \in  \K \setminus \{K\}$,
    \begin{subequations}
    \begin{align}
         x^\mathrm{a}_{k+1}
        & = - \frac{x^\text{r}_1}{\Gamma'- \Delta t} \int_{(\ubar i(k+1) - 1)\Delta t}^{k\Delta t} \xi(\tau) \, \mathrm{d}\tau \\
        \implies 
        \vert x^\mathrm{a}_{k+1} \vert 
        & =
        \frac{x^\text{r}_1}{\Gamma' - \Delta t} \left\vert \int_{(\ubar i(k+1) - 1)\Delta t}^{k\Delta t} \xi(\tau) \, \mathrm{d}\tau \right\vert \\
        & \leq
        \frac{x^\text{r}_1}{\Gamma' - \Delta t}  \int_{(\ubar i(k+1) - 1)\Delta t}^{k\Delta t} \left\vert \xi(\tau) \right\vert \, \mathrm{d}\tau \\
        & \leq 
        \frac{\gamma'}{\Gamma' - \Delta t} x^\text{r}_1.
    \end{align}
    \end{subequations}
    For $\gamma' = \Delta t$, the constraints on charging and discharging power are thus tightened by a fraction~$\frac{\gamma'}{\Gamma' - \gamma'}$ of the regulation bid. In our case study, we will set~$\gamma' = 15$min, so intraday trading tightens the power constraints~\eqref{reformulation:xlower} and~\eqref{reformulation:xupper} by~$\frac{\gamma'}{\Gamma' - \gamma'} = \frac{1}{8}$ times the regulation bids, while reducing the deviation time budget $11$-fold from $\gamma = 2.75$h to $\gamma' = 15$min, which relaxes the SOC constraints~\eqref{reformulation:soc_lb:soc} and~\eqref{reformulation:soc_ub:soc}. It thus seems likely that intraday trading enables higher FCR bids, as has been empirically observed by~\cite{zhang2025joint}. \hfill $\Box$ 
\end{Ex}

%% file: 5_applications.tex
\section{Applications}\label{sec:app}
\subsection{Market Setup}
We consider a battery storage operator located in France who participates on the European day-ahead market for electricity and on the European FCR market. The day-ahead market is operated by \href{https://www.europex.org/members/epex-spot/}{EPEX SPOT} and spans 13~European countries. The FCR market is operated by the \href{https://www.regelleistung.net/en-us/European-cooperations/FCR-Cooperation}{FCR Cooperation} and spans Austria, Belgium, the Czech Republic, Denmark, France, Germany, the Netherlands, Slovenia, and Switzerland. The regulation signal~$\xi(t)$ is given by a clipped ramp function of the grid frequency~$f(t)$ at time~$t$,
\begin{equation}
    \xi(t) = \begin{cases}
        +1 & \mathrm{if } f_0 - f(t) \geq \Delta f, \\
        -1 & \mathrm{if } f(t) - f_0 \geq \Delta f, \\
        \frac{f_0 - f(t)}{\Delta t} & \mathrm{otherwise},
    \end{cases}
\end{equation}
where $f_0 = 50$Hz is the nominal frequency and $\Delta f = 200$mHz is a threshold beyond which all promised regulation power must be delivered. The \citet[Art.~154]{EU17} requires that the frequency be measured at least every $10$s to compute the regulation signal. Participating on these markets requires ($i$) symmetric bids for up- and downregulation, \ie, $x^\mathrm{r}(t) \coloneq x^\uparrow(t) = x^\downarrow(t)$ for all $t \in \T$, ($ii$) regulation power bids in 4-hour blocks, \ie, $x^\mathrm{r}(t) = x^\mathrm{r}(\floor{t/4\text{h}} \cdot 4\text{h})$ for all $t \in \T$, and ($iii$) day-ahead power bids in 1-hour blocks, \ie, $x^0(t) = x^0(\floor{t/1\text{h}} \cdot 1\text{h})$ for all $t \in \T$.

The cost incurred over a planning horizon is given by the negative of arbitrage and FCR profits
\begin{equation}
    c'\big(x^0, x^r, \xi \big) = -\int_\T p^0(t) x^0(t) + \big(p^\mathrm{a}(t) + \xi(t) p^0(t) \big) x^\mathrm{r}(t) \, \mathrm{d}t,
\end{equation}
where $p^0(t)$ is the day-ahead price at time~$t$ and $p^\mathrm{a}(t)$ is the availability price for FCR at time~$t$. The availability payment $p^\mathrm{a}(t) x^\text{r}(t)$ is independent of the regulation signal. The regulation power~$\xi(t) x^\mathrm{r}(t)$ provided at time $t$ is valued at the day-ahead price in France. If $\xi(t)$ is positive, the battery operator provides power for regulation and receives payment. If $\xi(t)$ is negative, the battery operator consumes power for regulation and incurs costs. The expected cost over a planning horizon is
\begin{equation}
    c \big(x^0, x^\mathrm{r}, x^\mathrm{r} \big)
    = - \int_\T p^0(t) x^0(t) + p^\mathrm{a}(t) x^\mathrm{r} \, \mathrm{d}t 
\end{equation}
because the expected average value of the regulation signal over each hour of the planning horizon vanishes. In fact, if there was a systematic bias, grid operators would notice and change their dispatch accordingly. Empirically, from 1 July 2020 through 30 June 2024, the regulation signal deviated the most from zero between 11pm and midnight taking an average value of~$-0.032$. 
 
\subsection{Assumptions}
We make the following simplifying assumptions, which are discussed in Appendix~\ref{apx:ass}.
\begin{enumerate}
    \item \textbf{Prices:} Day-ahead and availability prices are known one day in advance and exogenous. 
    \item \textbf{Market granularity:} There are no volume ticks, \ie, there is no floor, ceiling, or predefined increment for bid volumes.
    \item \textbf{Storage degradation:} Degradation is negligible when the SOC remains between 10\% and 90\% of total capacity.
    \item \textbf{Storage dynamics:} Maximum charging and discharging power, as well as charging and discharging efficiencies, are constant.
    \item \textbf{Terminal condition:} The cost-to-go function is zero if the terminal SOC induced by trading on the day-ahead market is greater or equal to $y_0$, and infinite otherwise. This is enforced via the constraint $y_0 - \Delta t \sum_{k \in \K} \alpha_k \geq y^\star$.
    \item \textbf{Intraday market:} Intraday prices equal day-ahead prices. Bids are submitted over 15min intervals and can be placed immediately before the interval begins. 
    \item \textbf{Availability:} The battery is available for operation on all days.
\end{enumerate}

\subsection{Test Setup}
We evaluate our optimization model on four years of data from 1 July 2020, when FCR bidding blocks transitioned to~4h durations, through 30 June 2024. Days affected by transitions into or out of daylight savings time are excluded. Each day at 8am, the storage operator measures the SOC, solves problem~\eqref{pb:P} for the following day while accounting for frequency deviations from 8am to midnight on the current day, and submits FCR bids. At the end of each day, we record profits, total energy charged and discharged, minimum and maximum SOC, and minimum and maximum power output. At the end of the simulation horizon, we assess the impact of joint FCR and arbitrage participation on profits and energy output. Finally, we compare results for France with those obtained for other European countries.

\subsection{Input Data}
We now characterize the input data for our back test to guide the interpretation of numerical results. Data sources are listed in Appendix~\ref{apx:input_data}.

The European energy market faced two disruptions during our simulation horizon: the COVID-19 pandemic and the 2022 energy crisis. The pandemic reduced economic activity, lowering energy demand and prices. As the economy recovered, both rebounded through 2021~\citep{kuik2022energy}.
In February 2022, the Russian invasion of Ukraine disrupted natural gas supplies, triggering a sharp rise in prices that remained high through early~2023. Since then, markets have stabilized and returned to more typical conditions~\citep{emiliozzi2023european}.

These disruptions affect storage operators by altering arbitrage opportunities. During the pandemic, day-ahead price spreads were relatively narrow, but widened significantly during the energy crisis. The left panel of Figure~\ref{fig:prices} shows this trend through the difference between daily minimum and maximum prices, discounted by charging and discharging efficiencies. It also displays the evolution of average daily FCR prices. Both signals follow similar trends, though FCR prices have stayed at or below pandemic-era levels since summer of~2023 and have neared zero since January~2024. 

The right panel of Figure~\ref{fig:prices} shows average day-ahead and FCR prices over the full four year simulation horizon. Day-ahead prices tend to peak in the morning (7--10am) and in the evening (6--9pm), and reach their lowest level overnight (3--5am) and in the afternoon (1pm--4pm). These patterns enable either a single arbitrage cycle, charging at night and discharging in the evening, or two cycles, charging at night and in the afternoon, and discharging in the morning and evening. In contrast, FCR prices are highest between 4--8am and lowest from 8pm to midnight.

\begin{figure}
\centering
\begin{subfigure}{0.60\textwidth}
\centering
\begin{tikzpicture}[scale=0.7, transform shape]
    \begin{axis}[
      width = 14cm,
      height = 6cm,
      xtick = {1, 185, 366, 550, 731, 915, 1096, 1280, 1461},
      xticklabels={Jul 20, Jan 21, Jul 21, Jan 22, Jul 22, Jan 23, Jul 23, Jan 24, Jul 24},
      enlarge x limits=false,
      ymin = 0, ymax = 400,
      ytick distance = 50,
      grid=both,
      legend style={
      font=\scriptsize,
      legend columns = 2,
    at={(0.99,0.97)},
    anchor=north east,
  },
    ]      
    \addplot[opacity=0.5, double, blue] table [x=day, y=spread, col sep=tab] {data/input_spread.txt};
    \addplot[opacity=0.5, double, Bittersweet] table [x=day, y=fcr_4h, col sep=tab] {data/input_spread.txt};
     \legend{Spread ($\nicefrac{\text{\EUR{}}}{\text{MWh}}$), 
     FCR ($\nicefrac{\text{\EUR{}}}{\text{MW}\cdot 4\text{h}}$)
     }
    \end{axis}
\end{tikzpicture}
\end{subfigure}
\begin{subfigure}{0.35\textwidth}
\begin{tikzpicture}[scale=0.7, transform shape]
    \begin{axis}[
      width = 8.5cm,
      height = 6cm,
      enlarge x limits=false,
      ymin = 0, ymax = 200,
      ytick distance = 25,
      xtick = {0, 4, 8, 12, 16, 20, 24},
      xticklabels={0:00, 4:00, 8:00, 12:00, 16:00, 20:00, 0:00},
      grid=both,
      legend style={
        font=\scriptsize,
        legend columns = 2,
        at={(0.02,0.03)},  
        anchor=south west,
  },
    ]      
    \addplot[const plot, opacity=0.75, double, blue] table [x=hour_starting, y=day_ahead, col sep=tab] {data/input_day_ahead_average.txt};
    \addplot[const plot, opacity=0.75, double, Bittersweet] table [x=hour_starting, y=fcr_4h, col sep=tab] {data/input_fcr_average.txt};
    \legend{
    Day-Ahead ($\nicefrac{\text{\EUR{}}}{\text{MWh}}$),
    FCR ($\nicefrac{\text{\EUR{}}}{\text{MW}\cdot 4\text{h}}$)
    }
    \end{axis}
\end{tikzpicture}
\end{subfigure}
\caption{Price data.}
\label{fig:prices}
\end{figure}
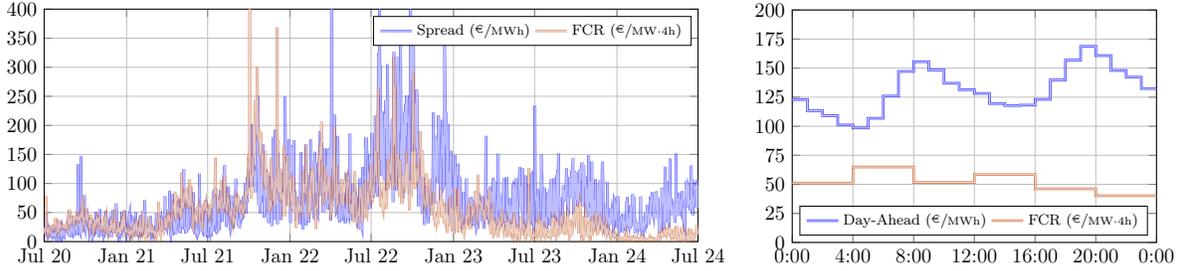

The remaining model parameters relate to the FCR signal and battery specifications. We set $\Delta t = \gamma'=15$min, satisfying Assumption~\ref{ass:div} and yielding $K = 96$ intervals for a 24h planning horizon. The empirical FCR signal uses, on average, 70\% of the deviation time budget, reaching a maximum of~106\% on the most extreme day. For the battery, we assume a storage capacity of~100kWh, and a maximum charging and discharging power of~$50$kW. Charging and discharging efficiencies are set to~$0.92$, consistent with commercially available lithium-ion batteries~\citep{WEC20}. Formally, $\ubar y = 10$kWh, $\bar y = 90$kWh, $\bar x = -\ubar x = 50$kW, and $\eta^+ = \eta^- = 0.92$.

\subsection{Numerical Implementation}
All numerical experiments are conducted on Intel Xeon Platinum 8260 CPUs with 24~cores, 48~threads, 192GB of RAM, and a 2.4GHz base clock speed~\citep{reuther2018interactive}. Simulations are implemented in Julia~1.10.1 using JuMP~1.22.2, with Gurobi~11.0.2 as the solver. All code and data are available at \url{www.github.com/lauinger/robust-storage-bidding}.

\subsection{Numerical Experiments}
We first test our model using data from France, then compare the results with those from other countries participating in the European balancing reserve platform since 1 July 2020: Austria, Belgium, Germany, the Netherlands, and Switzerland.

\subsubsection{The French Case}
Table~\ref{tab:results} in Appendix~\ref{apx:results} summarizes nine experiments, which yield the following results:
\begin{enumerate}
    \item \textbf{The relaxation and restriction are tight.} Experiments~1 and~2 solve the mixed-integer linear relaxation and restriction, respectively, using $ y_0 = \frac{\bar y + \eta^c\cdot\eta^d\cdot\ubar y}{1+\eta^c\cdot\eta^d} = 53.328$kWh as the initial SOC on each day, ensuring symmetric headroom for charging and discharging. Both yield a mean daily objective of $6.990$\EUR. The relaxation violates the bilinear SOC bound in 11.995 out of $\frac{(K-1)K}{2} = 4560$ constraints on average. As both formulations yield identical objectives but the relaxation is occasionally infeasible, we use the restriction in all subsequent experiments.
    \item \textbf{The terminal constraint    
    supports continuous operations.} Experiment~3  initializes each day’s SOC to the previous day's terminal SOC. The mean daily profit is $6.855$\EUR, and the mean SOC at midnight is $52.364$kWh, both close to Experiment~2.
    \item \textbf{Without FCR, profits and solve time decrease, while energy output increases.} Experiment~4 disables FCR. The restriction becomes exact as $\nu_{1k} + \nu_{2k} = 1$ for all $k \in \K$, deactivating the bilinear terms. Compared to Experiment~3, the mean daily profit drops to 5.624\EUR (-$18$\%), energy output rises to $143.045$kWh (+$54$\%), and the mean solve time, as reported by Gurobi, drops to 0.010s ($50,000$x~speedup).
    \item \textbf{Early bidding reduces profits.} In Experiment~5, market bids must be submitted at 8am the day preceding delivery, which reduces profit to $4.961$\EUR (-$12$\% vs. Experiment~4) because the uncertainty on the SOC at midnight effectively tightens the SOC constraints.
    \item \textbf{Intraday trading increases profits, reduces output and solve time, may slightly violate SOC bounds.} Experiment~6 enables intraday trading: mean daily profits increase to 14.987\EUR (12.883\EUR from FCR), energy output drops to 78.484kWh, and solve time falls to 3.538s (150x faster than Experiment~3). SOC ranges from 4.864kWh to 90.182kWh, exceeding operational bounds  (10--90kWh), but within physical limits (0--100kWh). Figure~\ref{fig:btsocecdf} in Appendix~\ref{apx:results} shows that the 10kWh bound is satisfied on 91\% of all test days and the 90kWh bound is violated on one day only. 
    \item \textbf{Intraday trading mitigates early-bidding penalty.} Experiment~7 mirrors Experiment~6 but assumes bids are submitted at midnight. The mean daily profit is 14.985\EUR, only 0.59\% higher, suggesting that intraday flexibility offsets early-bidding penalties.
    \item \textbf{Perfect efficiency raises profits and output.} Experiment~8  assumes lossless charging and discharging. Profits rise to 17.474\EUR (+17\%) and energy output to 126.831kWh (+62\%) relative to Experiment~6. The SOC respects operational bounds as guaranteed by Proposition~\ref{prop:intraday}.
    \item \textbf{Restricting day-ahead trading reduces profits and output.} Experiment~9 limits day-ahead trading to the energy needed to compensate FCR-induced SOC fluctuations, \ie, to available headroom. Profits drop to 13.134\EUR (-$12$\%) and energy output to 37.377kWh (-$52$\%) compared to Experiment~6.
\end{enumerate}

In conclusion, joint participation in arbitrage and FCR increases profits and reduces energy output compared to no FCR participation. The restriction is tight and solvable in 25min. Compensating FCR-induced SOC fluctuations via intraday trading doubles profits, reduces energy output, and reduces solve time to under 5s, but may slightly violate SOC bounds. 

While intraday trading is promising, it may be risky in illiquid markets or if there are large differences between day-ahead and intraday prices. Some of its benefits can be achieved by (\emph{i})~reoptimizing day-ahead bids after FCR bids have been submitted, \eg, based on the information available at noon on the day preceding delivery; or by (\emph{ii})~offering no FCR between noon and midnight, which reduces the uncertainty on the initial SOC. 

The results suggest that it may be attractive to limit arbitrage and accept lower profits in return for a markedly lower energy output. Figure~\ref{fig:bt_market} shows that through 2022, the profit gap from limited arbitrage was small, but widened afterwards. Since January 2024, profits under limited arbitrage have nearly vanished, which is consistent with the decline in FCR prices shown in Figure~\ref{fig:prices}. 

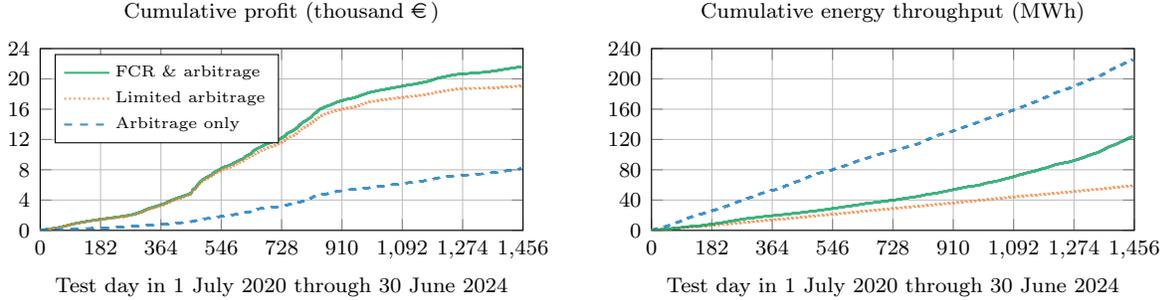
\begin{figure}
\begin{subfigure}{0.48\textwidth}
\centering
\begin{tikzpicture}[font=\scriptsize]
    \begin{axis}[
        height=4cm,
        width=8cm,
        enlarge x limits=false,
        ymin = 0, ymax = 24,
        ytick distance = 4,
        yticklabel style={
        /pgf/number format/fixed,
        /pgf/number format/precision=0
        },
        scaled y ticks=false,
        xmin = 0, xmax = 1456,
        xtick distance=182,
        xlabel={Test day in 1 July 2020 through 30 June 2024},
        ylabel={},
        title={Cumulative profit (thousand \EUR{})},
        grid=both,
        legend style={
        cells={anchor=west},
        legend pos= north west,
        font=\tiny
        }, 
    ]
    \addplot+[no marks, const plot, opacity=.75, solid, thick, ForestGreen] table [x=testday, y=profit_e6, col sep=tab] {data/bt_market.txt};
    \addplot+[no marks, const plot, opacity=.75, densely dotted, thick, Orange] table [x=testday, y=profit_e9, col sep=tab] {data/bt_market.txt};
    \addplot+[no marks, const plot, opacity=.75, dashed, thick, NavyBlue] table [x=testday, y=profit_e4, col sep=tab] {data/bt_market.txt};
    \legend{
        FCR \& arbitrage, Limited arbitrage, Arbitrage only
    }
    \end{axis}
\end{tikzpicture}
\end{subfigure}
\begin{subfigure}{0.48\textwidth}
\centering
\begin{tikzpicture}[font=\scriptsize]
    \begin{axis}[
        height=4cm,
        width=8cm,
        enlarge x limits=false,
        ymin = 0, ymax = 240,
        ytick distance = 40,
        yticklabel style={
        /pgf/number format/fixed,
        /pgf/number format/precision=0
        },
        scaled y ticks=false,
        xmin = 0, xmax = 1456,
        xtick distance=182,
        xlabel={Test day in 1 July 2020 through 30 June 2024},
        ylabel={},
        title={Cumulative energy output (MWh)},
        grid=both,
        legend style={
        cells={anchor=west},
        legend pos= north west,
        font=\scriptsize
        }, 
    ]
    \addplot+[no marks, const plot, opacity=.75, dashed, thick, NavyBlue] table [x=testday, y=output_e4, col sep=tab] {data/bt_market.txt};
    \addplot+[no marks, const plot, opacity=.75, densely dotted, thick, Orange] table [x=testday, y=output_e9, col sep=tab] {data/bt_market.txt};
    \addplot+[no marks, const plot, opacity=.75, solid, thick, ForestGreen] table [x=testday, y=output_e6, col sep=tab] {data/bt_market.txt};
\end{axis}
\end{tikzpicture}
\end{subfigure}
\caption{Evolution of profits and output for different types of market participation.}
\label{fig:bt_market}
\end{figure}

\subsubsection{Multicountry Comparison}
We repeat Experiment~6 for Austria, Belgium, Germany, the Netherlands, and Switzerland, and observe similar trends. As shown in Figure~\ref{fig:multicountry} in Appendix~\ref{apx:results}, cumulative profits level off and energy output increases in the last year of the planning horizon, reflecting a shift toward higher arbitrage volumes and reduced FCR participation in reaction to declining FCR prices. 

\section*{Conclusion}\addcontentsline{toc}{section}{Conclusion}
We addressed two questions for storage operators: ($i$) How to capture nonlinear SOC dynamics in joint arbitrage and FCR participation on the timescale of market decisions; and ($ii$) What is the effect of joint market participation on profits and energy output?

Theorem~\ref{th:P} establishes that the nonconvex robust market participation probem~\eqref{pb:P}, with functional uncertainties and continuous-time constraints, is equivalent to a finite-dimensional mixed-integer bilinear optimization problem, which discretizes time on the scale of market decisions. Maybe surprisingly, enforcing the continuous-time constraints requires just one additional linear constraint per trading interval. Standard practice discretizes time before optimization, but  Example~\ref{ex:time_discretization} shows that this may lead to SOC violations. Continuous-time constraints avoid such issues with negligible computational overhead.

Although commercial solvers can handle the exact bilinear formulation, solve times are too long for practical use in our numerical experiments. We circumvent this challenge by introducing a mixed-integer linear relaxation and restriction and find the gap between them to be negligible in practice. The restriction solves in 30min on average.

Compared to no FCR participation, joint market participation increases profits and reduces energy output. FCR participation introduces SOC uncertainty, as power production depends on FCR signals that are revealed in real time. To address this uncertainty, we design an intraday trading strategy that compensates FCR-induced SOC fluctuations. Proposition~\ref{prop:intraday} establishes that the strategy relaxes SOC constraints by narrowing the set of admissible FCR signals, and slightly tightens constraints on charging and discharging power. In a four-year backtest, the strategy enables higher FCR bids, which more than double profits and reduce energy output by~15\%. Intraday trading also reduces the solve time to under~5s on average and eliminates any penalty from submitting FCR bids in the morning of the day preceding delivery, as required under current market rules, rather than submitting bids at midnight on the day of delivery. 

Our intraday strategy guarantees SOC feasibility only under lossless charging and discharging. With~85\% roundtrip efficiency, the SOC ranges from 4.9kWh to 90.2kWh, within physical limits (0-100kWh), but outside the tighter operational range (10-90kWh) used to mitigate storage degradation. The operational bounds are met on 91\% of all days in our backtest. While physically feasible, the strategy leaves open future work on SOC guarantees under charging and discharging losses. The fast solve time makes our method suited for integration into more complex models with uncertain price forecasts or discrepancies between day-ahead and intraday prices.

\paragraph{Acknowledgments.}
We acknowledge the MIT SuperCloud and Lincoln Laboratory Supercomputing Center for providing HPC resources that have contributed to our research results.

%% file: 6_appendix.tex
\section{Lifting and Adjoint Operators, Assumptions, Data, Results}

\subsection{Lifting and Adjoint Operators}\label{apx:L}
The lifting operator maps any vector $\bm{v} \in \R^K$ to a piecewise constant function defined as $(L_t \bm{v} )(\tau) = v_{\ceil{\tau / \Delta t}}$ if $\tau \in [0, t)$, $= 0$ otherwise, for all $\tau \in \T$. The scaled adjoint operator maps any function $w\in\set{R}(\T,\R)$ to a $K$-dimensional vector defined as $(L^\dagger_t w)_l = \frac{1}{\sigma_l(t)} \int_{\T_l}w(\tau) \, \mathrm{d}(\tau)$ if $l < k$, $= \frac{1}{\sigma_k(t)} \int_{(k-1)\Delta t}^t w(\tau) \, \mathrm{d}(\tau)$ if $l = k$, $= 0$ otherwise, for all $l \in \K$. Note that $L^\dagger_t$ is indeed adjoint to $L_t$ because $\int_\T (L_t \bm v)w(\tau) \, \mathrm{d}\tau = \bm \sigma \odot \bm v^\top (L^\dagger_t w)$ for all $\bm v \in \mathbb{R}^K$ and $w \in \set{R}(\set{T}, \mathbb{R})$.

\subsection{Assumptions}\label{apx:ass}
In the following, we discuss the assumptions in Section~\ref{sec:app}.
\begin{enumerate}
    \item Day-ahead and availability prices are known one day in advance and exogenous. In practice, price risk can be reduced through new market products and better forecasting. Forecasts have been improving \citep{zhang2022short, kraft2020modeling} and 
    electricity markets are offering new products that allow storage operators to reduce price risk \citep{de2019implications}, such as loop block orders on the day-ahead market.
    \item There are no volume ticks. In practice, the volume tick is 1MW for frequency regulation bids and 0.1MW for day-ahead bids.
    \item There is no degradation if the state-of-charge is constrained to lie within 10\% and 90\% of the storage capacity. In practice, limiting the usable range of storage reduces but does not fully prevent degradation \citep{thompson2018economic}.
    \item The maximum charging and discharging power as well as the charging and discharging efficiencies are constant. In practice, they depend, among others. on the state-of-charge and ambient temperature \citep{pandzic2018storage}.
    \item The cost-to-go function is zero if the terminal state-of-charge induced by trading on the day-ahead market is greater or equal to $y_0$, and infinite otherwise, which is modeled by the constraint $y_0 - \Delta t \sum_{k \in \K} \alpha_k \geq y^\star$.
    \item Intraday prices are the same as day-ahead prices, intraday bids are made over 15 minute intervals and can be submitted right before the start of an interval. In practice, there can be a lead-time of, \eg, 5~minutes in German intraday markets \citep{kaya2024delivering}. Such lead times may be accounted for by considering a greater deviation time budget. Intraday prices have a limited influence on the results because intraday bids compensate FCR-induced state-of-charge fluctuations only. The bids will thus be small and tend to cancel themselves. 
    \item The battery is available for use on all days. In practice, there would be some downtime for maintenance and repairs.
\end{enumerate}

\subsection{Input Data}\label{apx:input_data}
Our input data stems from the following sources.
\begin{itemize}
    \item Day-ahead prices come from the European Network of Transmission System Operators for Electricity (ENTSOE, \url{https://transparency.entsoe.eu}).
    \item Availability prices for FCR come from a European platform for balancing reserves (\url{https://regelleistung.net}) and from the French grid operator (\url{https://www.services-rte.com}).
    \item Frequency measurements with 10s~resolution come from the French grid operator (\url{https://www.services-rte.com}).
\end{itemize}

\subsubsection{Empirical Regulation Signal}
\begin{figure}[t]
\centering
\includegraphics[width=12cm]{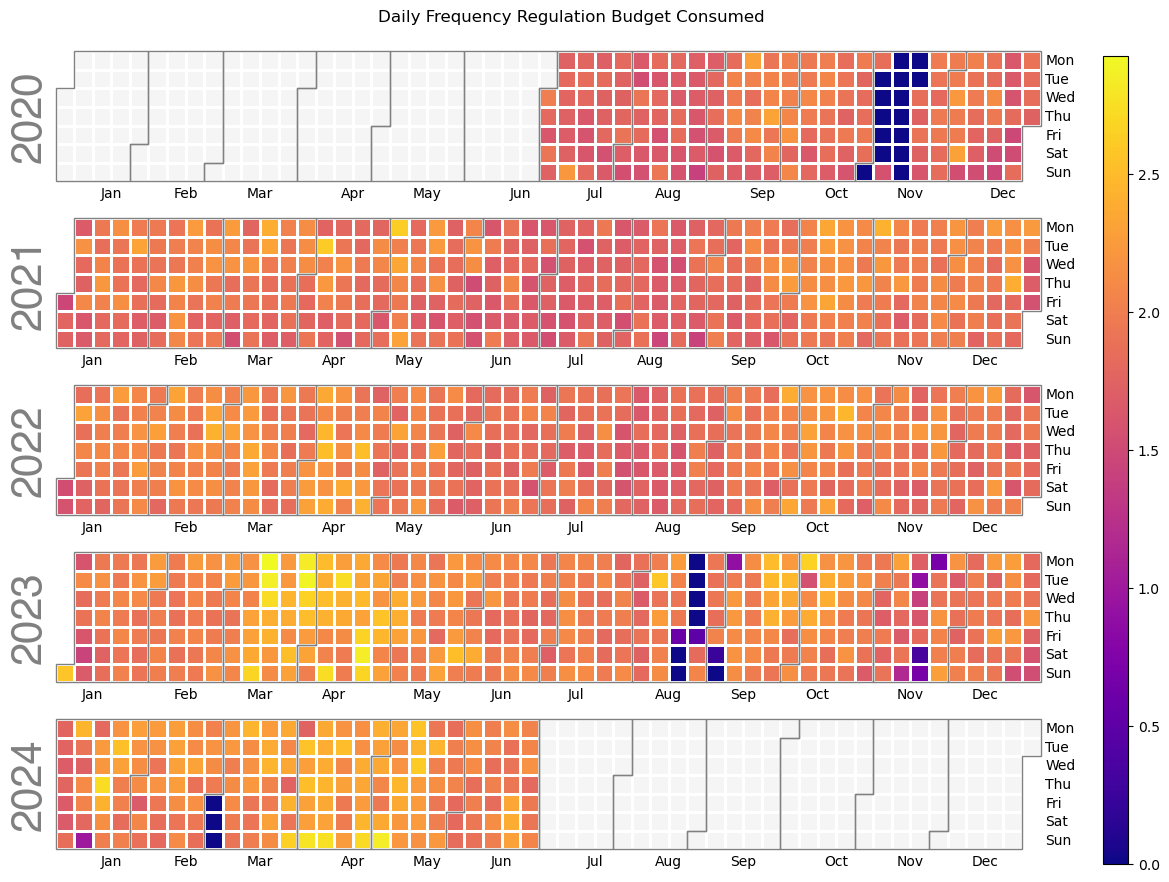}
\caption{Total daily frequency budget consumed by the regulation signal for the French balancing zone across the experimental time period}
\label{fig:freq_budget}
\end{figure}
Based on the frequency measurements, we compute the empirical regulation signal and check whether it falls in the uncertainty set~$\Xi$. We choose $\gamma' = 15$min, which yields a deviation budget of $\gamma = 2.75$h according to equation~\ref{eq:gamma}. The empirical regulation signal has a minimum value of $-0.789$ and a maximum value of $1$, which respects the admissible range of $[-1, 1]$. The deviation budget is exceeded on March 13, 14, 27, 28, and April 22 in 2023 and on April 7, 14, and May 5 in 2024. The empirical regulation signal uses 70\% of the deviation time budget on average and 106\% of the budget on the day with the highest deviation, see Figure~\ref{fig:freq_budget}. Choosing a higher value for $\gamma'$ would reduce the number of days above the budget. However, it would also reduce the amount of regulation power that storage operators can sell because they would need to withhold a greater amount of energy for each unit of regulation power. 

\subsection{Results}\label{apx:results}

\begin{figure}[H]
\centering
\begin{tikzpicture}[font=\scriptsize]
    \begin{axis}[
        enlarge x limits=false,
        ymin = 0, ymax = 1,
        ytick distance = 0.1,
        scaled y ticks=false,
        xmin = 0, xmax = 100,
        xtick distance=10,
        xlabel={State-of-charge (kWh)},
        ylabel={Probability (-)},
        grid=both, 
        legend style={
        cells={anchor=west},
        legend pos= north west,
        font=\scriptsize
        }, 
    ]
    \addplot+[no marks, const plot, opacity=1, solid, Orange] table [x=btsocmin, y=ep, col sep=tab] {data/bt_soc_ecdf.txt};  
    \addplot+[no marks, const plot, opacity=1, dashed, NavyBlue] table [x=btsocmax, y=ep, col sep=tab] {data/bt_soc_ecdf.txt};
    \legend{
        Daily min, Daily max
    }
    \end{axis}
\end{tikzpicture}
\caption{Empirical cumulative distribution of the daily minimum and maximum state-of-charge under intraday trading in Experiment~6.}
\label{fig:btsocecdf}
\end{figure}
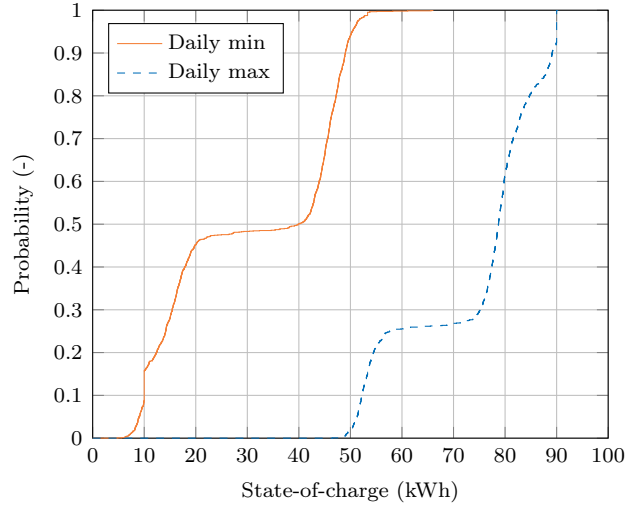

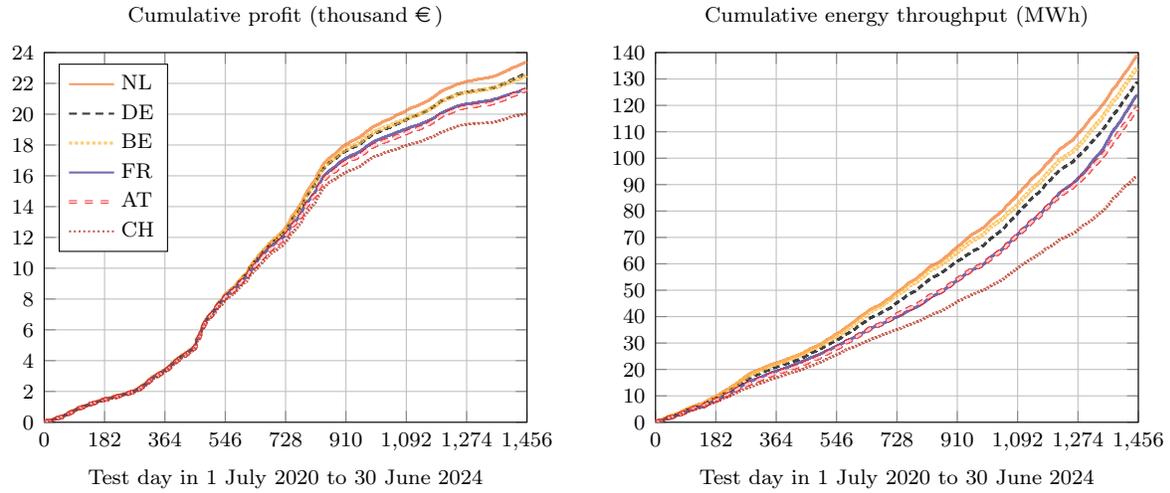
\begin{figure}[H]
\begin{subfigure}{0.48\textwidth}
\centering
\begin{tikzpicture}[font=\scriptsize]
    \begin{axis}[
        height=6.5cm,
        width=8cm,
        enlarge x limits=false,
        ymin = 0, ymax = 24,
        ytick distance = 2,
        yticklabel style={
        /pgf/number format/fixed,
        /pgf/number format/precision=0
        },
        scaled y ticks=false,
        xmin = 0, xmax = 1456,
        xtick distance=182,
        xlabel={Test day in 1 July 2020 to 30 June 2024},
        ylabel={},
        title={Cumulative profit (thousand \EUR{})},
        grid=both, 
        legend style={
        cells={anchor=west},
        legend pos= north west,
        font=\scriptsize
        }, 
    ]
    \addplot+[no marks, const plot, opacity=.75, solid, thick, Orange] table [x=testday, y=NL, col sep=tab] {data/bt_profit.txt};
    \addplot+[no marks, const plot, opacity=.75, densely dashed, thick, black] table [x=testday, y=DE, col sep=tab] {data/bt_profit.txt};
    \addplot+[no marks, const plot, opacity=.75, densely dotted, very thick, Dandelion] table [x=testday, y=BE, col sep=tab] {data/bt_profit.txt};
    \addplot+[no marks, const plot, opacity=.75, solid, thick, BlueViolet] table [x=testday, y=FR, col sep=tab] {data/bt_profit.txt};
    \addplot+[no marks, const plot, opacity=.75, dashed, thin, double, Red] table [x=testday, y=AT, col sep=tab] {data/bt_profit.txt};
    \addplot+[no marks, const plot, opacity=.75, densely dotted,  thick, BrickRed] table [x=testday, y=CH, col sep=tab] {data/bt_profit.txt};
    \legend{
        NL, DE, BE, FR, AT, CH
    }
    \end{axis}
\end{tikzpicture}
\end{subfigure}
\begin{subfigure}{0.48\textwidth}
\centering
\begin{tikzpicture}[font=\scriptsize]
    \begin{axis}[
        height=6.5cm,
        width=8cm,
        enlarge x limits=false,
        ymin = 0, ymax = 130,
        ytick distance = 10,
        yticklabel style={
        /pgf/number format/fixed,
        /pgf/number format/precision=0
        },
        scaled y ticks=false,
        xmin = 0, xmax = 1456,
        xtick distance=182,
        xlabel={Test day in 1 July 2020 to 30 June 2024},
        ylabel={},
        title={Cumulative energy output (MWh)},
        grid=both, 
        legend style={
        cells={anchor=west},
        legend pos= north west,
        font=\scriptsize
        }, 
    ]
    \addplot+[no marks, const plot, opacity=.75, solid, thick, Orange] table [x=testday, y=NL, col sep=tab] {data/bt_output.txt};
    \addplot+[no marks, const plot, opacity=.75, densely dashed, thick, black] table [x=testday, y=DE, col sep=tab] {data/bt_output.txt};
    \addplot+[no marks, const plot, opacity=.75, densely dotted, very thick, Dandelion] table [x=testday, y=BE, col sep=tab] {data/bt_output.txt};
    \addplot+[no marks, const plot, opacity=.75, solid, thick, BlueViolet] table [x=testday, y=FR, col sep=tab] {data/bt_output.txt};
    \addplot+[no marks, const plot, opacity=.75, dashed, thin, double, Red] table [x=testday, y=AT, col sep=tab] {data/bt_output.txt};
    \addplot+[no marks, const plot, opacity=.75, densely dotted, thick, BrickRed] table [x=testday, y=CH, col sep=tab] {data/bt_output.txt};
    \end{axis}
\end{tikzpicture}
\end{subfigure}
\caption{Profits and energy output for Austria (AT), Belgium (BE), Switzerland (CH), Germany (DE), France (FR), and the Netherlands (NL).}
\label{fig:multicountry}
\end{figure}

\begin{table}[H]
\tiny
\centering
\begin{tabular}{l|rrrrrrrrr}
    \toprule
    \multicolumn{10}{c}{Experiments} \\
    \midrule
    Number & 1 & 2 & 3 & 4 & 5 & 6 & 7 & 8 & 9 \\ 
    \midrule
    \multicolumn{10}{c}{Parameters} \\
    \midrule
    Model type & Relax & Restrict & Restrict & Exact & Restrict & Restrict & Restrict & Restrict & Restrict \\
    Bidding time & Midnight	& Midnight & Midnight &	Midnight & 8am & 8am & Midnight & 8am & 8am \\
    Roundtrip efficiency (-) & 0.846 & 0.846 & 0.846 & 0.846 &	0.846	& 0.846	& 0.846	& 1.000	& 0.846 \\
    FCR participation & Yes & Yes & Yes & No & Yes & Yes & Yes & Yes & Yes \\
    Day-ahead participation & Full & Full & Full & Full & Full & Full & Full & Full & Limited \\
    Intraday trading & No & No & No & No & No & Yes & Yes & Yes & Yes \\
    Coupling between days & No & No & Yes & Yes & Yes & Yes & Yes & Yes & Yes \\
    Solve time limit (min) & 60 & 60 & 15 & 15 & 15 & 15 & 15 & 15 & 15 \\
    \midrule
    \multicolumn{10}{c}{Objective quality} \\
    \midrule
    Mean objective value (EUR) & 6.990 & 6.990 & 6.858 & 5.624 & 9.760 & -31.096 & 15.335 & 30.935 & -33.549 \\
    Mean objective bound (EUR) & 7.003 & 7.017 & 6.899 & 5.624 & 9.809 & -31.094 & 15.335 & 30.935 & -33.549 \\
    Mean gap (\%) & 0.186 & 0.385 & 0.594 & 0.000 & 0.500 & 0.006 & 0.000 & 0.000 & 0.000 \\
    Mean bilinear violations (\#) & 11.995 & 0.000 & 0.000 & 0.000 & 0.000 & 0.000 & 0.000 & 0.000 & 0.000 \\
    \midrule
    \multicolumn{10}{c}{Mean daily profit (EUR)} \\
    \midrule
    \textbf{Total} & \textbf{6.988} & \textbf{6.987} & \textbf{6.855} & \textbf{5.624} & \textbf{4.961} & \textbf{14.897} & \textbf{14.985} & \textbf{17.474} & \textbf{13.134} \\
    FCR & 2.815 & 2.816 & 2.795 & 0.000 & 1.643 & 12.882 & 12.849 & 11.680 & 13.844 \\
    Day-ahead & 4.173 & 4.171 & 4.061 & 5.624 & 3.318 & 3.056 & 2.567 & 6.121 & 0.560 \\
    Intraday & 0.000 & 0.000 & 0.000 & 0.000 & 0.000 & -1.042 & -0.431 & -0.328 & -1.270 \\
    \midrule
    \multicolumn{10}{c}{Mean daily energy output (kWh)} \\
    \midrule
    & \textbf{91.706} & \textbf{91.707} & \textbf{92.753} & \textbf{143.045} & \textbf{75.011} & \textbf{78.484} & \textbf{78.003} & \textbf{126.831} & \textbf{37.377} \\
    \midrule
    \multicolumn{10}{c}{Solve time} \\
    \midrule
    Mean (s) & 980.229 & 1492.190 & 514.994 & 0.010 & 555.172 & 3.538 & 1.537 & 0.716 & 0.562 \\
    Maximum (min) & 60.039 & 60.031 & 15.023 & 0.031 & 15.014 & 2.590 & 0.641 & 0.046 & 0.096 \\
    \midrule
    \multicolumn{10}{c}{Planned worst-case state-of-charge (kWh)} \\
    \midrule
    Minimum & 10.000 & 10.000 & 10.000 & 10.000 & 10.000 & 10.000 & 10.000 & 10.000 & 28.006 \\
    True maximum & 91.731 & 90.000 & 90.000 & 90.000 & 90.000 & 90.000 & 90.000 & 90.000 & 90.000 \\
    Modeled maximum & 90.000 & 90.000 & 90.000 & 90.000 & 90.000 & 90.000 & 90.000 & 90.001 & 90.000 \\
    \midrule
    \multicolumn{10}{c}{Empirical state-of-charge levels (kWh)} \\
    \midrule
    Minimum & 10.000 & 10.000 & 10.000 & 10.000 & 9.999 & 4.864 & 5.022 & 9.999 & 31.010 \\
    Mean daily minimum & 23.996 & 23.965 & 23.449 & 10.622 & 29.754 & 30.605 & 30.161 & 21.252 & 45.914 \\
    Mean SOC at midnight & 53.328 & 53.328 & 52.364 & 53.355 & 52.305 & 52.905 & 50.763 & 52.307 & 52.384 \\
    Mean daily maximum & 76.953 & 76.912 & 77.141 & 90.000 & 71.480 & 74.190 & 75.101 & 82.064 & 57.791 \\
    Maximum & 90.000 & 90.000 & 90.000 & 90.000 & 90.001 & 90.182 & 90.047 & 90.000 & 82.380 \\
    \midrule
    \multicolumn{10}{c}{Planned worst-case power levels (kW)} \\
    \midrule
    Minimum & -50.000 & -50.000 & -50.000 & -50.000 & -50.000 & -50.000 & -50.000 & -50.000 & -44.671 \\
    Maximum & 50.000 & 50.000 & 50.000 & 50.000 & 50.000 & 50.000 & 50.000 & 50.000 & 44.788 \\
    \midrule
    \multicolumn{10}{c}{Empirical power levels (kW)} \\
    \midrule
    Minimum & -50.000 & -50.000 & -50.000 & -50.000 & -50.000 & -50.000 & -50.000 & -50.000 & -35.251 \\
    Mean daily minimum & -34.002 & -33.967 & -34.539 & -49.687 & -30.994 & -28.961 & -28.645 & -34.670 & -18.520 \\
    Mean daily maximum & 33.879 & 33.887 & 34.265 & 49.162 & 28.918 & 30.479 & 30.282 & 36.296 & 19.767 \\
    Maximum & 50.000 & 50.000 & 50.000 & 50.000 & 50.000 & 50.000 & 50.000 & 50.000 & 44.197 \\
    \bottomrule
\end{tabular}
\caption{Numerical results for the French test case.}
\label{tab:results}
\end{table} 

%% file: 7_proofs.tex
\section{Proofs}\label{sec:proofs}

This section contains the proofs of all theorems and propositions. We first prove basic properties on the power output function (Proposition~\ref{prop:x}), the SOC function (Proposition~\ref{prop:y}), and the uncertainty set (Propositions~\ref{prop:sym} and~\ref{prop:Xi_k}). Next, we prove the reformulations of the bounds on power output (Proposition~\ref{prop:charger}) and the SOC (Propositions~\ref{prop:soc_lb} and~\ref{prop:soc_ub}). Subsequently, we prove the finite-dimensional reformulation of the robust decision problem (Theorem~\ref{th:P}), and show that problem~\eqref{pb:P} reduces to a linear program for specific parameter values~(Proposition~\ref{prop:tractable}) and admits problem~\eqref{pb:P:restrict} as restriction~(Proposition~\ref{prop:pres}). We prove the bound on the difference between maximum SOC estimates from the relaxed problem~\eqref{pb:P:relax} and the restricted problem~\eqref{pb:P:restrict}~(Proposition~\ref{prop:bound}). Finally, we prove Proposition~\ref{prop:intraday} on feasible intraday trading strategies.

\subsection{Basic Properties}

\begin{proof}[Proof of Proposition~\ref{prop:x}]
    The result follows directly from $x^\uparrow$ and $x^\downarrow$ being nonnegative functions.
\end{proof}

\begin{proof}[Proof of Proposition~\ref{prop:y}]
    By definition, the SOC function is given as
    \begin{align*}
        & y\left( x^0, x^\uparrow, x^\downarrow, \xi, y_0, t \right) \\
        = & y_0 + \int_0^t 
        \eta^\mathrm{c} \left[ x\left(x^0(\tau), x^\uparrow(\tau), x^\downarrow(\tau), \xi(\tau)\right)  \right]^-
        - \frac{1}{\eta^\mathrm{d}} \left[ x\left(x^0(\tau), x^\uparrow(\tau), x^\downarrow(\tau), \xi(\tau)\right)  \right]^+
        \, \mathrm{d}\tau \\
        = & y_0 + \int_0^t \min\left\{
        -\eta^\mathrm{c} x\left(x^0(\tau), x^\uparrow(\tau), x^\downarrow(\tau), \xi(\tau)\right),
        - \frac{x\left(x^0(\tau), x^\uparrow(\tau), x^\downarrow(\tau), \xi(\tau)\right)}{\eta^\mathrm{d}} \right\}
        \, \mathrm{d}\tau,
    \end{align*}
    where the second equality holds because $0 < \eta^\mathrm{c} \leq \frac{1}{\eta^\mathrm{d}}$. The integrand is concave in $x(\cdot)$ because the minimum of affine functions is a concave function, and decreasing in $x(\cdot)$ because $0 < \eta^\mathrm{c} \leq \frac{1}{\eta^\mathrm{d}}$. Moreover, the integrand is concave in $x^0$, $x^+$, and $x^-$ because  $x(\cdot)$ is affine in $x^0$, $x^+$, and $x^-$, and the composition of a concave function with an affine function yields a concave function~\citep[p.~79]{SB04}. Thus, the SOC function is concave in $x^0$, $x^+$, and $x^-$, because integration preserves convexity~\citep[p.~79]{SB04}. The monotonicity properties hold because the   composition of a decreasing function with an increasing/nondecreasing/nonincreasing function yields a decreasing/nonincreasing/nondecreasing function.
\end{proof}

\begin{proof}[Proof of Proposition~\ref{prop:sym}]
    We have $\pm \xi \in \set{R}(\set{T},[-1,1]) \iff \pm \vert \xi \vert \in \set{R}(\set{T},[-1,1])$ and 
    \begin{equation*}
        \int_\T \vert \xi(t) \vert \, \mathrm{d}t
        =
        \int_\T \vert (\vert \xi(t) \vert) \vert \, \mathrm{d}t
        =
        \int_\T \vert -(\vert \xi(t) \vert) \vert \, \mathrm{d}t. \qedhere
    \end{equation*}
\end{proof}

\begin{proof}[Proof of Proposition~\ref{prop:Xi_k}]
    We first show that $L_t L^\dagger_t \Xi^+ \subseteq \Xi^+ $ for any $t \in \T_k$ and any $k \in \K$. If $\xi \in \set{R}(\T, [0,1])$, then $L_t L^\dagger_t \xi \in \set{R}(\T, [0,1])$ because $\inf_{\tau \in \T} \, (L_t L^\dagger_t \xi)(\tau) = 0$ if $t < T$,
    \begin{align}
        & \inf_{\tau \in \T} \, (L_T L^\dagger_T \xi)(\tau)
        =
        \min_{l \in  \{1,\ldots,K\}} \, 
        \frac{\int_{\T_l} \xi(\tau) \, \mathrm{d}\tau}{\Delta t} 
        \geq
        \min_{l \in  \{1,\ldots,K\}} \, 
        \inf_{\tau \in \T_l} \,
        \xi(\tau)
        \geq 0,~~\text{and} \\
        & \sup_{\tau \in \T} \, (L_t L^\dagger_t \xi)(\tau)
        =
        \max_{l \in  \{1,\ldots,k\}} \, 
        \frac{\int_{(l-1)\Delta t}^{\min\{t, \, l \Delta t\}}
        \xi(\tau) \, \mathrm{d}\tau}{\sigma_l(t)} 
        \leq
         \max_{l \in  \{1,\ldots,k\}} \, 
        \sup_{\tau \in [(l-1)\Delta t, \, \min\{t, \, l \Delta t\}]} \,
        \xi(\tau)
        \leq 1.
    \end{align}
    In addition, if $\int_\T \xi(\tau) \, \mathrm{d}\tau \leq \gamma$, then $\int_\T L_t L^\dagger_t \xi(\tau) \, \mathrm{d}\tau \leq \gamma$ because
    \begin{equation}
        \int_\T L_t L^\dagger_t \xi(\tau) \, \mathrm{d}\tau
        =
        \sum_{l = 1}^{k-1} \int_{\T_l} (L^\dagger_t \xi)_l \, \mathrm{d}\tau
        + \int_{(k-1)\Delta t}^t (L^\dagger_t \xi)_k \, \mathrm{d}\tau
        = 
        \int_0^t \xi(\tau) \, \mathrm{d}\tau
        \leq 
        \int_\T \xi(\tau) \, \mathrm{d}\tau,
    \end{equation}
    where the inequality holds because $\xi$ is nonnegative and $0 \leq t \leq T$. To prove that 
    \begin{equation*}
        L^\dagger_t \Xi^+ = \left\{ \bm \xi \in [0, 1]^K: \sum_{l = 1}^{k} \sigma_l(t) \xi_l \leq \gamma, \, \xi_l = 0 ~ \forall l \in \{k+1,\ldots,K\} \right\},
    \end{equation*}
    we first note for any function $\xi \in \Xi^+$ that $L^\dagger_t \xi \in [0, 1]^K$ because $\xi \in \set{R}(\T, [0,1])$, 
    \begin{equation}
        \sum_{l = 1}^{k} \sigma_l(t) (L^\dagger_t \xi)_l = \int_0^t \xi(\tau) \, \mathrm{d}\tau \leq \int_\T \xi(\tau) \leq \gamma,
    \end{equation}
    and $(L^\dagger_t \xi)_l = 0$ for all $l \in \{k+1,\ldots,K\}$.
    Next, we fix an arbitrary vector $\bm \xi \in [0,1]^K$ such that $\sum_{l = 1}^{k} \sigma_l(t) \xi_l \leq \gamma$ and $\xi_l = 0$ for all $l \in \{k+1,\ldots,K\}$. Then, $L_t \bm \xi \in \Xi^+$ as $L_t \bm \xi \in \set{R}(\T, [0,1])$ and 
    \begin{equation}
        \int_\T (L_t \bm \xi)(\tau) \, \mathrm{d}\tau
        = \int_0^t (L_t \bm \xi)(\tau) \, \mathrm{d}\tau
        = \sum_{l = 1}^{k-1} \int_{\T_l} \xi_l \, \mathrm{d}\tau
        + \int_{(k-1)\Delta t}^t \xi_k \, \mathrm{d}\tau
        = \sum_{l = 1}^{k} \int_{\T_l} \sigma_l(t) \xi_l
        \leq \gamma.
    \end{equation}
    Applying the adjoint operator to both sides of the set relation yields $L^\dagger_t L_t \bm \xi = \bm \xi \in L^\dagger_t \Xi^+$, where the equality holds because $\xi_l = 0$ for all $l$ in $\{k+1, \ldots, K\}$.
\end{proof}

\subsection{Lower and Upper Bounds on Power Output}

\begin{proof}[Proof of Proposition~\ref{prop:charger}]
    The first equivalence holds because
    \begin{equation}
        \max_{\xi \in \Xi, \, t \in \T} ~ x(x^0(t), x^\uparrow(t), x^\downarrow(t), \xi(t))
        =
        \max_{t \in \T} ~ x^0(t) + x^\uparrow(t) 
        = 
        \max_{k \in \K} ~ x^0_k + x^\uparrow_k,
    \end{equation}
    where the first equality holds because the power output function is nondecreasing in~$\xi$. For any fixed~$t$, it is thus optimal to set $\xi(t)$ as large as possible, \ie, $\xi(t) = 1$. The second equality holds because $x^0$ and $x^\uparrow$ are piecewise constant. The proof for the second equivalence is similar and omitted for brevity.
\end{proof}

\subsection{Lower Bound on the SOC}

\begin{proof}[Proof of Proposition~\ref{prop:soc_lb}]
    We have
    \begin{equation}
        y(x^0, x^\uparrow, x^\downarrow, \xi, y_0, t) \geq \ubar y ~~ \forall t \in \T,~\forall \xi \in \Xi 
        ~~ \iff ~~
        \min_{\xi \in \Xi, \, t \in \T}~y(x^0, x^\uparrow, x^\downarrow, \xi, y_0, t) \geq \ubar y.
    \end{equation}
    We first consider the minimization over $\xi$ for $t = k \Delta t$ and any $k \in \{ 0,\ldots, K\}$,
    \begin{align}
        & \min_{\xi \in \Xi}~y( x^0, x^\uparrow, x^\downarrow, \xi, y_0, t )
        \overset{(1)}{=} \min_{\xi \in \Xi^+}~y( x^0, x^\uparrow, x^\downarrow, \xi, y_0, t ) \\
        \overset{(2)}{=} & \min_{\xi \in \Xi^+} y_0 + \sum_{l = 1}^k \int_{\T_l} \min\left\{ -\eta^\mathrm{c} \left( x^0_l + \xi(\tau) x^\uparrow_l \right), - \frac{1}{\eta^\mathrm{d}} \left( x^0_l + \xi(\tau) x^\uparrow_l \right) \right\} \, \mathrm{d}\tau \\
        \overset{(3)}{=} & \min_{\xi \in \Xi^+ \cap \set{R}(\T, \{0,1\})} y_0 + \sum_{l = 1}^k \int_{\T_l} \min\left\{ -\eta^\mathrm{c} \left( x^0_l + \xi(\tau) x^\uparrow_l \right), - \frac{1}{\eta^\mathrm{d}} \left( x^0_l + \xi(\tau) x^\uparrow_l \right) \right\} \, \mathrm{d}\tau \\
        \overset{(4)}{=} & \min_{\xi \in \Xi^+ \cap \set{R}(\T, \{0,1\})} y_0 + \sum_{l = 1}^k \int_{\T_l} \min\left\{ -\eta^\mathrm{c} x^0_l, -\frac{x^0_l}{\eta^\mathrm{d}} \right\} (1 - \xi(\tau)) \\
        & \phantom{\min_{\xi \in \Xi^+ \cap \set{R}(\T, \{0,1\})} y_0 + \sum_{l = 1}^k \int_{\T_l}} + \min\left\{ -\eta^\mathrm{c} (x^0_l + x^\uparrow_l), -\frac{x^0_l + x^\uparrow_l}{\eta^\mathrm{d}} \right\} \xi(\tau) \, \mathrm{d}\tau, \notag
    \end{align}
    where the first equality follows from $y$ being nonincreasing in~$\xi$ and from the symmetry of~$\Xi$. In fact, for any given regulation signal~$\xi \in \Xi$, the signal~$\vert \xi \vert$ will also be in $\Xi$ and achieve a SOC that is at least as low as the one achieved by~$\xi$. We can thus restrict $\xi$ to be nonnegative without loss of optimality. The second equality holds because $0 < \eta^\mathrm{c} \leq \frac{1}{\eta^\mathrm{d}}$ and because $x^0$ and $x^\uparrow$ are piecewise constant. The third equality follows from Lemma~1 in \cite{my_v2g_lp}, which applies because the integrand is concave, continuous, and nonincreasing in~$\xi(\tau)$ and because $\Xi$ is a special case of the uncertainty set
     \begin{equation}
        \left\{ \xi \in \set{R}(\set{T}, [-1, 1]): \int_{[t - \Gamma]^+}^t \vert \xi(t') \vert \, \mathrm{d}t' \leq \gamma ~~ \forall t \in \set{T} \right\},
    \end{equation}
    used in \cite{my_v2g_lp}. In fact, $\Xi$ can be obtained by setting the additional problem parameter $\Gamma$ to $T$. The fourth equality holds because the integrand is only ever evaluated at $\xi(\tau) = 0$ and $\xi(\tau) = 1$, and can thus be linearized. To simplify notation, we set
    \begin{equation}
        \alpha_l = \max\left\{ \eta^\mathrm{c} x^0_l, \, \frac{x^0_l}{\eta^\mathrm{d}} \right\}
        ~~ \mathrm{and} ~~
        \beta_l = \max\left\{ \eta^\mathrm{c} (x^0_l + x^\uparrow_l), \frac{x^0_l + x^\uparrow_l}{\eta^\mathrm{d}} \right\}.
    \end{equation}
    So far, we have transformed a continuous nonconvex optimization problem into a continuous linear program. Now, we will transform the continuous linear program into a standard linear program with vectorial decision variables,
    \begin{align}
        & \min_{\xi \in \Xi^+ \cap \set{R}(\T, \{0,1\})} y_0 + \sum_{l = 1}^k \int_{\T_l} -\alpha_l (1 - \xi(\tau)) - \beta_l \xi(\tau) \, \mathrm{d}\tau \\
        \overset{(1)}{=} & \min_{\xi \in \Xi^+} y_0 + \sum_{l = 1}^k \int_{\T_l} -\alpha_l (1 - \xi(\tau)) - \beta_l \xi(\tau) \, \mathrm{d}\tau  \\
        \overset{(2)}{=} & \min_{\xi \in L_T L^\dagger_T \Xi^+} y_0 + \sum_{l = 1}^k \int_{\T_l} -\alpha_l (1 - \xi_l ) - \beta_l \xi_l \, \mathrm{d}\tau \\
        \overset{(3)}{=} & \min_{\bm \xi \in L^\dagger_T \Xi^+} y_0 + \Delta t \sum_{l = 1}^k -\alpha_l (1 - \xi_l) - \beta_l \xi_l.
    \end{align}
    The first equality follows again from Lemma~1 in \cite{my_v2g_lp}, which continues to apply because the integrand is still continuous, nonincreasing (as $\beta_l \geq \alpha_l$), and concave (in fact, affine) in~$\xi(\tau)$. The second equation holds because $\xi$ is integrated against a piecewise constant function and can thus be averaged over the trading intervals, and because $L_T L^\dagger_T \Xi^+ \subseteq \Xi^+$. The third equality holds because piecewise constant functions can be modeled by vectors.

    For $t \in \T_k$, following a similar reasoning about linear programming sensitivity analysis as in Proposition~8 in \cite{my_v2g_lp}, which applies thanks to Assumption~\ref{ass:div}, one can show that
    \begin{equation}
        \min_{t \in \T_k} \min_{\xi \in \Xi}~y(x^0, x^\uparrow, x^\downarrow, \xi, y_0, t)
        = \min_{l \in \{k-1,k\}} \min_{\xi \in \Xi}~y(x^0, x^\uparrow, x^\downarrow, \xi, y_0, l \Delta t).
    \end{equation}
    Instead of having to consider all $t \in \T$, it is thus sufficient to only consider $t = k \Delta t$ for any $k \in \{ 0,\ldots, K\}$. Using the explicit expression for $L^\dagger_T \Xi^+$ from Proposition~\ref{prop:Xi_k}, we now dualize the linear minimization problem over~$\bm \xi$,
    \begin{align}
        & \min_{\bm \xi \in L^\dagger_T \Xi^+} y_0 + \Delta t \sum_{l = 1}^k -\alpha_l (1 - \xi_l) - \beta_l \xi_l \\
        = & \min_{0 \leq \bm \xi \leq 1}~y_0 - \Delta t \sum_{l = 1}^k \alpha_l (1 - \xi_l) + \beta_l \xi_l ~~\text{s.t.}~~\Delta t \sum_{l = 1}^k \xi_l \leq \gamma \\
        = & \min_{0 \leq \bm \xi \leq 1} ~ \max_{0 \leq \ubar \lambda_k}~y_0 - \gamma \ubar \lambda_k - \Delta t \sum_{l = 1}^k \alpha_l + (\beta_l - \alpha_l - \ubar \lambda_k) \xi_l \\
        = & \max_{0 \leq \ubar \lambda_k}~\min_{0 \leq \bm \xi \leq 1} ~ y_0 - \gamma \ubar \lambda_k - \Delta t \sum_{l = 1}^k \alpha_l + (\beta_l - \alpha_l - \ubar \lambda_k) \xi_l \\
        = & \max_{0 \leq \ubar \lambda_k} ~ y_0 - \gamma \ubar \lambda_k - \Delta t \sum_{l = 1}^k \alpha_l + \left[ \beta_l - \alpha_l - \ubar \lambda_k \right]^+ \\
        = & \max_{0 \leq \ubar \lambda_k, \, \ubar {\bm \Lambda}_k \in \mathbb{R}^k_+} ~ y_0 - \gamma \ubar \lambda_k - \Delta t \sum_{l = 1}^k \alpha_l + \ubar \Lambda_{kl} ~~\text{s.t.}~~ \ubar \Lambda_{kl} \geq \beta_l - \alpha_l - \ubar \lambda_k ~~\forall l = 1,\ldots,k.
    \end{align}
    The optimal value of the maximization problem is decreasing in $\alpha$ and $\beta$, which can thus be used as hypographical variables to linearize the dependence on $x^0$ and $x^\uparrow$ by adding the constraints~\eqref{reformulation:soc_lb:ldr}.
    
    The lower bound $\ubar y$ on the SOC is valid, if $y_0 \geq 0$ and if the optimal value of the maximization problem over the SOC exceeds $\ubar y$ for all $k \in \K$, which is the case if and only if~\eqref{reformulation:soc_lb} is feasible.
\end{proof}

\subsection{Upper Bound on the SOC}

\begin{proof}[Proof of Proposition~\ref{prop:soc_ub}]
    We have
    \begin{equation}
        y(x^0, x^\uparrow, x^\downarrow, \xi, y_0, t) \leq \bar y ~~ \forall t \in \T,~\forall \xi \in \Xi 
        ~~ \iff ~~
        \max_{\xi \in \Xi, \, t \in \T}~y(x^0, x^\uparrow, x^\downarrow, \xi, y_0, t) \leq \bar y.
    \end{equation}

    The upper bound is an upper bound on a concave function with functional uncertainty that must hold at all points in time. To make it tractable, we (1) show that we can consider piecewise constant regulation signals, (2) linearize the SOC function, (3) transform the robust constraints into deterministic constraints, (4) reduce the problem of checking if the upper bound holds for fixed market decisions and a fixed point in time in a given trading interval to a one-dimensional piecewise linear convex optimization problem, (5) show that adding a single linear constraint to that problem insures that the upper bound will hold throughout the trading interval, and (6) show how disjunctive constraints can be used to account for the upper bound when optimizing the market decisions.
    
    \subsubsection{Uncertainty Discretization}
    We first consider the maximization over $\xi$ for any fixed $t \in \T_k$ and any $k \in \K$,
    \begin{align}
        & \max_{\xi \in \Xi}~y( x^0, x^\uparrow, x^\downarrow, \xi, y_0, t )
        \overset{(1)}{=} \max_{\xi \in \Xi^+}~y( x^0, x^\uparrow, x^\downarrow, -\xi, y_0, t ) \\
        \overset{(2)}{=} & \max_{\xi \in \Xi^+}~y_0 + \sum_{l = 1}^{k} \int_{(l-1)\Delta t}^{\min\{t,\, l \Delta t\}} \min\left\{ \eta^\mathrm{c} \left( \xi(\tau) x^\downarrow_l - x^0_l \right), \frac{\xi(\tau) x^\downarrow_l - x^0_l}{\eta^\mathrm{d}} \right\} \, \mathrm{d}\tau \\
        \overset{(3)}{=} & \max_{\xi \in L_t  L^\dagger_t \Xi^+}~y_0 + \sum_{l = 1}^{k} \int_{(l-1)\Delta t}^{\min\{t,\, l \Delta t\}} \min\left\{ \eta^\mathrm{c} \left( \xi(\tau) x^\downarrow_l - x^0_l \right), \frac{\xi(\tau) x^\downarrow_l - x^0_l}{\eta^\mathrm{d}} \right\} \, \mathrm{d}\tau \\
        \overset{(4)}{=} & \max_{\bm \xi \in L^\dagger_t \Xi^+}~ y_0 + \sum_{l = 1}^{k} \sigma_l(t) \min\left\{ \eta^\mathrm{c} \left( \xi_l x^\downarrow_l - x^0_l \right), \frac{\xi_l x^\downarrow_l - x^0_l}{\eta^\mathrm{d}} \right\},
    \end{align}
    where $\sigma_l(t) = \min\{t, \, l \Delta t\} - (l-1) \Delta t$. The first equality follows from $y$ being nonincreasing in $\xi$ and from the symmetry of $\Xi$. For any given regulation signal $\xi \in \Xi$, the signal $-\vert \xi \vert$ is also in $\Xi$ and achieves a SOC that is at least as high as the one achieved by~$\xi$. We thus restrict $\xi$ to be nonnegative and flip its sign in the argument of~$y$. The second equality holds because $0 < \eta^\mathrm{c} \leq \frac{1}{\eta^\mathrm{d}}$ and because $x^0$ and $x^\downarrow$ are piecewise constant. The objective function is independent of $\xi(\tau)$ for $\tau \in [t, T]$. It is thus optimal to set $\xi(\tau) = 0$ for $\tau \in [t, T]$. This restriction maximizes the flexibility in choosing regulation signals $\xi(\tau)$ for $\tau \in [0, t)$. In addition, the objective function is concave in $\xi$ because minima of affine functions are concave and integration preserves convexity \citep[p.~79]{SB04}. The third equality holds because it is optimal to only consider functions in $L_t L^\dagger_t \Xi^+$, which maximize the concave objective function by virtue of being piecewise constant and are guaranteed to be feasible as $L_t L^\dagger_t \Xi^+ \subseteq \Xi^+$ by Proposition~\ref{prop:Xi_k}. The fourth equality holds because piecewise constant functions can be modeled by vectors.
    
    \subsubsection{Linearization}
    Using the explicit expression for $L^\dagger_t \Xi^+$ from Proposition~\ref{prop:Xi_k}, we now linearize the objective function,  
    \begin{align}
    & \max_{0 \leq \bm \xi \leq 1}~ y_0 + \sum_{l = 1}^{k} \sigma_l(t) \min\left\{ \eta^\mathrm{c} \left( \xi_l x^\downarrow_l - x^0_l \right), \frac{\xi_l x^\downarrow_l - x^0_l}{\eta^\mathrm{d}} \right\}
    ~~\text{s.t.}~~ \sum_{l = 1}^{k} \sigma_l(t) \xi_l \leq \gamma \\
    = & \max_{0 \leq \bm \xi \leq 1} \, \min_{0 \leq \bm u \leq 1} \, y_0 + \sum_{l = 1}^{k} \sigma_l(t) \Big( \eta^\mathrm{c} u_l + \frac{1 - u_l}{\eta^\mathrm{d}} \Big)\left( \xi_l x^\downarrow_l - x^0_l \right) ~~\text{s.t.}~~ \sum_{l = 1}^{k-1} \sigma_l(t) \xi_l \leq \gamma \\
    = & \min_{0 \leq \bm u \leq 1} \, y_0 - \sum_{l = 1}^{k} \sigma_l(t) \Big( \eta^\mathrm{c} u_l + \frac{1 - u_l}{\eta^\mathrm{d}} \Big) x^0_{l} \\
    & + \max_{0 \leq \bm \xi \leq 1} \sum_{l = 1}^{k} \sigma_l(t) \Big( \eta^\mathrm{c} u_l + \frac{1 - u_l}{\eta^\mathrm{d}} \Big) \xi_l x^\downarrow_{l}~~\text{s.t.}~~\sum_{l = 1}^k \sigma_l(t) \xi_l \leq \gamma. \notag
    \end{align}
    The constraints $\xi_l = 0$, $l = k+1, \ldots, K$, are not explicitly enforced because they hold at optimality. The first equality holds because the optimal $\bm u$ is binary-valued as the objective function is affine in~$\bm u$. The second equality follows from \citeauthor{vonNeumman1928spiel}'s (1928) minimax theorem, which applies as the objective function is bilinear in $\bm u$ and $\bm \xi$, and both $\bm u$ and $\bm \xi$ lie in compact convex sets.

    \subsubsection{Robust Reformulation}
    We will now dualize the inner maximization problem over $\bm \xi$ so that it becomes a minimization problem and can be combined with the outer minimization problem over $\bm u$. We have
    \begin{align}
        & \max_{0 \leq \bm \xi \leq 1} \sum_{l = 1}^{k} \sigma_l(t) \Big( \eta^\mathrm{c} u_l + \frac{1 - u_l}{\eta^\mathrm{d}} \Big) \xi_l x^\downarrow_{l} ~\text{s.t.}~ \sum_{l = 1}^k \sigma_l(t) \xi_l \leq \gamma \\
        = & \max_{0 \leq \bm \xi \leq 1} \, \min_{0 \leq \bar \lambda_k} \,
        \bar \lambda_k \Big( \gamma - \sum_{l = 1}^k \sigma_l(t) \xi_l \Big)
        + \sum_{l = 1}^{k} \sigma_l(t) \Big( \eta^\mathrm{c} u_l + \frac{1 - u_l}{\eta^\mathrm{d}} \Big) \xi_l x^\downarrow_{l} \\
        = & \min_{0 \leq \bar \lambda_k} \, \max_{0 \leq \bm \xi \leq 1} \,
        \gamma \bar \lambda_k + \sum_{l = 1}^{k}  \sigma_l(t) \Big( \Big( \eta^\mathrm{c} u_l + \frac{1 - u_l}{\eta^\mathrm{d}} \Big) x^\downarrow_l - \bar \lambda_k \Big) \xi_l \\
        = & \min_{0 \leq \bar \lambda_k} \,
        \gamma \bar \lambda_k + \sum_{l = 1}^{k}  \sigma_l(t) \Big[ \Big( \eta^\mathrm{c} u_l + \frac{1 - u_l}{\eta^\mathrm{d}} \Big) x^\downarrow_l - \bar \lambda_k \Big]^+.
    \end{align}
    Introducing the Lagrange multiplier $\bar \lambda_k$ for the constraint $\sum_{l = 1}^k \sigma_l(t) \xi_l \leq \gamma$ yields the first equality. The second equality holds because of linear programming duality, which applies because $\bm \xi$ is bounded. The third equality holds because it is optimal to set $\xi_l = 1$ if its multiplier is nonnegative and $= 0$ otherwise for all $l = 1,\ldots, k$.

    \subsubsection{Dimensionality Reduction: Solving for \texorpdfstring{$\bm u$}{u}}\label{sec:u}
    Using the robust reformulation, we will now solve the minimization problem for $\bm u$. For shorthand notation, we introduce $\Delta \eta = \frac{1}{\eta^\mathrm{d}} - \eta^\mathrm{c}$, which is nonnegative because $0 < \eta^\mathrm{c}, \, \eta^\mathrm{d} < 1$. We have
    \begin{align}
        & \min_{\substack{0 \leq \bm u \leq 1 \\ 0 \leq \bar \lambda_k}} \, y_0 + \gamma \bar \lambda_k + \sum_{l = 1}^k \sigma_l(t) \Big( \Big[ \Big( \eta^\mathrm{c} u_l + \frac{1 - u_l}{\eta^\mathrm{d}} \Big) x^\downarrow_l - \bar \lambda_k \Big]^+ 
        - \Big( \eta^\mathrm{c} u_l + \frac{1 - u_l}{\eta^\mathrm{d}} \Big) x^0_l \Big) \\
        = & \min_{\substack{0 \leq \bm u \leq 1 \\ 0 \leq \bar \lambda_k}} \, y_0 + \gamma \bar \lambda_k + \sum_{l = 1}^k \sigma_l(t)
        \max\Big\{ \Big( \frac{1}{\eta^\mathrm{d}} - \Delta \eta \, u_l \Big) \Big( x^\downarrow_l - x^0_l \Big) - \bar \lambda_k, \, \Big( \Delta \eta \, u_l - \frac{1}{\eta^\mathrm{d}} \Big) x^0_l \Big\} \\
        = & \min_{0 \leq \bar \lambda_k} \,  y_0 + \gamma \bar \lambda_k + \sum_{l = 1}^k \sigma_l(t)
        \min_{0 \leq u_l \leq 1} \max\Big\{ \Big( \frac{1}{\eta^\mathrm{d}} - \Delta \eta \, u_l \Big) \Big( x^\downarrow_l - x^0_l \Big) - \bar \lambda_k, \, \Big( \Delta \eta \, u_l - \frac{1}{\eta^\mathrm{d}} \Big) x^0_l \Big\}.
    \end{align}
    The inner minimization problem can be solved analytically by case distinction on $x^\downarrow_l$ and $x^0_l$ for each $l = 1,\ldots,k$. For example, if $0 \leq x^0_l \leq x^\downarrow_l$, then the first term in the max operator is nonincreasing in $u_l$, the second term is nondecreasing in $u_l$, and it is optimal to set $u_l$ as close as possible to the intersection between the first and the second term, \ie, to
    \begin{equation*}
        u^\ast_l = \max\Big\{ 0, \, \min\Big\{ 
        \frac{1}{1 - \eta^\mathrm{c} \eta^\mathrm{d}} - \frac{\bar \lambda_k}{\Delta \eta \, x^\downarrow_l}, \, 1 \Big\} \Big\}.
    \end{equation*}
    Table~\ref{tab:u} lists the monotonicity properties of the two terms in the max operator, the optimizer $u^\star_l$, and the optimal value of the inner minimization problem for all admissible cases. The optimal value function $\varphi(x^0_l, x^\downarrow_l, \bar \lambda_k)$ of the inner minimization problem  
    is convex nonincreasing and piecewise linear in~$\bar \lambda_k$. For $\bar \lambda_k \geq x^\downarrow_{l} / \eta^\mathrm{d}$, the value function is constant in~$\bar \lambda_k$, see Figure~\ref{fig:varphi}. In summary, the maximum value of the SOC at time~$t$ for fixed market decisions~$x^0$ and~$\xdn$ is the optimal value of the one-dimensional piecewise linear convex optimization problem
    \begin{equation}
    y_0 + \min_{0 \leq \bar \lambda_k } \, \gamma \bar \lambda_k + \sum_{l = 1}^{k} \sigma_l(t) \varphi(x^0_l, x^\downarrow_l, \bar \lambda_k).
    \end{equation}

    \begin{table}
    \centering
    \resizebox{\textwidth}{!}{%
	\begin{tabular}{c|cccl}
    Case & $\big( \frac{1}{\eta^\mathrm{d}} - \Delta \eta \, u_l \big) \big( x^\downarrow_l - x^0_l \big) - \bar \lambda_k$
    & $\big( \Delta \eta \, u_l - \frac{1}{\eta^\mathrm{d}} \big) x^0_l$
    & $u^\star_l$
    & Optimal value
    \\
    \midrule
    $x^\downarrow_{l} \leq x^0_{l}$ & nondecreasing & nondecreasing & 0 & $\max \left\{ \frac{x^\downarrow_{l} - x^0_{l}}{\eta^\mathrm{d}} - \bar \lambda_k, \, -\frac{x^0_{l}}{\eta^\mathrm{d}} \right\}$ \\
    $x^0_{l} \leq 0$ & nonincreasing & nonincreasing & 1 & $\max\left\{ \eta^\mathrm{c}(x^\downarrow_{l} - x^0_{l}) - \bar \lambda_k, \, - \eta^\mathrm{c} x^0_{l} \right\}$ \\
    $ 0 \le x^0_{l} \le x^\downarrow_l \, \land \, x^\downarrow_l > 0 $ & nonincreasing & nondecreasing & $u^\ast_l$ & $\max\left\{ \eta^\mathrm{c}(x^\downarrow_{l} - x^0_{l}) - \bar \lambda_k, \, - \bar \lambda_k \frac{x^0_{l}}{x^\downarrow_{l}}, \, -\frac{x^0_{l}}{\eta^\mathrm{d}} \right\}$
    \end{tabular}
    }
    \caption{The optimal solution to the minimization problem over $u_l$ for $l = 1,\ldots,k$.}
    \label{tab:u}
    \end{table}

    \subsubsection{Time Discretization}
    We will now compute the maximum SOC for any time $t \in \T_k$. To ease notation, we define $\bar{\bar \lambda} = \frac{\bar x - \ubar x}{\eta^\mathrm{d}} \geq \max_{l=1,\ldots,k} x^\downarrow_l/\eta^\mathrm{d}$, where the inequality follows from Proposition~\ref{prop:charger}. We have
    \begin{align}
        & y_0 + \sup_{t \in \T_k} \, \min_{0 \leq \bar \lambda_k } \, \gamma \bar \lambda_k + \sum_{l = 1}^{k} \sigma_l(t) \varphi(x^0_l, x^\downarrow_l, \bar \lambda_k) \\
        = & y_0 + \sup_{t \in \T_k} \, \min_{0 \leq \bar \lambda_k \leq \bar{ \bar \lambda}} \, \gamma \bar \lambda_k + \Delta t \sum_{l = 1}^{k-1} \varphi(x^0_l, x^\downarrow_l, \bar \lambda_k) + (t - (k-1)\Delta t) \varphi(x^0_k, x^\downarrow_k, \bar \lambda_k) \\
        = & y_0 + \min_{0 \leq \bar \lambda_k \leq \bar{\bar \lambda}} \, \gamma \bar \lambda_k + \Delta t \sum_{l = 1}^{k-1} \varphi(x^0_l, x^\downarrow_l, \bar \lambda_k) + \sup_{t \in \T_k} \, (t - (k-1)\Delta t) \varphi(x^0_k, x^\downarrow_k, \bar \lambda_k) \\
         = & y_0 + \min_{0 \leq \bar \lambda_k \leq \bar{\bar \lambda}} \, \gamma \bar \lambda_k + \Delta t \sum_{l = 1}^{k-1} \varphi(x^0_l, x^\downarrow_l, \bar \lambda_k) + \Delta t \, \big[\varphi(x^0_k, x^\downarrow_k, \bar \lambda_k)\big]^+.
    \end{align}
     The first equality follows from substituting $\sigma_l$ by its definition, from $\gamma$ being nonnegative, and from $\varphi(x^0_l, x^\downarrow_l, \bar \lambda_k)$ being constant in $\bar \lambda_k$ for $\bar \lambda_k \geq x^\downarrow_l/\eta^\mathrm{d}$ for $l=1,\ldots,k$. The objective function is thus nondecreasing in $\bar \lambda_k$ for $\bar \lambda_k \geq \bar{\bar \lambda}_k$ and it is optimal to choose $\bar \lambda_k \leq \bar{\bar \lambda}_k$. The second equality follows from \citeauthor{vonNeumman1928spiel}'s (1928) minimax theorem, which applies because the objective function is affine in $t$ for a fixed $\bar \lambda_k$ and convex in $\bar \lambda_k$ for a fixed $t$, and because both $t$ and $\bar \lambda_k$ are restricted to line segments. The third equality holds because it is optimal to set $t = k\Delta t$ if $\varphi(x^0_k, x^\downarrow_k, \bar \lambda_k) \geq 0$ and $=(k-1)\Delta t$ otherwise. 

     \subsubsection{Disjunctive Constraints}\label{sec:M}

     We will now incorporate the upper bound on the SOC directly into the storage operator's decision problem in which $\bm x^0$ and $\bm x^\downarrow$ are decision variables. We have
     \begin{align}
     & \max_{\xi \in \Xi, \, t \in \T_k}\, y(x^0, x^\uparrow, x^\downarrow, \xi, y_0, t) \leq \bar y \\ 
    \iff & \exists \, \bar \lambda_k \in \big[0, \bar{\bar \lambda}], \, \bar{\bm \Lambda}_k \in \mathbb{R}^K : 
    \eqref{reformulation:soc_ub:soc}
    , \,
    \bar \Lambda_{kl} \geq \varphi(x^0_l, x^\downarrow_l, \bar \lambda_k), \, l=1,\ldots,k, \,
    \bar \Lambda_{kk} \geq 0,
    \end{align}
    where the upper bounds on $\varphi(x^0_l, x^\downarrow_l, \bar \lambda_k)$, $l = 1,\ldots,k$, may be complicated because of the case distinction in Table~\ref{tab:u}. For $l = k$, the upper bound can be modeled by affine constraints because
    \begin{equation}
    \begin{aligned}
    \big[\varphi(x^0_k, x^\downarrow_k, \bar \lambda_k)\big]^+ 
    = &
    \max\big\{ 0, \, \eta^\mathrm{c}(x^\downarrow_k - x^0_k) - \bar \lambda_k, \, -\eta^\mathrm{c} x^0_k \big\} \\
    = &
        \begin{cases}
            \varphi(x^0_k, x^\downarrow_k, \bar \lambda_k) & \text{if } x^0_k \leq 0, \\
            \big[ \eta^\mathrm{c}(x^\downarrow_k - x^0_k) - \bar \lambda_k \big]^+ & \text{if } 0 \leq x^0_k \leq x^\downarrow_k, \\
            0 & \text{if } x^\downarrow_k \leq x^0_k.
        \end{cases}
    \end{aligned}
    \end{equation}

    For $l=1,\ldots,k-1$, we introduce the binary variables $\upsilon_{1l}$ and $\upsilon_{2l}$ to model the case distinction. If $\upsilon_{1l} = 1$, the constraints $x^\downarrow_l \leq x^0_l$ and $\bar \Lambda_{kl} \geq \max\{ (x^\downarrow_l - x^0_l)/\eta^\mathrm{d} - \bar \lambda_k, \, -x^0_l/\eta^\mathrm{d} \}$ must hold, otherwise they may not hold. Similarly, if $\upsilon_{2l} = 1$, the constraints $x^0_l \leq 0$ and $\bar \Lambda_{kl} \geq \max\{ \eta^\mathrm{c}(x^\downarrow_l - x^0_l) - \bar \lambda_k, \, - \eta^\mathrm{c} x^0_l \}$ must hold, otherwise they may not hold. Finally, if $\upsilon_{1l}+\upsilon_{2l} = 0$, the constraint $0 \leq x^0_l \leq x^\downarrow_l$ must hold. If $0 = x^0_l = x^\downarrow_l$, then $\bar \Lambda_{kl} \ge 0$. Otherwise, if $0 \leq x^0_l \leq x^\downarrow_l$ and $x^\downarrow_l > 0$, the constraint $\bar \Lambda_{kl} \geq \max\{ \eta^\mathrm{c}(x^\downarrow_l - x^0_l) - \bar \lambda_k, \, - \bar \lambda_l x^0_l/x^\downarrow_l , \,  -x^0_l/\eta^\mathrm{d} \}$ must hold. If $\upsilon_{1l}+\upsilon_{2l} > 0$, the constraint may not hold. The optional constraints can be expressed as
    \begin{align}
        & (1- \upsilon_{1l}) \ubar x \leq x^0_l - x^\downarrow_l \leq \upsilon_{1l} \bar x, \, 
        \upsilon_{2l} \ubar x \leq x^0 \leq (1 - \upsilon_{2l}) \bar x \label{proof:soc_ub:binaries} \\
        & \bar \Lambda_{kl} \geq \frac{x^\downarrow_{l} - x^0_{l}}{\eta^\mathrm{d}} - \bar \lambda_k + (1 - \upsilon_{1l}) \Delta \eta \, \ubar x, \,
        \bar \Lambda_{kl} \geq -\frac{x^0_{l}}{\eta^\mathrm{d}} + \upsilon_{2l} \Delta \eta \, \ubar x \label{proof:soc_ub:Lbd1} \\
        & \bar \Lambda_{kl} \geq \eta^\mathrm{c} (x^\downarrow_{l} - x^0_{l}) - \bar \lambda_k - \upsilon_{1l} \Delta \eta \, \bar x, \,
        \bar \Lambda_{kl} \geq -\eta^\mathrm{c} x^0_{l} - (1 - \upsilon_{2l}) \Delta \eta \, \bar x \label{proof:soc_ub:Lbd2} \\
        & \bar \Lambda_{kl} x^\downarrow_{l} + \bar \lambda_k x^0_{l} \geq \upsilon_{2l} \frac{\ubar x(\bar x - \ubar x)}{\eta^\mathrm{d}} - \upsilon_{1l}\frac{\bar x^2}{4\eta^\mathrm{d}}.\label{proof:soc_ub:bilinear}
    \end{align}
To show this, we rely on two inequalities that arise from Proposition~\ref{prop:charger}: $x^0_l - x^\downarrow_l \geq \ubar x$ and $x^0_l \leq \bar x$, which holds as $x^\uparrow_l \geq 0$. Inequalities~\eqref{proof:soc_ub:binaries} ensure that $\upsilon_{1l} = 0 \implies \ubar x \leq x^0_l - x^\downarrow_l \leq 0$, $\upsilon_{1l} = 1 \implies 0 \leq x^0_l - x^\downarrow_l \leq \bar x$, $\upsilon_{2l} = 0 \implies 0 \leq x^0_l \leq \bar x$, and $\upsilon_{2l} = 1 \implies \ubar x \leq x^0_l \leq 0$. The inequalities~\eqref{proof:soc_ub:Lbd1} and~\eqref{proof:soc_ub:Lbd2} can be shown to be valid via a case distinction on $\upsilon_{1l}$, $\upsilon_{2l}$, and $\bar \lambda_k$. For example, if $\upsilon_{1l} = 1$, then $ \bar \Lambda_{kl} \geq (x^\downarrow_l - x^0_l)/\eta^\mathrm{d} - \bar \lambda_k$ as desired. If $\upsilon_{1l} = 0$, then the constraint may not be binding. We are thus searching for a constant $M$ such that $(x^\downarrow_l - x^0_l)/\eta^\mathrm{d} - \bar \lambda_k - M \leq \varphi(x^0_l, x^\downarrow_l, \bar \lambda_k)$. Using a case distinction on $\upsilon_{2l}$ and $\bar \lambda_k$, we find $M = -\Delta \eta \, \ubar x$, see Table~\ref{tab:M}.
Inequality~\ref{proof:soc_ub:bilinear} models the constraint $\bar \Lambda_{kl} \geq -\bar \lambda_k x^0_l/x^\downarrow_l$ on $0 \leq x^0_l \le x^\downarrow_l$ for $x^\downarrow_l > 0$. If $\upsilon_{1l} + \upsilon_{2l} = 0$, then $\bar \Lambda_{kl} x^\downarrow_l + \bar \lambda_k x^0_l \geq 0$ is equivalent to the original constraint if $0 \le x^0_l \le x^\downarrow_l$ and $x^\downarrow_l > 0$, and trivially true if $x^0_l = x^\downarrow_l = 0$ as desired. If $\upsilon_{1l} + \upsilon_{2l} > 0$, the constraint may not be binding. We are thus searching for a constant $M$ such that
\begin{equation}
    \bar \Lambda_{kl} x^\downarrow_l + \bar \lambda_k x^0_l + M 
    \geq \bar \Lambda_{kl} x^\downarrow_l - x^\downarrow_l \varphi(x^0_l, x^\downarrow_l, \bar \lambda_k) \geq 0
    \iff
    M \geq - \bar \lambda_k x^0_l - x^\downarrow_l \varphi(x^0_l, x^\downarrow_l, \bar \lambda_k),
\end{equation}
which, similarly to before, can be found via case distinction on $\upsilon_{1l}$, $\upsilon_{2l}$, and $\bar \lambda_k$, see Table~\ref{tab:M}. \qedhere 
\end{proof}
\begin{table}
    \centering
    \resizebox{\textwidth}{!}{%
	\begin{tabular}{ll|llc}
    $\upsilon_{1l} = 0$ & & $\varphi(x^0_l, x^\downarrow_l, \bar \lambda_k)$
    & $(x^\downarrow_l - x^0_l)/\eta^\mathrm{d} - \bar \lambda_k - \varphi(x^0_l, x^\downarrow_l, \bar \lambda_k)$
    & $M$ \\
    \midrule
    & $0 \leq \bar \lambda_k \leq  \eta^\mathrm{c} x^\downarrow_l$
    & $\eta^\mathrm{c} (x^\downarrow_l - x^0_l) - \bar \lambda_k$
    & $\Delta \eta \, (x^\downarrow_l - x^0_l)$
    & $-\Delta \eta \, \ubar x$ 
    \\
    $\upsilon_{2l} = 1$
    & $\eta^\mathrm{c} x^\downarrow_l \leq \bar \lambda_k$
    & $- \eta^\mathrm{c} x^0_l$
    & $ x^\downarrow_l/\eta^\mathrm{d} - \Delta \eta \, x^0_l - \bar \lambda_k$
    & $-\Delta \eta \, \ubar x$ 
    \\
    $\upsilon_{2l} = 0$
    & $\eta^\mathrm{c} x^\downarrow_l \leq \bar \lambda_k \leq x^\downarrow_l/\eta^\mathrm{d} $
    & $- \bar \lambda_k x^0_l/x^\downarrow_l$
    & $ (x^\downarrow_l - x^0_l)/\eta^\mathrm{d} - (1 - x^0_l/x^\downarrow_l) \bar \lambda_k$
    & $-\Delta \eta \, \ubar x$
    \\
    $\upsilon_{2l} = 0$
    & $\bar \lambda_k \geq x^\downarrow_l/\eta^\mathrm{d} $
    & $-x^0/\eta^\mathrm{d}$
    & $ x^\downarrow_l /\eta^\mathrm{d} - \bar \lambda_k$
    & $0$ \\
    \midrule
    $\upsilon_{2l} = 1$ &&& $-x^0/\eta^\mathrm{d} - \varphi(x^0_l, x^\downarrow_l, \bar \lambda_k)$ \\
    \midrule
    & $0 \leq \bar \lambda_k \leq \eta^\mathrm{c} x^\downarrow_l$ & $\eta^\mathrm{c}(x^\downarrow_l - x^0_l) - \bar \lambda_k$ & $-\eta^\mathrm{c} x^\downarrow_l - \Delta \eta \, x^0_l + \bar \lambda_k$ & $-\Delta \eta \, \ubar x$ \\
    & $\bar \lambda_k \geq \eta^\mathrm{c} x^\downarrow_l$ & $-\eta^\mathrm{c} x^0_l$ & $-\Delta \eta \, x^0_l$ & $-\Delta \eta \, \ubar x$ \\
    \midrule
    $\upsilon_{1l} = 1 $ &&& $\eta^\mathrm{c}(x^\downarrow_l - x^0_l) - \bar \lambda_k - \varphi(x^0_l, x^\downarrow_l, \bar \lambda_k)$ \\
    \midrule
    & $0 \leq \bar \lambda_k \leq x^\downarrow_l/\eta^\mathrm{d}$ & $(x^\downarrow_l - x^0_l)/\eta^\mathrm{d} - \bar \lambda_k$ & $\Delta \eta \, (x^0_l - x^\downarrow_l)$ & $\Delta \eta \, \bar x$ \\
    & $\bar \lambda_k \geq x^\downarrow_l/\eta^\mathrm{d}$ & $ -x^0_l/\eta^\mathrm{d}$ & $\eta^\mathrm{c} x^\downarrow_l + \Delta \eta \, x^0_l - \bar \lambda_k$ & $\Delta \eta \, \bar x$ \\
    \midrule
    $\upsilon_{2l} = 0$ &&& $-\eta^\mathrm{c} x^0_l - \varphi(x^0_l, x^\downarrow_l, \bar \lambda_k)$ \\
    \midrule
    $\upsilon_{1l} = 1$ & $0 \leq \bar \lambda_k \leq x^\downarrow_l/\eta^\mathrm{d}$ & $(x^\downarrow_l - x^0_l)/\eta^\mathrm{d} - \bar \lambda_k$ & $\Delta \eta \, x^0_l - x^\downarrow_l/\eta^\mathrm{d} + \bar \lambda_k$ & $\Delta \eta \, \bar x$ \\
    & $\bar \lambda_k \geq x^\downarrow_l/\eta^\mathrm{d}$ & $ -x^0_l/\eta^\mathrm{d}$ & $\Delta \eta \, x^0_l$ & $\Delta \eta \, \bar x$ \\
    $\upsilon_{1l} = 0$ & $0 \leq \bar \lambda_k \leq \eta^\mathrm{c} x^\downarrow_l$ & $\eta^\mathrm{c} (x^\downarrow_l - x^0_l) - \bar \lambda_k$ & $-\eta^\mathrm{c} x^\downarrow_l + \bar \lambda_k$ & $0$ \\
    $\upsilon_{1l} = 0$ & $\eta^\mathrm{c} x^\downarrow_l \leq \bar \lambda_k \leq x^\downarrow_l/\eta^\mathrm{d}$ & $-\bar \lambda_k x^0_l/x^\downarrow_l$ & $-\eta^\mathrm{c} x^0_l + x^0_l/x^\downarrow_l \bar \lambda_k$ & $\Delta \eta \, \bar x$ \\
    \midrule
    &&& $-\bar \lambda_k x^0_l - x^\downarrow_l\varphi(x^0_l, x^\downarrow_l, \bar \lambda_k)$ \\
    \midrule
    $\upsilon_{1l} = 1$ & $0 \leq \bar \lambda_k \leq x^\downarrow_l/\eta^\mathrm{d}$ & $(x^\downarrow_l - x^0_l)/\eta^\mathrm{d} - \bar \lambda_k$ & $(x^0_l - x^\downarrow_l) ( x^\downarrow_l/\eta^\mathrm{d} - \bar \lambda_k )$ & $\bar x^2/4\eta^\mathrm{d}$ \\
    $\upsilon_{1l} = 1$ & $\bar \lambda_k \geq x^\downarrow_l/\eta^\mathrm{d}$ & $-x^0_l/\eta^\mathrm{d}$ & $x^0_l (x^\downarrow_l/\eta^\mathrm{d} - \bar \lambda_k)$ & $0$ \\
    $\upsilon_{2l} = 1$ & $0 \leq \bar \lambda_k \leq \eta^\mathrm{c} x^\downarrow_l$ & $\eta^\mathrm{c}(x^\downarrow_l - x^0_l) - \bar \lambda_k$ & $(x^\downarrow_l - x^0_l)(\bar \lambda_k - \eta^\mathrm{c} x^\downarrow_l)$ & $0$ \\
    $\upsilon_{2l} = 1$ & $\eta^\mathrm{c} x^\downarrow_l \leq \bar \lambda_k \leq (\bar x - \ubar x)/\eta^\mathrm{d} $ & $-\eta^\mathrm{c} x^0_l$ & $x^0_l(\eta^\mathrm{c} x^\downarrow_l - \bar \lambda_k)$ & $-\ubar x (\bar x - \ubar x)/\eta^\mathrm{d}$ 
    \end{tabular}
    }
    \caption{Minimum values of $M$ for which the optional constraints are valid.}
    \label{tab:M}
    \end{table}

\subsection{Finite-Dimensional Reformulation, Relaxation, and Restriction}

\begin{proof}[Proof of Theorem~\ref{th:P}]
    Theorem~\ref{th:P} follows immediately from Propositions~\ref{prop:charger}, \ref{prop:soc_lb}, and \ref{prop:soc_ub}.
\end{proof}

\begin{proof}[Proof of Proposition~\ref{prop:tractable}]
    We will consider the cases in order.
    \begin{enumerate}
        \item If $\xdn = 0$, constraints~\eqref{reformulation:soc_ub:upsilon1} and~\eqref{reformulation:soc_ub:upsilon2} imply that $\bm \upsilon_1 = 1 - \bm \upsilon_2$, which trivially satisfies the bilinear constraint~\eqref{reformulation:soc_ub:bilinear}. Hence, we recover a mixed-integer linear program.
        \item First, if $x^0 \leq 0$, it is optimal to fix the binary variables~$\upsilon_1 = 0$ and~$\upsilon_2 = 1$, which trivially satisfies the bilinear constraint~\eqref{reformulation:soc_ub:bilinear}. Second, if $K=1$, none of the mixed-integer constraints~\eqref{reformulation:soc_ub:upsilon1}--\eqref{reformulation:soc_ub:bilinear} are generated. Third, if $\eta^\mathrm{c} = \eta^\mathrm{d} = 1$, then $\Delta \eta = 0$. The binary variables thus only appear in the bilinear constraint~\eqref{reformulation:soc_ub:bilinear}, which becomes redundant, since we can choose~$\upsilon_1 = 1$ and $\upsilon_2 = 1$, without impacting the the objective function or any other constraint. In all cases we recover a linear program. \qedhere
    \end{enumerate}
\end{proof}

\begin{proof}[Proof of Proposition~\ref{prop:pres}]
    We first establish an upper bound on $\varphi(x^0_l, x^\downarrow_l, \bar \lambda_k)$, so that we will overestimate the SOC in the restricted problem~\eqref{pb:P:restrict}. Let
    \begin{equation}
       \bar \varphi(x^0_l, x^\downarrow_l, \bar \lambda_k)
       = \begin{cases}
            \max\big\{ -\frac{x^0_l}{\eta^\mathrm{d}}, \, \frac{x^\downarrow_l - x^0_l}{\eta^\mathrm{d}} - \bar \lambda_k \big\} & \text{if } x^0_l \geq x^\downarrow_l \text{ or } x^0_l \geq \big[ \frac{x^\downarrow_l - \eta^\mathrm{d} \bar \lambda_k}{1 - \eta^\mathrm{c} \eta^\mathrm{d}} \big]^+, \\
            \max\big\{ -\eta^\mathrm{c} x^0_l, \, \eta^\mathrm{c}(x^\downarrow_l - x^0_l) - \bar \lambda_k \big\} & \text{otherwise.}
        \end{cases}
    \end{equation}
    It is easy to verify that $\bar \varphi(x^0_l, x^\downarrow_l, \bar \lambda_k) = \varphi(x^0_l, x^\downarrow_l, \bar \lambda_k)$ except if $\varphi(x^0_l, x^\downarrow_l, \bar \lambda_k) = - \bar \lambda_k x^0_l / x^\downarrow_l$, which occurs if and only if
    $0 \leq x^0_l \leq x^\downarrow_l$, $x^\downarrow_l > 0$, and $\eta^\mathrm{c} x^\downarrow_l \leq \bar \lambda_k \leq x^\downarrow_l / \eta^\mathrm{d}$. In this case, $-\eta^\mathrm{c} x^0_l + \bar \lambda_k x^0_l/x^\downarrow_l$ is in $[0, \Delta \eta \, x^0_l]$ and $(x^\downarrow_l - x^0_l)/\eta^\mathrm{d} - \bar \lambda_k + \bar \lambda_k x^0_l/x^\downarrow_l$ is in $[0, \Delta \eta \, (x^\downarrow_l - x^0_l)]$. Thus, $\bar \varphi(x^0_l, x^\downarrow_l, \bar \lambda_k)$ is indeed an upper bound on $\bar \varphi(x^0_l, x^\downarrow_l, \bar \lambda_k)$ for all feasible $x^0_l$, $x^\downarrow_l$, and $\bar \lambda_k$. Next, we introduce the binary variable $\upsilon_{3l}$ to distinguish if the inequality $x^0_l \geq (x^\downarrow_l - \eta^\mathrm{d} \bar \lambda_k) / (1 - \eta^\mathrm{c} \eta^\mathrm{d})$ holds for $l = 1,\ldots,k-1$. If $\upsilon_{3l} = 0$, the inequality should hold, otherwise the reverse inequality should hold. Following similar steps as in Section~\ref{sec:M} in the proof of Proposition~\ref{prop:soc_ub}, yields the constraints~\eqref{reformulation:restriction}.    
\end{proof}

\begin{proof}[Proof of Proposition~\ref{prop:bound}]
The relaxed problem~\eqref{pb:P:relax} is obtained by deleting the bilinear constraints~\eqref{reformulation:soc_ub:bilinear} in problem~\eqref{pb:P}, which is equivalent to ignoring the bilinear-over-linear term in the $\varphi$-function. The corresponding lower bound $\ubar \varphi$ is given by
\begin{equation}
    \ubar \varphi(x^0_l, x^\downarrow_l, \bar \lambda_k)
    =
    \begin{cases}
        \max\big\{ -\eta^\mathrm{c} x^0_l, \, \eta^\mathrm{c} (x^\downarrow_l - x^0_l) - \bar \lambda_k \big\} & \text{if } x^0_l \leq 0, \\
        \max\big\{ -\frac{x^0_l}{\eta^\mathrm{d}}, \, \frac{x^\downarrow_l - x^0_l}{\eta^\mathrm{d}} - \bar \lambda_k \big\}
        & \text{if } x^\downarrow_l \leq x^0_l, \\
        \max\big\{ -\frac{x^0_l}{\eta^\mathrm{d}}, \, \eta^\mathrm{c} (x^\downarrow_l - x^0_l) - \bar \lambda_k \big\} 
        & \text{if } 0 \leq x^0_l \leq x^\downarrow_l.
    \end{cases}
\end{equation}
The maximum difference between the maximum SOC estimated in~\eqref{pb:P:restrict} and~\eqref{pb:P:relax}  is 
\begin{align}
    && \max_{(\bm x^0, \bm x^\downarrow) \in \set{X}, k \in \K} & \left[ \min_{0 \leq \bar \lambda^\mathrm{u} \leq \bar{\bar \lambda}} \gamma \bar \lambda^\mathrm{u} + \Delta t \sum_{l = 1}^{k-1} \bar \varphi(x^0_l, x^\downarrow_l, \bar \lambda^\mathrm{u}) + \Delta t \left[ \bar \varphi(x^0_k, x^\downarrow_k, \bar \lambda^\mathrm{u}) \right]^+ \right] \\
    && - & \left[ \min_{0 \leq \bar \lambda^\mathrm{l} \leq \bar{\bar \lambda}} \gamma \bar \lambda^\mathrm{l} + \Delta t \sum_{l = 1}^{k-1} \ubar \varphi(x^0_l, x^\downarrow_l, \bar \lambda^\mathrm{l}) + \Delta t \left[ \ubar \varphi(x^0_k, x^\downarrow_k, \bar \lambda^\mathrm{l}) \right]^+ \right] \notag \\
    \leq && \max_{(\bm x^0, \bm x^\downarrow) \in \set{X}, k \in \K} & \, \max_{0 \leq \bar \lambda \leq \bar{\bar \lambda}} \Delta t \sum_{l = 1}^{k-1} \bar \varphi(x^0_l, x^\downarrow_l, \bar \lambda) - \ubar \varphi(x^0_l, x^\downarrow_l, \bar \lambda) \\
    = && (T - \Delta t) \max_{x^0_l, x^\downarrow_l} & \, \max_{0 \leq \bar \lambda \leq \bar{\bar \lambda}} \left(  \bar \varphi(x^0_l, x^\downarrow_l, \bar \lambda) - \ubar \varphi(x^0_l, x^\downarrow_l, \bar \lambda) \right) ~~ \text{s.t.} ~~
    \ubar x \leq x^0_l - x^\downarrow_l, \,
    x^0_l \leq \bar x, \,
    x^\downarrow_l \geq 0,
\end{align}
where $\set{X}$ captures the bounds on power output and the nonnegativity of $\bm x^\downarrow$, and $\bar \varphi$ is the upper bound on the $\varphi$-function defined in Proposition~\ref{prop:pres}. The inequality holds because constraining $\bar \lambda^\mathrm{u}$ to be equal to $\bar \lambda^\mathrm{l}$ overestimates the first inner minimum and because  $[ \bar \varphi(x^0_k, x^\downarrow_k, \bar \lambda)]^+ = [ \ubar \varphi(x^0_k, x^\downarrow_k, \ubar \lambda)]^+ $ for all $(\bm x^0, \bm x^\downarrow) \in \set{X}$ and all $\bar \lambda \in [0, \bar{\bar \lambda}]$. The equality holds because $\ubar \varphi$, $\bar \varphi$, and $\set{X}$ do not directly depend on~$k$. The maximum difference between the upper and the lower bound on~$\varphi$ is given by
\begin{align}
    & \max_{x^0_l, x^\downarrow_l, \bar \lambda} \, \bar \varphi(x^0_l, x^\downarrow_l, \bar \lambda) - \ubar \varphi(x^0_l, x^\downarrow_l, \bar \lambda) 
    ~~ \text{s.t.} ~~
    \ubar x \leq x^0_l - x^\downarrow_l, \,
    x^0_l \leq \bar x, \,
    x^\downarrow_l \geq 0, \,
    0 \leq \bar \lambda \leq \frac{\bar x - \ubar x}{\eta^\mathrm{d}} \\
    = & \max_{x^0_l, x^\downarrow_l, \bar \lambda} \, \min\left\{ -\eta^\mathrm{c} x^0_l, \, \frac{x^\downarrow_l - x^0_l}{\eta^\mathrm{d}} - \bar \lambda \right\} - \max\left\{ -\frac{x^0_l}{\eta^\mathrm{d}}, \, \eta^\mathrm{c} (x^\downarrow_l - x^0_l) - \bar \lambda \right\} \\
    & ~~~ \text{s.t.} ~~~
    \ubar x \leq x^0_l - x^\downarrow_l, \,
    x^0_l \leq \bar x, \,
    0 \leq x^0_l \leq x^\downarrow_l, \,
    \eta^\mathrm{c} x^\downarrow_l \leq \bar \lambda \leq \frac{x^\downarrow_l}{\eta^\mathrm{d}} \\
    = & \max_{x^0_l, x^\downarrow_l, \bar \lambda} \, \min\left\{ \Delta \eta \, x^0_l, \, \bar \lambda - \eta^\mathrm{c} x^\downarrow_l, \, \frac{x^\downarrow_l}{\eta^\mathrm{d}} - \bar \lambda, \, \Delta \eta \, (x^\downarrow_l - x^0_l) \right\} \\
    & ~~~ \text{s.t.} ~~~
    \ubar x \leq x^0_l - x^\downarrow_l, \,
    x^0_l \leq \bar x, \,
    0 \leq x^0_l \leq x^\downarrow_l, \,
    \eta^\mathrm{c} x^\downarrow_l \leq \bar \lambda \leq \frac{x^\downarrow_l}{\eta^\mathrm{d}} \\
    = & \max_{x^0_l} \, \Delta \eta \, x^0_l 
    ~~ \text{s.t.} ~~
    x^0_l \leq \bar x, \, x^0_l \leq -\ubar x, \, x^0_l \leq \frac{\bar x - \ubar x}{1 + \eta^\mathrm{c} \eta^\mathrm{d}} \\
    = &
    \Delta \eta \cdot \min\left\{ -\ubar x, \, \bar x, \, \frac{\bar x - \ubar x}{1 + \eta^\mathrm{c} \eta^\mathrm{d}} \right\}.
\end{align}
The first equality holds as $\bar \varphi(x^0_l, x^\downarrow_l, \bar \lambda) \neq \ubar \varphi(x^0_l, x^\downarrow_l, \bar \lambda)$ only if $ 0 \leq x^0_l \leq x^\downarrow_l$ and $\eta^\mathrm{c} x^\downarrow_l \leq \bar \lambda \leq \frac{x^\downarrow_l}{\eta^\mathrm{d}}$. The second equality follows by inverting the sign of the inner maximization problem. The third equality holds because the minimum is maximized if $\bar \lambda - \eta^\mathrm{c} x^\downarrow_l = \frac{x^\downarrow_l}{\eta^\mathrm{d}} - \bar \lambda$ and $ \Delta \eta \, x^0_l = \Delta \eta \, (x^\downarrow_l - x^0_l)$.
\end{proof}

\subsection{Intraday Trading}
\begin{proof}[Proof of Proposition~\ref{prop:intraday}] The proof relies on the difference between the observed SOC at the end of each trading interval and the SOC induced by the arbitrage decisions only. For any $k \in \K$, let $\Delta y_k$ denote the difference at the end of the $k$-th trading interval. Without intraday adjustments, the difference evolves as
\begin{align}
    \Delta y_{k} - \Delta y_{k-1} = &
    \int_{(k-1)\Delta t}^{k\Delta t}
    \max\left\{
    \frac{x^0_k}{\eta^\text{d}}, \, \eta^\text{c} x^0_k
    \right\}
    - \max\left\{ \frac{\xi(\tau)x^\text{r}_k + x^0_k}{\eta^\text{d}}, \, \eta^\text{c} \left(\xi(\tau)x^\text{r}_k + x^0_k\right)
    \right\}
    \, \mathrm{d}\tau \\
    = & - x^\text{r}_{k} \int_{(k-1)\Delta t}^{k \Delta t} \xi(\tau) \, \mathrm{d}\tau,
\end{align}
where the second equality holds because we assumed that $\eta^\text{c} = \eta^\text{d} = 1$. We correct for the mismatch in the SOC by setting the intraday adjustment for period~$k+1$ to 
\begin{equation}
    x^\text{a}_{k+1} 
    = \frac{\Delta y_k - \Delta y_{\ubar i(k)}}{\Gamma'-\Delta t}
    = \sum_{i = \ubar i(k+1)}^k \frac{\Delta y_i - \Delta y_{i-1}}{\Gamma'-\Delta t} 
    = - \frac{1}{\Gamma'-\Delta t} \sum_{i = \ubar i(k+1)}^k x^\text{r}_{i} \int_{(i-1)\Delta t}^{i \Delta t} \xi(\tau) \, \mathrm{d}\tau,
\end{equation}
which ensures that the SOC difference with intraday adjustments evolves as
\begin{equation}
    \Delta y_k = - \sum_{i = \ubar i(k)}^k \frac{\Gamma' - (k-i+1)\Delta t}{\Gamma' - \Delta t} x^\text{r}_i \int_{(i-1)\Delta t}^{i \Delta t} \xi(\tau) \, \mathrm{d}\tau.
\end{equation}
Any feasible solution to problem~\eqref{pb:P} guarantees that the bounds on the SOC are respected for any SOC difference whose absolute value does not exceed $\sum_{i = \ubar i(k)}^k x^\text{r}_i \int_{(i-1)\Delta t}^{i \Delta t} \vert \xi(\tau) \vert \, \mathrm{d}\tau$. The intraday adjustments ensure that the bounds are respected because
\begin{equation}
    \left\vert \sum_{i = \ubar i(k)}^{k} \frac{\Gamma' - (k-i+1)\Delta t}{\Gamma'-\Delta t } x^\text{r}_i \int_{(i-1)\Delta t}^{i \Delta t} \xi(\tau) \, \mathrm{d}\tau \right\vert
    \leq \sum_{i = \ubar i(k)}^k x^\text{r}_i \int_{(i-1)\Delta t}^{i \Delta t} \vert \xi(\tau) \vert \, \mathrm{d}\tau.
\end{equation}
The bounds on the charging and discharging power are guaranteed to be respected if and only if
\begin{align}
    \ubar x \leq x^0_k + x^\text{a}_k - x^\text{r}_k,~~
    x^0_k + x^\text{a}_k + x^\text{r}_k \leq \bar x~~
    \forall k \in \K, ~ \forall \xi \in \Xi'.
\end{align}
For any fixed $k \in \K \setminus {K}$, we now compute the maximum and minimum values of $x^\text{a}_{k+1}$. We have 
\begin{align}
    \max_{\xi \in \Xi'} \, x^\text{a}_{k+1}
    = & \max_{\xi \in \Xi'} \, \frac{1}{\Gamma'-\Delta t} \sum_{i = \ubar i(k+1)}^k x^\text{r}_{i} \int_{(i-1)\Delta t}^{i \Delta t} \xi(\tau) \, \mathrm{d}\tau \\
    = & \max_{0 \leq \bm \xi \leq 1}\, \frac{\Delta t}{\Gamma'-\Delta t} \sum_{i = \ubar i(k+1)}^k \xi_i x^\text{r}_{i}
    ~ \text{s.t.}
    \sum_{i = \ubar i(k)}^k \xi_i \leq \frac{\gamma'}{\Delta t} \\
    = & \min_{\lambda_k \geq 0} \,  \frac{\gamma'}{\Delta t} \lambda_k + \sum_{i = \ubar i(k+1)}^k \left[ \frac{\Delta t}{\Gamma'-\Delta t} x^\text{r}_i - \lambda_k \right]^+.
\end{align}
The third equality follows from standard linear programming duality. The minimum value of $x^\text{a}_{k+1}$ can be derived in a similar manner thanks to the symmetry of $\Xi'$.
\end{proof}